\documentclass[graybox,envcountsect,envcountsame]{svmult}

\usepackage{type1cm}        
\usepackage{makeidx}         
\usepackage{graphicx}        
\usepackage{multicol}        
\usepackage[bottom]{footmisc}

\usepackage{newtxtext}       %
\usepackage{newtxmath}       

\usepackage{tikz}
\usepackage{mathrsfs}

\usepackage{url}
\usepackage[scriptsize]{subfigure}
\usepackage[all]{xy}
\usepackage{caption}
\newcommand{\norm}[1]{\lVert#1\rVert}

\usepackage{amssymb,amsfonts,latexsym}

\usepackage{float}
\usepackage{pgfplots}
\pgfplotsset{width=7cm,compat=1.8}
\usepackage{subfigure}

\usetikzlibrary{positioning}
\usetikzlibrary{decorations.markings}
\usetikzlibrary{
  knots,
  hobby,
  decorations.pathreplacing,
  shapes.geometric,
  calc
}

\tikzset{
  knot diagram/every strand/.append style={
    ultra thick,
    red
  },
  show curve controls/.style={
    postaction=decorate,
    decoration={show path construction,
      curveto code={
        \draw [blue, dashed]
        (\tikzinputsegmentfirst) -- (\tikzinputsegmentsupporta)
        node [at end, draw, solid, red, inner sep=2pt]{};
        \draw [blue, dashed]
        (\tikzinputsegmentsupportb) -- (\tikzinputsegmentlast)
        node [at start, draw, solid, red, inner sep=2pt]{}
        node [at end, fill, blue, ellipse, inner sep=2pt]{}
        ;
      }
    }
  },
  show curve endpoints/.style={
    postaction=decorate,
    decoration={show path construction,
      curveto code={
        \node [fill, blue, ellipse, inner sep=2pt] at (\tikzinputsegmentlast) {}
        ;
      }
    }
  }
}

\usepackage[pagebackref,hypertexnames=false]{hyperref}
\hypersetup{
        colorlinks=true,
        linkcolor=blue,
        citecolor=magenta,
        urlcolor=cyan,
}

\newcommand{\Q}{{\mathbb Q}}
\newcommand{\R}{{\mathbb R}}
\newcommand{\C}{{\mathbb C}}

\spnewtheorem*{CSL}{Curve Selection Lemma}{\bf}{\it}
\spnewtheorem*{Lem}{Lemma}{\bf}{\it}

\makeindex             

\begin{document}

\title*{An introduction to Lipschitz geometry of
  complex singularities}
\author{Anne Pichon}
\institute{Anne Pichon \at Aix Marseille Universit\'e, CNRS, Centrale Marseille, I2M, UMR 7373, 13453 Marseille, FRANCE \email{anne.pichon@univ-amu.fr}
}
%
%
\maketitle

\abstract{
  The aim of this paper to introduce the reader to a recent point of view on the  Lipschitz classifications of complex singularities. 
It presents the complete classification of Lipschitz geometry of complex plane curves singularities and in particular, it introduces the so-called bubble trick, which is a key tool to study Lipschitz geometry of germs. It  describes also  the thick-thin decomposition of a normal complex surface singularity and built two geometric decompositions of a normal surface germ into standard pieces  which are invariant by respectively  inner and outer   bilipschitz homeomorphisms.  This leads in particular to the complete classification of Lipschitz geometry for the inner metric.}

\section*{Introduction} \label{sec:introduction}

The aim of this paper is to introduce the reader to a recent point of view on the  Lipschitz classifications of complex singularities. It is an expansion of my notes prepared for the course given at the International school on singularities and Lipschitz geometry, which took place in Cuernavaca (Mexico) from June 11th to 22nd 2018.

The notes  are structured as follows.  Section \ref{sec:preliminaries} explains what is  Lipschitz geometry  for the inner and outer metrics of singularities and why it is interesting for the classification of space singularities. Section \ref{part 2} gives the complete classification of Lipschitz geometry of complex curves and covers the results of \cite{NeumannPichon2014}. In particular, it introduces what we call the bubble trick and the bubble trick with jumps, which are key tools to study Lipschitz geometry of germs. This techniques, which consists in exploring a germ of analytic space $(X,0)$ by using horns centered on a germ of real arc, was pioneered in \cite{HP2003} and used in several recent works (see e.g. \cite{NeumannPichon2012, NeumannPichon2014, FdBHPPS19}). Section \ref{part 3}  describes the thick-thin decomposition of a normal complex surface germ following \cite{BirbrairNeumannPichon2014}. Section   \ref{part 4} describes two geometric decompositions of a normal surface germ into standard pieces  which are invariant by respectively  inner and outer   bilipschitz homeomorphism, following the results of \cite{BirbrairNeumannPichon2014} and \cite{NeumannPichon2012}. This leads in particular to the complete classification of Lipschitz geometry for the inner metric. 

The paper contains a lot of detailed examples which were presented and discussed  during  the afternoon exercise sessions of the school and also an appendix (Part \ref{appendix}) which gives the computation of the resolution graph of a surface singularity with equation $x^2+f(y,z)=0$ following Hirzebruch-Jung and Laufer's method. This enables the readers to produce a lot of examples by themself.


In these notes, I do not give the detailed proofs of the invariance of the inner and outer Lipschitz decompositions (Theorem \ref{th:classification} and Theorem \ref{thm:outer invariant}). We refer to \cite{BirbrairNeumannPichon2014} and \cite{NeumannPichon2012} respectively. However, it has to be noted that even if  the two statements   look similar, the techniques used in the proofs are radically different. The invariance of the inner decomposition uses the Lipschitz invariance of fast loops (introduced in  Section \ref{part 3}) of minimal length in their homology class (\cite[Section 14]{BirbrairNeumannPichon2014}) while that of the outer  invariance uses  sophisticated bubble trick arguments (\cite{NeumannPichon2012}). 

Notice that   the pioneering paper \cite{BirbrairNeumannPichon2014} is written for a normal complex surface, as well as the initial version of    \cite{NeumannPichon2012}. However,  the extensions of the inner and outer  geometric decompositions  to the general case of a reduced singularity (not necessarily   isolated) are fairly easy. A version of \cite{NeumannPichon2012} in this general setting will appear soon. 

Finally, notice that the inner and outer geometric decompositions are the analogs of the pizza decompositions of a real surface germ.  In the real surface case, these decompositions give complete classifications for the inner and outer metrics. As already mentioned, the inner decomposition in the complex case also gives a complete classification after adding a few more invariants (Theorem \ref{th:classification}). In contrast, a complete classification for the outer metric of complex surface singularities would need more work and is still an open question.

\section{Preliminaries} \label{sec:preliminaries}
 
\subsection{What is Lipschitz geometry of singular spaces?} \label{subsec:what is}

In the sequel, $\mathbb K $ will denote either $ \mathbb R$ or $\mathbb C$.

Let $(X,0)$ be a germ of analytic space  in $\mathbb K^n$  which contains the origin. So $X$ is defined by 
$$X=\{(x_1,\ldots,x_n) \in \mathbb K^n \mid \  f_j(x_1,\ldots,x_n)=0, j=1,\ldots,r\},$$
where the $f_j$'s are  convergent power series,  $f_j \in \mathbb K\{x_1,\ldots,x_n\}$ and $f_j(0)=0$.

 \vskip0,3cm\noindent
 {\bf Question 1.}  How does $X$ look in a small neighbourhood of the origin?
  \vskip0,3cm

 There are multiple answers to this vague question depending on the category we work in, i.e., on the chosen equivalence relation between germs. 

First, we can consider the topological equivalence relation: 
\begin{definition} 
 Two analytic germs $(X,0)$ and $(X',0)$ are {\bf topologically equivalent} if there exists a germ of  homeomorphism $\psi \colon (X,0)  \to (X',0)$. The {\bf topological type} of $(X,0)$ is the equivalence class of $(X,0)$ for this equivalence relation.
 
  Two analytic germs $(X,0) \subset (\mathbb K^n,0)$ and $(X',0) \subset (\mathbb K^n,0)$  are {\bf topologically equisingular} \index{equisingularity!topological}    if there exists a germ of  homeomorphism $\psi \colon (\mathbb K^n,0)  \to (\mathbb K^n,0)$ such that $\psi(X)=X'$. We call {\bf embedded topological type} \index{embedded topological type} of $(X,0)$ the equivalence class of $(X,0)$ for this equivalence relation.
 \end{definition}
 The  embedded topological type  of $(X,0) \subset (\mathbb R^n,0)$ is completely determined by the embedded topology of its link as stated in the following  famous Conical Structure Theorem: 

\begin{theorem} [Conical Structure Theorem]  \index{Conical Structure Theorem} Let $B^n_{\epsilon}$ be the ball with radius $\epsilon>0$ centered at the origin of $\R^n$ and let $S^{n-1}_{\epsilon}$ be its boundary.  

Let $(X,0) \subset (\mathbb R^n,0)$ be an analytic germ.  For $\epsilon>0$, set $X^{(\epsilon)}=S^{n-1}_{\epsilon} \cap X$. 
 There exists $\epsilon_0>0$ such that for every $\epsilon >0$ with $0 < \epsilon \leq \epsilon_0$, the pair $(B^n_{\epsilon}, X \cap B^n_{\epsilon})$ is homeomorphic to the pair 
 $( B^n_{\epsilon_0}, Cone (X^{(\epsilon_0)}))$, where $Cone (X^{(\epsilon_0)})$ denotes the cone over $X^{(\epsilon_0)}$, i.e., the union of the segments $[0, x]$ joining the origin to a  point $x\in X^{(\epsilon_0)}$. 
\end{theorem}
In other words,  the homeomorphism  class of the pair $(S^{n-1}_{\epsilon}, X^{(\epsilon)})$ does not depend on $\epsilon$ when $\epsilon>0$ is sufficiently small and  it determines completely  the embedded topological type of $(X,0)$. 
\begin{definition} When $0 < \epsilon \leq \epsilon_0$, the intersection $X^{(\epsilon)}$ is called  the {\bf link} of $(X,0)$. 
 \end{definition}
 
 \begin{example}  
  
 \begin{enumerate}
\item \label{ex1} Assume that $X$ is the real cusp in $\R^2$ with equation $x^3-y^2=0$. Then its link at $0$ consists of two points in the circle $S^1_{\epsilon}$. 
\item \label{ex2} If  $X$ is the complex cusp in $\C^2$ with equation $x^3-y^2=0$, its link at $0$  is the trefoil knot in  the $3$-sphere $S^3_{\epsilon}$. 
\item \label{ex3}  If  $X$ is the complex surface $E_8$    in $\C^3$ with equation $x^2+y^3+z^5=0$, its (non embedded) link  at $0$  is a Seifert manifold whose homeomorphism class is completely described through plumbing theory  by its minimal resolution graph. 
The resolution graph is explicitely computed in the appendix \ref{appendix} of the present notes. 
\end{enumerate}
  \end{example}
  
 The Conical Structure Theorem gives a complete answer to Question 1 in the topological category, but it completely ignores the geometric properties of the set $(X,0)$. In particular, a very interesting question is:

 \vskip0,3cm\noindent
 {\bf Question 2.} 
  How does the link $ X^{(\epsilon)}$ evolve metrically as $\epsilon$ tends to $0$?
   \vskip0,3cm

 In other words,  is $X\cap B_{\epsilon}$  bilipschitz  homeomorphic to the straight cone $Cone(X^{(\epsilon)})$? Or are there some  parts of $X^{(\epsilon)}$ which shrink faster than linearly when $\epsilon$ tends to 0? 
 
  This question  can be studied from different points of view depending on the choice of the metric.  If  $(X,0) \subset (\mathbb R^n,0)$ is the germ of a real analytic  space, there are two natural metrics  on $(X,0)$ which are defined from the Euclidean metric of the ambient space $ \mathbb R^n$:
 
 \begin{definition} The {\bf outer metric}  \index{outer metric} $d_o$ on $X$ is  the metric induced by the ambient Euclidean metric, i.e., for all $x,y \in X$, $d_o(x,y)=\lVert   {x-y}  \rVert_{\R^n}$.
 
  The {\bf inner metric}  \index{inner metric}  $d_i$ on $X$ is the length metric defined for all $x,y \in X$ by: $d_i(x,y) = \inf length (\gamma)$, where $\gamma \colon  [0,1] \to X$ varies among   rectifyable arcs on $X$ such that $\gamma(0)=x$ and $\gamma(1)=y$. 
  \end{definition} 
    
    \begin{definition} Let $(M,d)$ and $(M',d')$ be two metric spaces. A map $f \colon M \to M'$ is a  {\bf bilipschitz homeomorphism} if $f$ is a bijection and there exists a real constant $K \geq 1$ such that for all $x,y \in M$, $$\frac{1}{K} d(x,y) \leq d'(f(x),f(y))\leq K  d(x,y).$$
  \end{definition}

  \begin{definition} Two real analytic germs $(X,0) \subset (\R^n,0)$ and $(X',0) \subset (\R^m,0)$ are  {\bf inner Lipschitz equivalent} (resp. {\bf outer Lipschitz equivalent}) if there exists a germ of bilipschitz homeomorphism $\psi \colon (X,0) \to (X',0)$ with respect to the inner (resp. outer) metrics. 
  
  The equivalence classes of the germ $(X,0) \subset (\R^n,0)$  for these equivalence relations are called  respectively  the {\bf inner Lipschitz geometry} and the  {\bf outer Lipschitz geometry} of $(X,0)$.
  \end{definition}
  
    Throughout these notes, we will use the ``big-Theta" asymptotic notations of Bachmann-Landau in the following form: 
   
      \begin{definition} \label{dfn:Bachmann}
     Given two function germs $f,g\colon ([0,\infty),0)\to ([0,\infty),0)$, we say  that $f$ is   {\bf big-Theta}   of $g$ and we write   $f(t) = \Theta (g(t))$ if there exist real numbers $\eta>0$ and $K >0$ such that for all $t$ such that  for all $t \in [0, \eta)$, $\frac{1}{K }g(t) \leq f(t) \leq K g(t)$. 
   \end{definition}
   
   \begin{example} Consider the real cusp $C$ with equation $y^2-x^3=0$ in $\R^2$ (see Figure \ref{fig:1}).   For a real number $t>0$, consider the two points  $p_1(t) = (t^2,t^3)$ and   $p_2(t) = (t^2,-t^3)$ on $C$. Then $d_o(p_1(t), p_2(t)) = \Theta(t^{3/2})$  while  the inner distance  is obtained by taking infimum of  lengths of   paths on $C$ between the two points $p_1(t)$ and $p_2(t)$. The shortest length is obtained by taking a path going through the origin, and we have  $d_i(p_1(t), p_2(t)) = \Theta(t)$. Therefore, taking the limit of the quotient as $t$ tends to $0$, we obtain: 
   $$\frac{d_o(p_1(t), p_2(t))}{d_i(p_1(t), p_2(t))} = \Theta(t^{1/2}) \to 0.$$

    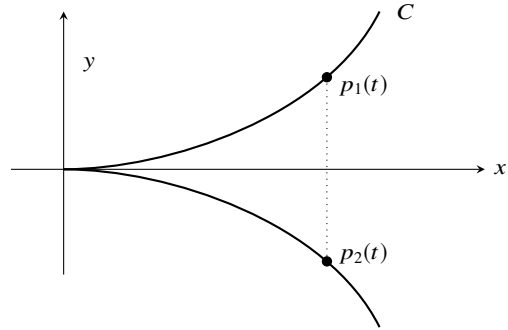
\begin{figure}[ht]
    \centering
\begin{tikzpicture} 
\begin{scope}[scale=0.7]
\draw[thick](0,0) .. controls (2,0) and (5,1)..  (6,3) ;
\draw[thick](0,0) .. controls (2,0) and (5,-1)..  (6,-3) ;

\draw[thin,>=stealth,->](-1,0)--(8,0);
\draw[thin,>=stealth,->](0,-2)--(0,3);

\draw[fill](5,1.75)circle(2.5pt);
\draw[fill](5,-1.75)circle(2.5pt);

\draw[dotted](5,1.75)--(5,-1.75);
    
    \node(a)at(8.3,0){$x$};
    \node(b)at(6.5,3){$C$};
     \node(b)at(5.7,1.6){$p_1(t)$};
          \node(b)at(5.7,-1.6){$p_2(t)$};

              \node(c)at(0.5,2){$y$};
                                 
     \end{scope}             
\end{tikzpicture}
  \caption{The real cusp $y^2-x^3=0$}
    \label{fig:1}
\end{figure}

 \end{example}
 
  Using this, you are ready to make the following:

   \begin{exercise}   
   \begin{enumerate}
\item Prove  that  there is no bilipschitz homeomorphism between the outer and inner metrics on  the real cusp $C$ with equation $y^2-x^3=0$ in $\R^2$.
\item Prove that $(C,0)$ equipped with the inner metric is metrically conical, i.e. bilipschitz equivalent to the cone over its link. 
\end{enumerate}
 
  \end{exercise} 

  \begin{example} \label{example:real surface} Consider the real surface $S$ in $\R^3$ with equation $x^2+y^2-z^3=0$ in $\R^2$.   For a real number $t>0$, consider the two points  $p_1(t) = (t^3,0,t^2)$ and   $p_2(t) = (t^3,0,-t^2)$ on $S$. Then $d_o(p_1(t), p_2(t)) = \Theta(t^{3/2})$. We also have  $d_i(p_1(t), p_2(t)) = \Theta(t^{3/2})$ since $d_i(p_1(t), p_2(t))$ is the length of a half-circle  
  joining  $ p_1(t)$ and  $p_2(t)$ on the circle $\{x=t^3\} \cap S$. 
  
   \end{example}

  \begin{exercise}  Consider the real surface $S$  of Example \ref{example:real surface}. 
  \begin{enumerate}
\item  Prove that  the identity map  is  a bilipschitz homeomorphism between the outer and inner metrics on $(S,0)$.   
\item  Prove that $(S,0)$ equipped with the inner metric is not metrically conical. 
\end{enumerate}

   \end{exercise}
 
\subsection{Independence of the embedding and motivations} \label{subsec:independence embedding}

 If $(X,0)$ is a germ of a real analytic space, the two metrics $d_o$ and $d_i$ defined above obviously depend on the choice of an embedding $(X,0) \subset (\R^n,0)$ since they are defined by using the Euclidean metric of the ambient  $\R^n$. The aim of this section is to give a proof of one of the main results which motivates the study of Lipschitz geometry of singularities:  
 

\begin{proposition} \label{prop:independance}
  The Lipschitz geometries of $(X,0)$ for the outer and inner metrics
  are independent of the embedding $(X,0) \subset (\R^n,0)$.
\end{proposition}
 
In other words, bilipschitz classes of $(X,0)$  just depend on the analytic type of $(X,0)$.    Before proving this result, let us give some consequences which motivate the study of Lipschitz geometry of germs of singular spaces. 
 
The outer Lipschitz  geometry  determines the inner Lipschitz  geometry  since the inner metric is determined by the outer one through integration along paths. Moreover, the inner Lipschitz geometry obviously determines the topological type  of $(X,0)$. Therefore, an important consequence of Proposition \ref{prop:independance}  is  that the Lipschitz geometries give two intermediate classifications between the analytical type  and the topological type. 

A very  small  amount of analytic invariants are determined by the topological type of an analytic germ (even if one considers the embedded topological type). In particular, a natural question is to ask whether the Lipschitz classification is sufficiently rigid to catch analytic invariants:

  \vskip0,3cm\noindent
 {\bf Question 3.} 
   Which analytical invariants are in fact Lipschitz invariants? 
      \vskip0,3cm

Recent results show that in the case of a complex surface singularity, a large amount of analytic invariants  are determined by the outer Lipschitz geometry. For example, the multiplicity of a complex surface singularity  is an outer Lipschitz invariant (\cite{NeumannPichon2012} for a normal surface, \cite{Sampaio2017} for a hypersurface in $\C^3$ and \cite{FFS} for the general case). However it is now known that the multiplicity is not a Lipschitz invariant in higher dimensions (\cite{BFSV}). In  \cite{NeumannPichon2012} it is shown that many other data are in fact Lipschitz invariants  in the case of surface singularities, such as the geometry of hyperplane sections and the geometry of polar curves and discriminant curves of generic projections (Theorem \ref{th:invariants from geometry}); higher dimensions remain  almost unexplored.  This shows that the  outer Lipschitz class contains potentially a lot of information on the singularity and  that outer Lipschitz geometry of singularities is a very promising area to explore.

Here is another motivation. Analytic types of singular space germs contain continuous moduli, and this is why it is difficult to describe a complete analytic classification.  For example, consider the family of curves germs $(X_t,0)_{t \in \C}$ where $X_t$ is the union of four transversal lines with equation  $xy(x-y)(x-ty)=0$. For every pair $(t,t')$ with $t \neq t'$, $(X_t,0)$ is not analytically equivalent to $(X_{t'},0)$. On the contrary, it is  known since the works of T. Mostowki in the complex case (\cite{mostowski}), and Parusi\'nski in the real case (\cite{Parusinski1988} and \cite{Parusinski1994}),  that  the  outer Lipschitz classification of germs of singular spaces is {\bf tame},   which means that it admits a discrete complete invariant. Then a complete classification of Lipschitz geometry of singular spaces seems to be a more reachable goal.

\begin{proof}[of Proposition \ref{prop:independance}] Let $(f_1,\ldots,f_n)$ and $(g_1,\ldots,g_m)$ be two
  systems of generators of the maximal ideal $\cal M$ of $(X,0)$. We will first prove that the outer
  metrics $d_I$ and $d_J$ for the embeddings 
  \[I= (f_1,\ldots,f_n) \colon (X,0) \to (\R^n,0)\text{\quad and\quad  }
    J = (g_1,\ldots,g_m) \colon (X,0) \to (\R^m,0)\]
  are bilipschitz
  equivalent.  It suffices to prove that the outer metric for the
  embedding $(f_1,\ldots,f_n,g_1,\ldots,g_m)$ is bilipschitz
  equivalent to the metric $d_I$. By induction, we just have to
  prove that for any $g \in \cal M$, the metric $d_{I'}$ associated
  with the embedding
  $I' = (f_1,\ldots,f_n,g)\colon (X,0) \to  (\R^{n+1},0)$ is
  bilipschitz equivalent to $d_I$.

  Since $g$ is in the ideal $\cal M$, it may be expressed as  $G(f_1,\dots,f_n)$ where $G \colon (\R^n,0) \to (\R,0)$ is real analytic. Let $\Gamma$ be the graph of the function $G(x_1, \dots, x_n)$ in $(\R^n,0)\times \R$. It is defined over a neighbourhood of $0$ in $\R^n$. The projection $\pi\colon \Gamma\to  (\R^n,0)$ is bilipschitz over any compact neighbourhood of $0$ in $\R^n$ on which it is defined. We have $I'(X,0)\subset \Gamma\subset (\R^n,0) \times \R$, so $\pi|_{I'(X,0)}\colon I'(X,0)\to I(X,0)$ is bilipschitz for the outer metrics $d_{I'}$ and $d_I$.
\end{proof}

 \section{The Lipschitz geometry of a complex curve singularity} \label{part 2}
 
 \index{complex curve}
 
\subsection{Complex curves have trivial inner Lipschitz geometry} \label{sec:inner curves}

Let  $X \subset \C^2$ be the complex cusp   with equation $y^2-x^3=0$. Let $t \in \R$ and consider the two points $p_1(t) = (t^2, t^3)$ and $p_2(t) = (t^2, -t^3)$ on $X$. Since these two points are on two distinct  strands of the braid $X \cap (S_{|t|}^1 \times \C)$, it  is easy to see that the shortest path in $X$ from $p_1(t)$ to $p_2(t)$ passes through the origin and that $d_i(p_1(t), p_2(t)) = \Theta(t)$.  This suggests that $(X,0)$ is locally inner bilipschitz homeomorphic to the cone over its link. This  means that the  inner Lipschitz geometry tells one no more than the topological type, i.e., the number of connected components of the link (which are circles), and is therefore uninteresting. The aim of this section is to prove this for any complex curve. 

\begin{definition} An analytic germ $(X,0)$ is {\bf metrically conical}  \index{metrically conical}  if  it is   inner Lipschitz homeomorphic to the straight  cone over its link. 
\end{definition}

In this paper, a {\bf complex curve germ} or {\bf complex curve singularity} will mean a germ of reduced  complex analytic space of dimension $1$.

\begin{proposition} \label{prop:inner} Any  complex space curve germ $(C,0)
  \subset (\C^N,0)$ is metrically conical. 
\end{proposition}

\begin{proof} 
  Take a linear projection $p \colon \C^N \to \C$ which is generic for
  the curve $(C,0)$ ({\it i.e.}, its kernel contains no tangent line
  of $C$ at $0$) and let $\pi:=p|_C$, which is a branched cover of
  germs. Let $D_\epsilon=\{z\in \C: |z|\le \epsilon\}$ with $\epsilon$
  small, and let $E_\epsilon$ be the part of $C$ which branched covers
  $D_\epsilon$. Since $\pi$ is holomorphic away from $0$ we have a
  local Lipschitz constant $K(x)$ at each point $x\in C \setminus
  \{0\}$ given by the absolute value of the derivative map of $\pi$ at
  $x$. On each branch $\gamma$ of $C$ this $K(x)$ extends continuously over
  $0$  by taking for $K(0)$ the absolute value  of the restriction $p \mid_{T_0\gamma} \colon T_0\gamma \to \C$ where $T_0 \gamma$ denotes the tangent cone to $\gamma$ at $0$. So the infimum and supremum $K^-$ and $K^+$ of $K(x)$ on
  $E_\epsilon\setminus\{0\}$ are defined and positive.  For any arc
  $\gamma$ in $E_\epsilon$ which is smooth except where it passes
  through $0$ we have $K^{-}\ell(\gamma)\le \ell'(\gamma)\le
  K^+\ell(\gamma)$, where $\ell$ respectively $\ell'$ represent arc length
  using inner metric on $\E_\epsilon$ respectively the metric lifted from
  $B_\epsilon$. Since $E_\epsilon$ with the latter metric is strictly
  conical, we are done.
\end{proof}

\subsection{The outer Lipschitz geometry of a complex curve}

Let  ${\mathbf G}(n-2,\C^n)$ be the Grassmanian of  $(n-2)$-planes in $\C^n$. 

Let ${\cal D} \in {\mathbf G}(n-2,\C^n)$ and let $\ell_{\cal D}
\colon \C^n \to \C^2$ be the linear projection with
kernel $\cal D$. Suppose $(C,0)\subset (\C^n,0)$ is a complex curve germ.  There exists an open
dense subset $\Omega_C$ of  ${\mathbf G}(n-2,\C^n)$  such that for ${\cal D} \in \Omega_C$, $\cal D$ contains no limit of secant lines to the curve $C$ (\cite[pp.\ 354]{Teissier1982}). 
\begin{definition} \label{def:generic projection curve}  The
  projection $\ell_{\cal D}$ is  said  to be {\bf generic for $C$} \index{generic!projection} if  ${\cal D} \in \Omega_C$. 
\end{definition}
 In the sequel, we will use extensively the following result 
\begin{theorem}[{\cite[pp.\ 352-354]{Teissier1982}}]  \label{generic projection bilipschitz}If $\ell_{\cal D}$ is a generic projection for $C$, then the restriction $\ell_{\cal D}|_{C} \colon C \to \ell_{\cal D}(C)$ is a bilipschitz homeomorphism for the outer metric.
\end{theorem} 

As a consequence of Theorem \ref{generic projection bilipschitz}, in order to understand Lipschitz geometry of curve germs, it suffices to understand Lipschitz geometry of plane curve germs. 

 Let us start with an example.
\begin{example} \label{ex:carrousel tree} Consider the plane curve germ  $(C,0)$ with two branches having Puiseux expansions $$y=x^{3/2}+x^{13/6},\quad y=x^{5/2}\,.$$
Its topological type is completely described by the sets of characteristic  exponents \index{Puiseux!characteristic exponent} of the branches:  $\{ 3/2, 13/6\}$ and $\{5/2\}$ and by the contact exponents between the two branches: $3/2$. Those  data are  summarized in the Eggers-Wall tree \index{Eggers-Wall tree}  of the curve  germ (see \cite{Wall2004, GPPP19}), or equivalently, in what we will call the {\bf carrousel tree}  \index{carrousel tree} (see the proof of Lemma \ref{le:curve geometry}  and Figure \ref{fig:2}), which is exactly the Kuo-Lu tree defined in \cite{KuoLu1977} but with the horizontal bars contracted to points. 

\begin{figure}[ht]
  \centering
 
\begin{tikzpicture}

  \draw[thin ](0,-2)--(0.2,-4);
   \draw[thin ](0,-2)--(-0.2,-4);
  \draw[thin ](0,0)--(0,-4);
   \draw[thin ](0,-1)--(-1,-3);
       \draw[thin ](-1.1,-4)--(-1,-3);
         \draw[thin ](-0.9,-4)--(-1,-3);
         
         \draw[thin ](0,-1)--(1,-2);
            \draw[thin ](1,-4)--(1,-2);
       \draw[thin ](1.2,-4)--(1,-2);
         \draw[thin ](0.8,-4)--(1,-2);

\draw[fill=white] (0,0)circle(2pt);
\draw[fill=white] (0,-1)circle(2pt);
\draw[fill=white] (0,-2)circle(2pt);
\draw[fill=white] (0,-4)circle(2pt);
\draw[fill=white] (-0.2,-4)circle(2pt);
\draw[fill=white] (0.2,-4)circle(2pt);

\draw[fill=white] (-1,-3)circle(2pt);
\draw[fill=white] (-1.1,-4)circle(2pt);
\draw[fill=white] (-0.9,-4)circle(2pt);

\draw[fill=white] (1,-2)circle(2pt);
\draw[fill=white] (1,-4)circle(2pt);
\draw[fill=white] (0.8,-4)circle(2pt);
\draw[fill=white] (1.2,-4)circle(2pt);

\node(a)at(0.2,0){   $1$};
\node(a)at(0.2,-.8){ \small{$\frac{3}{2}$}};
\node(a)at(0.2,-1.8){ \small{$\frac{13}{6}$}};
\node(a)at(1.2,-1.8){ \small{$\frac{13}{6}$}};
\node(a)at(-1.2,-2.8){ \small{$\frac{5}{2}$}};

\end{tikzpicture}\caption{The carrousel   tree}\label{fig:2}
\end{figure}
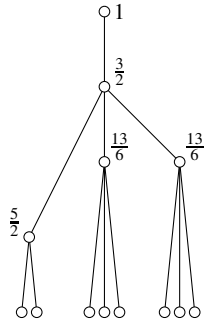 

Now, for small $t \in \R^+$,  consider the intersection $C \cap \{x=t\}$. This gives $8$ points $p_i(t), i=1 \ldots,8$ and then, varying $t$,  this gives $8$ real semi-analytic arcs $p_i \colon [0,1) \to X$  such that $p_i(0)=0$ and $\lVert p_i(t) \rVert = \Theta(t)$. 

Figure \ref{fig:3} gives pictures of sections of $C$ with complex lines
 $x=0.1$, $0.05$, $0.025$ and $0$. { The central two-points set corresponds to the  branch $y=x^{5/2}$ while the two lateral  three-points sets correspond to the other branch. }

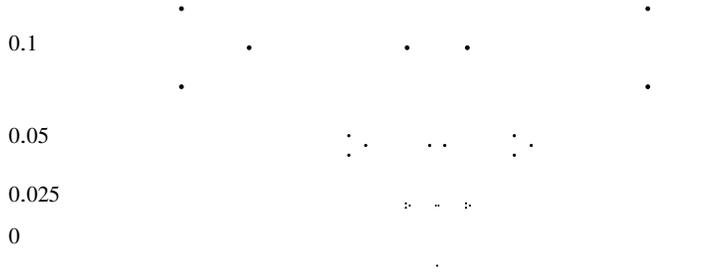
\begin{figure}[ht]
  \centering
\vbox to 0 pt{\vglue12pt\hbox to 0 pt{$0.1$\hss}
\vglue25pt\hbox to 0 pt{$0.05$\hss}
\vglue12pt\hbox to 0 pt{$0.025$\hss}
\vglue6pt\hbox to 0 pt{$0$\hss}\vss}

\begin{tikzpicture}

\draw[fill=black] (3.1,0)+(0:.6)circle(.7pt);
\draw[fill=black] (3.1,0)+(120:.6)circle(.7pt);
\draw[fill=black] (3.1,0)+(240:.6)circle(.7pt);
\draw[fill=black] (-3.1,0)+(0:.6)circle(.7pt);
\draw[fill=black] (-3.1,0)+(120:.6)circle(.7pt);
\draw[fill=black] (-3.1,0)+(240:.6)circle(.7pt);
\draw[fill=black] (0,0)+(0:0.4)circle(.7pt);
\draw[fill=black] (0,0)+(180:0.4)circle(.7pt);

\begin{scope}[yshift=-0.3cm] 
\draw[fill=black] (1.1,-1)+(0:.15)circle(.4pt);
\draw[fill=black] (1.1,-1)+(120:.15)circle(.4pt);
\draw[fill=black] (1.1,-1)+(240:.15)circle(.4pt);
\draw[fill=black] (-1.1,-1)+(0:.15)circle(.4pt);
\draw[fill=black] (-1.1,-1)+(120:.15)circle(.4pt);
\draw[fill=black] (-1.1,-1)+(240:.15)circle(.4pt);
\draw[fill=black] (0,-1)+(0:0.1)circle(.4pt);
\draw[fill=black] (0,-1)+(180:0.1)circle(.4pt);
\end{scope}

\begin{scope}[yshift=-0.6cm] 
\draw[fill=black] (.4,-1.5)+(0:.034)circle(.2pt);
\draw[fill=black] (.4,-1.5)+(120:.034)circle(.2pt);
\draw[fill=black] (.4,-1.5)+(240:.034)circle(.2pt);
\draw[fill=black] (-.4,-1.5)+(0:.034)circle(.2pt);
\draw[fill=black] (-.4,-1.5)+(120:.034)circle(.2pt);
\draw[fill=black] (-.4,-1.5)+(240:.034)circle(.2pt);
\draw[fill=black] (0,-1.5)+(0:0.018)circle(.15pt);
\draw[fill=black] (0,-1.5)+(180:0.018)circle(.15pt);
\end{scope}
\begin{scope}[yshift=-0.9cm] 
\draw[fill=black] (0,-2)+(0:0)circle(.2pt);
\end{scope}

\end{tikzpicture}
\caption{Sections of $C$}\label{fig:3}
\end{figure}

It is easy to see on this example  that  for each  pair $(i, j)$ with $i \neq j$, we have $d_o(p_i(t),p_j(t)) = \Theta(t^{q(i,j)})$ where $q(i,j)\in \Q^+$ and that the set of such $q(i,j)$'s is exactly the set of  essential exponents $\{3/2, 13/6, 5/2\}$. This shows that one can recover the essential exponents by measuring the outer distance between points of $C$.  \end{example}

More generally, we will show that we can actually recover the carrousel tree  by measuring outer distances on $X$ even after a bilipschitz change of the metric. Conversely, the outer Lipschitz geometry of a plane curve is determined by its embedded topological type. This gives the complete classification of the outer geometry of complex plane curve germs:
\begin{theorem} \label{th:main} Let $(E_1,0)\subset (\C^2,0)$ and
  $(E_2,0)\subset (\C^2,0)$ be two germs of complex curves. The
  following are equivalent:
   \begin{enumerate}
  \item\label{it1} $(E_1,0)$ and $(E_2,0)$ have same outer  Lipschitz
    geometry. 
  \item\label{it2} there is a meromorphic germ 
    $\phi\colon (E_1,0)\to (E_2,0)$ which is a bilipschitz  homeomorphism for the outer
    metric;
  \item\label{it3} $(E_1,0)$ and $(E_2,0)$ have the same embedded
    topological type; 
  \item\label{it4} there is a bilipschitz homeomorphism
    of germs $h\colon (\C^2,0) \to (\C^2,0)$ with $h(E_1)=E_2$.
  \end{enumerate}
\end{theorem}

As a corollary of Theorem \ref{generic projection bilipschitz} and Theorem \ref{th:main}, we obtain: 

\begin{corollary} The outer Lipschitz geometry of a curve germ $(C,0) \subset(\C^N,0)$ determines and is determined by the embedded topological type of any generic linear projection $(\ell(C),0) \subset (\C^2,0)$. 
\end{corollary}

The equivalence of \eqref{it1}, \eqref{it3} and \eqref{it4}  of Theorem \ref{th:main} is
proved in \cite{NeumannPichon2014}. The equivalence of \eqref{it2} and \eqref{it3}
was first proved by  Pham and Teissier \cite{PhamTeissier1969} by developing the theory of Lipschitz saturation and revisited by Fernandes in \cite{Fernandes2003}.
 In the present lecture notes, we will give the proof of  \eqref{it1} $\Rightarrow$ \eqref{it3}, since it is based on the so-called {\it bubble trick} argument which can be considered as a prototype for exploring Lipschitz geometry of singular spaces in various settings. Another  more sophisticated  bubble trick argument   
 is developed in \cite{NeumannPichon2012} to study Lipschitz geometry of 
 surface germs (namely in the proof of Theorem \ref{thm:outer invariant}).

 \begin{proof}[  of \eqref{it1} $\Rightarrow$ \eqref{it3} of Theorem \ref{th:main}]
We want to prove  that the embedded topological type of a
plane curve germ $(C,0) \subset (\C^2,0)$ is determined by the outer
Lipschitz geometry of $(C,0)$.
  
We first prove this using the analytic structure and the outer metric
on $(C,0)$. The proof is close to Fernandes' approach in
\cite{Fernandes2003}.  We then modify the proof to
make it purely topological and to allow a bilipschitz change of the
metric.

The tangent cone to $C$ at $0$ is a union of lines $L^{(j)}$,
$j=1,\dots,m$, and by choosing our coordinates we can assume they are
all transverse to the $y$-axis.

There is $\epsilon_0 >0$ such that for every   $\epsilon \in (0,\epsilon_0]$,   the
curve $C$  meets transversely the set
$$T_\epsilon:=\{(x,y)\in \C^2: |x|= \epsilon\}\,.$$

Let $M$ be the multiplicity of $C$. The hypothesis of transversality to the $y$-axis means that the  lines $x=t$ for $t\in
(0,\epsilon_0]$ intersect $C$ in $M$ points
$p_1(t),\dots,p_{M}(t)$. Those points depend continuously on $t$.  Denote by $[M]$ the set
$\{1,2,\dots,M\}$. For each
$  j, k \in [M]$ with $j < k$, the distance $d(p_j(t),p_k(t))$ has the form
$O(t^{q(j,k)})$, where $q(j,k) = q(k,j) \in \mathbb{Q} \cap [1, + \infty)$ is either a characteristic Puiseux
exponent for a branch of the plane curve $C$ or a coincidence exponent
between two branches of $C$ in the sense of e.g., \cite{LeMichelWeber1989}. We call
such exponents {\bf essential}. \index{Puiseux!essential exponent}

 For $j\in [M]$, define $q(j,j)=\infty$.

\begin{lemma} \label{le:curve geometry}The map $q\colon
  [M]\times[M]\to \Q\cup\{\infty\}$, $(j,k)\mapsto q(j,k)$,
  determines the embedded topology of $C$.
\end{lemma}

\begin{proof}   To prove the lemma we will construct from $q$ the so-called {\it carrousel tree}\index{carrousel tree}.   Then, we will show that it encodes the same data as the Eggers tree. This implies that it determines the  embedded topology of $C$.

 The $q(j,k)$ have the property that $q(j,l) \ge min (q(j,k),q(k,l))$
  for any triple $j,k,l$. So for any $q\in \Q\cup\{\infty\}, q>0$, the binary 
  relation on the set $[M]$ defined by $j\sim_q k\Leftrightarrow
  q(j,k)\ge q$ is an equivalence relation.

  Name the elements of the set $q([M]\times[M])\cup\{1\}$ in
  decreasing order of size: $\infty=q_0>q_1>q_2>\dots>q_s=1$.
  For each $i=0,\dots,s$ let $G_{i,1},\dots,G_{i,M_i}$ be the
  equivalence classes for the relation $\sim_{q_i}$. So $M_0=M$
  and the sets $G_{0,j}$ are singletons while $M_s=1$ and
  $G_{s,1}=[M]$. We form a tree with these equivalence classes $G_{i,j}$ as
  vertices, and edges given by inclusion relations:  the
  singleton sets $G_{0,j}$ are the leaves and there is an edge between
  $G_{i,j}$ and $G_{i+1,k}$ if $G_{i,j}\subseteq G_{i+1,k}$. The vertex
  $G_{s,1}$ is the root of this tree. We weight each vertex   with its corresponding $q_i$.

  The {\bf carrousel tree} is the tree obtained from this
  tree by suppressing valence $2$ vertices (i.e., vertices with exactly two incident edges): we remove each such vertex
  and amalgamate its two adjacent edges into one edge. We follow the computer science
convention of drawing the tree with its root vertex at the top,
descending to its leaves at the bottom (see Figure \ref{fig:2}).

At any non-leaf vertex $v$ of the carrousel tree
  we have a weight $q_v$, $1\le q_v\le q_1$, which is one of the
  $q_i$'s. We write it as $m_v/n_v$, where $n_v$ is the ${\operatorname{lcm}}$ of the
  denominators of the $q$-weights at the vertices on the path from $v$
  up to the root vertex. If $v'$ is the adjacent vertex above $v$
  along this path, we put $r_v=n_v/n_{v'}$ and $s_v=n_v(q_v-q_{v'})$.
  At each vertex $v$ the subtrees cut off below $v$ consist of groups
  of $r_v$ isomorphic trees, with possibly one additional tree.  We
  label the top of the edge connecting to this additional tree at $v$,
  if it exists, with the number $r_v$, and then delete all but one from
  each group of $r_v$ isomorphic trees below $v$. We do this for each
  non-leaf vertex of the carrousel tree. The resulting tree,
  with the $q_v$ labels at vertices and the extra label on a downward
  edge at some vertices is easily recognized as a mild modification of
  the Eggers tree: there is a natural action of the  Galois group whose quotient is  the Eggers tree.
\end{proof}

As already noted, this reconstruction of the embedded topology involved the
complex structure and  the outer metric. We must show  that we can reconstruct 
it without using  the complex structure, even after applying a
bilipschitz change to the outer metric. We will use what we call a {\bf bubble trick}.

Recall that the tangent cone of $C$ is a union of lines $L^{(j)}$. We
denote by $C^{(j)}$ the part of $C$ tangent to the line $L^{(j)}$.  It
suffices to recover the topology of each $C^{(j)}$ independently,
since the $C^{(j)}$'s are distinguished by the fact that the distance
between any two of them outside a ball of radius $\epsilon$ around $0$
is $\Theta(\epsilon)$, even after bilipschitz change of
the metric. We therefore assume from now on that the tangent cone of 
$C$ is a single complex line.

We now arrive at a crucial moment of the proof and of the paper.

\noindent{\bf The bubble trick.} \index{bubble trick} The points $p_1(t),\dots,p_M(t)$ which we used in order to find the numbers
$q(j,k)$ were obtained by intersecting $C$ with the line $x=t$. The arc
$p_1(t)$, $t\in [0,\epsilon_0]$ satisfies $d(0,p_1(t))=\Theta(t)$. Moreover,
the other points $p_2(t),\dots,p_M(t)$ are in the transverse disk of 
radius $rt$ centered at $p_1(t)$ in the plane $x=t$. Here $r$ can be as
small as we like, so long as $\epsilon_0$ is then chosen sufficiently
small.

Instead of a transverse disk of radius $rt$, we can use a ball
$B(p_1(t),rt)$ of radius $rt$ centered at $p_1(t)$. This ball 
$B(p_1(t),rt)$ intersects $C$ in $M$ disks $D_1(t),\dots,D_M(t)$,
and we have $d(D_j(t),D_k(t))=\Theta(t^{q(j,k)})$, so we still recover the
numbers $q(j,k)$.  In fact, the ball in the outer metric on $C$ of
radius $rt$ around $p_1(t)$ is $B_C(p_1(t),rt):=C\cap B(p_1(t),rt)$,
which consists of these $M$ disks $D_1(t),\dots,D_M(t)$.

We now replace the arc $p_1(t)$ by any continuous arc $p'_1(t)$ on
$C$ with the property that $d(0,p'_1(t))=\Theta(t)$. If $r$ is
sufficiently small it is still true that $B_C(p'_1(t),rt)$ consists of
$M$ disks $D'_1(t),\dots,D'_M(t)$ with
$d\bigl(D'_j(t),D'_k(t)\bigr)=\Theta(t^{q(j,k)})$.   So at this point, we have gotten
rid of the dependence on analytic structure in discovering the
topology, but not yet  of the dependence on the outer geometry.

Let now $d'$ be a metric on $C$ such that the identity map is a $K$-bilipschitz homeomorphism in a neighbourhood of the origin. We work inside this neighbourhood, taking $t, \epsilon_0$ and $r$  sufficiently small.  $B'(p,\eta)$ will denote the ball in $C$ for the metric $d'$ centered at $p \in C$ with radius $\eta\geq 0$. 

The bilipschitz change of the metric may disintegrate the balls in many connected components, as sketched on Figure \ref{fig:4}, where  the  round ball  $B_C(p'_1(t), rt)$  has $3$ components ($3$ is the mulitplicity of $C$), while  $B'(p'_1(t), rt)$ has $6$  components  (for clarity of the picture,  we draw the ball  $B'(p'_1(t), rt)$    as if   the distance $d'$ were induced by an ambient metric, but  this is not the case in general).

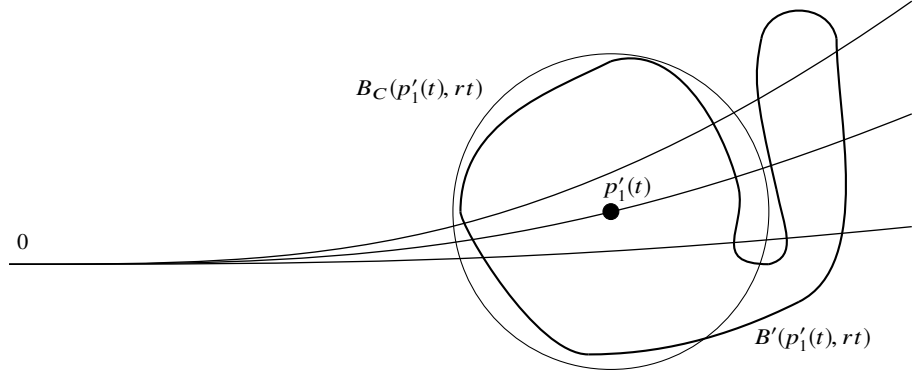
\begin{figure}[ht]
\centering
\begin{tikzpicture} 

\draw[line width=.5pt ]  (-6,0) .. controls (-2,0) and (1,0)  .. (6,3.5);
\draw[line width=.5pt ]  (-6,0) .. controls (-2,0) and (1,0)  .. (6,2);
\draw[line width=.5pt ]  (-6,0) .. controls (-2,0) and (1,0)  .. (6,.5);

 \draw[thick]  (2,2.7).. controls (1.2,2.3) and 
(0,1.9) ..(0,0.7);
 \draw[thick]   (0,0.7).. controls (0,.5) and 
(1,-1.2) ..(1.7,-1.2);
 \draw[thick]   (1.7,-1.2).. controls (2.6,-1.2) and 
(3.5,-1) ..(4.5,-0.5);
 \draw[thick]   (4.5,-0.5).. controls (5.5,0) and 
(5,2) ..(5,3);
\draw[thick]   (5,3).. controls (4.9,3.5) and 
(4.1,3.5) ..(4,3);
\draw[thick]   (4,3).. controls (3.9,2.5) and 
(3.9,2.3) ..(4.2,1);
\draw[thick]   (4.2,1).. controls (4.3,.5) and 
(4.5,.1) ..(4.1,0);
\draw[thick]   (4.1,0).. controls (3.3,0) and 
(3.8,.5) ..(3.7,1);
\draw[thick]   (3.7,1).. controls (3.6,1.5) and 
(3,3) .. (2,2.7);

 \node at(-6,0.3)[right]{$0$};

\draw[] (2,.7)circle(2.1cm);
\draw[fill=black] (2,.7)circle(3pt);

 \node at(1.8,1)[right]{$\small p'_1(t)$};
 
  \node at(3.8,-1)[right]{$\small B'(p'_1(t), rt)$};
 
  \node at(-1.5,2.3)[right]{$\small B_C(p'_1(t),rt)$};
  



 \end{tikzpicture} 
 \caption{Change of the metric}  
  \label{fig:4}
\end{figure}

  If we try to perform the same argument as before using the balls $B'(p'_1(t),rt)$ instead of $B_C(p'_1(t),rt)$, we get a problem since $B'(p'_1(t),rt)$ may have many irrelevant components and  we can no longer simply use distance between connected components. To resolve this, we consider the two balls  
$B'_1(t)=B'(p'_1(t), \frac{rt}{K^3})$ and $B'_2(t)=B'(p'_1(t),\frac{rt}{K})$, we have the inclusions:
$$B_C\big(p'_1(t), \frac{rt}{K^4}\big) \ \ \subset \ \  B'_1(t) \ \  \subset  \ \  B_C\big(p'_1(t), \frac{rt}{K^2}\big)  \ \ \subset  \ \ B'_2(t) \ \ \subset\ \   B_C\big(p'_1(t), rt\big)$$

 Using these inclusions, we obtain that only $M$ components of
$B'_1(t)$  intersect $B'_2(t)$ and that naming
these components $D'_1(t),\dots,D'_M(t)$ again, we still have
$d(D'_j(t),D'_k(t))=\Theta(t^{q(j,k)})$ so the $q(j,k)$ are determined as
before (prove this as an exercise). See Figure \ref{fig:4.1} for a schematic picture of the situation (again, for clarity of the picture,  we draw the balls $B'_1(t)$ and $B'_2(t)$ as if   the distance $d'$ were induced by an ambient metric, but  this is not the case in general). 
\begin{figure}[ht]
\centering
\begin{tikzpicture} 

\draw[line width=.5pt ]  (-6,0) .. controls (-2,0) and (1,0)  .. (6,3.5);
\draw[line width=.5pt ]  (-6,0) .. controls (-2,0) and (1,0)  .. (6,2);
\draw[line width=.5pt ]  (-6,0) .. controls (-2,0) and (1,0)  .. (6,.5);

\draw[line width=2.5pt, gray]   (0.02,0.52).. controls (1.6,1)  .. (3.4,1.87);
\draw[line width=2.5pt, gray]   (0.15,0.32).. controls (2,.65)  .. (3.65,1.15);
\draw[line width=2.5pt, gray]   (0.3,0.08).. controls (2,.2)  .. (3.65,0.27);

 \draw[thick]  (2,1.7).. controls (1,1.3) and 
(.9,.9) ..(1.5,0.7);
 \draw[thick]   (1.5,0.7).. controls (2.1,.5) and 
(1,.4) ..(1.3,-0.1);
 \draw[thick]   (1.3,-0.1).. controls (1.5,-.2) and 
(2.8,-.8) ..(2.9,0);
 \draw[thick]   (2.9,0).. controls (2.9,.7) and 
(2.7,.3) ..(2.7,0);
 \draw[thick]   (2.7,0).. controls (2.7,-.2) and 
(2.5,-.2) ..(2.5,0);
 \draw[thick]   (2.5,0).. controls (2.5,.3) and 
(2.8,.5) ..(3,1);
 \draw[thick]   (3,1).. controls (3.3,1.3) and 
(2.5,1.8) .. (2,1.7);

 \draw[thick]  (2,2.7).. controls (1.2,2.3) and 
(0,1.9) ..(0,0.7);
 \draw[thick]   (0,0.7).. controls (0,.5) and 
(1,-1.2) ..(1.7,-1.2);
 \draw[thick]   (1.7,-1.2).. controls (2.6,-1.2) and 
(3.5,-1) ..(4.5,-0.5);
 \draw[thick]   (4.5,-0.5).. controls (5.5,0) and 
(5,2) ..(5,3);
\draw[thick]   (5,3).. controls (4.9,3.5) and 
(4.1,3.5) ..(4,3);
\draw[thick]   (4,3).. controls (3.9,2.5) and 
(3.9,2.3) ..(4.2,1);
\draw[thick]   (4.2,1).. controls (4.3,.5) and 
(4.5,.1) ..(4.1,0);
\draw[thick]   (4.1,0).. controls (3.3,0) and 
(3.8,.5) ..(3.7,1);
\draw[thick]   (3.7,1).. controls (3.6,1.5) and 
(3,3) .. (2,2.7);

 \node at(-6,0.3)[right]{$0$};

\draw[] (2,.7)circle(1.5cm);
\draw[fill=black] (2,.7)circle(3pt);

 \node at(1.8,1)[right]{$\small p'_1(t)$};
 
  \node at(5.1,1)[right]{$\small B'_1(t)$};
 
  \node at(-3,1.3)[right]{$\small B_C(p'_1(t), \frac{rt}{K^2})$};
  
    \node at(-1.2,-1)[left]{$\small  B'_2(t)$};

\draw[very thin] (2,0)--(-1,-1); 

\draw[very thin] (0.8,1)--(-2,1); 

 \end{tikzpicture} 
 \caption{The bubble trick}  
  \label{fig:4.1}
\end{figure}
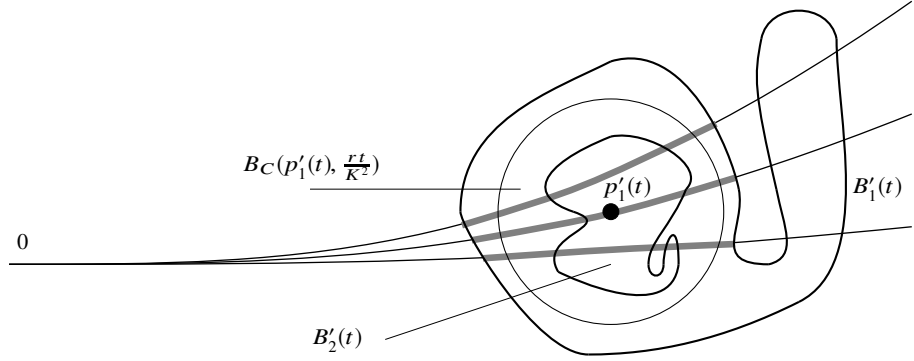
 \end{proof}

 \subsection{ The bubble trick with jumps} \label{sub:with jumps}

The  ``bubble trick"   introduced in the  proof of Theorem \ref{th:main} is a powerful tool to capture invariants of Lipschitz geometry of a complex curve germ. However, this first version of the bubble trick  is not well adapted  to explore the outer  Lipschitz geometry of   a singular space of dimension  $\geq 2$  for the following reason.    In the case of a plane curve germ $(C,0)$, the bubble trick is based on the fact that   the  distance orders  between points of $\ell^{-1}(t) \cap C$ with respect to $t \in \R$ are Lipschitz invariants, where  $\ell \colon  (C,0) \to (\C,0)$ denotes a generic projection of the curve germ.  Now, assume that $(X,0)$ is a complex surface germ  with multiplicity $m  \geq 2$ (so it has a singularity at $0$), and consider a  generic projection $\ell \colon (X,0) \to (\C^2,0)$.  Then  the critical locus of $\ell $  is a curve germ   $(\Pi_{\ell},0) \subset(X,0)$ called the {\bf polar curve}, \index{polar curve} and its image $\Delta_{\ell} = \ell(\Pi_{\ell})$ is a curve germ $(\Delta_{\ell},0) \subset (\C^2,0)$ called the {\bf discriminant curve} \index{discriminant curve} of $\ell$. Let $x \in  \C^2 \setminus \{0\}$. The number of points in $\ell^{-1}(x) \cap C$ depends on $x$: it equals $m-1$ if $x \in \Delta_{\ell}$, where $m$ denotes the multiplicity of $(X,0)$,  and $ m$ otherwise. Moreover, consider a semialgebraic real  arc germ  $p \colon t  \in [0, \eta) \mapsto p(t) \in \C^2$ such that $\norm{p(t)} = |t|$ and $\forall t \neq 0, p(t) \not\in \Delta_{\ell}$; then the distance orders between the $m$ points $p_1(t),\ldots,p_m(t)$ of $\ell^{-1}(p(t))$  will depend on the position of the arc $p(t)$ with respect to the curve $\Delta_{\ell}$. So the situation is much more complicated, even in dimension $2$. 

In \cite{NeumannPichon2014}, we use an adapted version of the bubble trick which enables us to explore the outer Lipschitz  geometry of a complex surface $(X,0)$. We call it the {\bf bubble trick with jumps}. \index{bubble trick!with jumps} Roughly speaking, it  consists in using horns 
$${\cal H}(p(t), r|t|^q) = \bigcup_{t \in [0,1)}  B((p(t), r|t|^q),$$ 
where $B(x,a)$ denotes the ball  in $X$ with center $x$ and radius $a$ and
where $p(t)$ is a real arc on $(X,0)$ such that  $||p(t)|| =\Theta(t)$ and $r \in ]0,+\infty[$, and in  exploring ``jumps" in the  topology of ${\cal H}(p(t), a|t|^q)$ when $q$ varies from $+\infty$ to $1$, for example, jumps of  the  number of connected components of ${\cal H}(p(t), r|t|^q) \setminus \{0\}$. 

In order to give a flavour of this bubble trick with jumps, we will perform it on a plane curve germ, giving an alternative proof of the implication  \eqref{it1} $\Rightarrow$ \eqref{it3} of Theorem \ref{th:main}. 
 
 \vskip0,3cm
 \noindent
 {\bf The bubble trick with jumps.} 
 
We use again the notations of the version 1 of the bubble trick from the proof of  Theorem \ref{th:main}. Let $(C,0)$ be a plane curve germ with multiplicity $M$  and with $s$ branches $C_1, \ldots, C_s$. Let $p'_1(t)$ be a continuous arc on
$C_1$ with the property that $d(0,p'_1(t))=\Theta(t)$.  Let us order the numbers $q(1,k), k=2,\ldots,M$ in decreasing order:
$$1 \leq q(1,M) < q(1,M-1) <\cdots < q(1,2) < q(1,1)=\infty. $$

 Let us consider the horns  ${\cal H}_{q,r}=    {\cal H}(p'_1(t),  r |t|^q)$ with $q \in [1,+\infty[$. 
 
 For $q >>1$ and small $\epsilon>0$, the number of connected components of    $ B(0,\epsilon) \cap \big( {\cal H}_{q,r}   \setminus \{0\} \big)$ equals $1$. Now, let us decrease $q$. For every  $\eta>0$ small enough, the number of connected components of  $ {\cal H}_{q_1,2+\eta} \setminus \{0\}$ equals $1$, while the number of connected components of   $ {\cal H}_{q_1,2-\eta} \setminus \{0\}$ is $>2$. Decreasing $q$, we have a jump in the number of connected components exactly when passing one of the rational numbers $q(1,k)$. So this enables one to recover all the characteristic exponents of $C_1$ and its  contact exponents with the other branches of $C$. We can do the same for a real arc $p'_i(t)$ in each branch $C_i$ of $(C,0)$ and this  will recover the integers $q(i,k)$ for $k=1,\ldots,M$. We then reconstruct the function $q \colon [M] \times [M] \to \Q_{\geq1}$ which characterizes the embedded topology of $(C,0)$, or equivalently the carrousel tree of $(C,0)$.  
 
 Moreover, the same jumps appear when one uses instead horns
  $${\cal H'}(p'(t), r|t|^q) = \bigcup_{t \in [0,1)}  B'((p'(t), r|t|^q),$$
 where $B'$ denotes balls with respect to a metric  $d'$ which is bilipschitz equivalent to the initial outer metric.  Indeed, if $K$ is the bilipschitz constant of such a bilipschitz change, then we have the inclusions 
 $${\cal H}\big(p'(t), \frac{rt}{K^4}\big) \ \ \subset \ \  {\cal H'}\big(p'(t), \frac{rt}{K^3}\big) \ \  \subset  \ \  {\cal H}\big(p'(t), \frac{rt}{K^2}\big)  $$
 $$ \hskip6cm\ \ \subset  \ \  {\cal H'}\big(p'(t), \frac{rt}{K^3}\big)\ \ \subset\ \   {\cal H}\big(p'(t), rt\big).$$
Then the same argument as in the version 1 of the bubble trick shows that for $q$ fixed and different from $q(1,k), k=2,\ldots,M$,  the numbers of connected components of $ B(0,\epsilon) \cap \big( {\cal H}_{q,r}   \setminus \{0\} \big)$ and $ B(0,\epsilon) \cap \big( {\cal H'}_{q,r}   \setminus \{0\} \big)$  are equal.

\begin{example} Consider again the plane curve singularity with two branches of Example \ref{ex:carrousel tree} given by the Puiseux series:
$$C_1: \ \ y=x^{3/2}+x^{13/6},\quad C_2: \ \  y=x^{5/2}\,.$$
Consider first an arc $p'_1(t)$   inside $C_1$ parametrized by $x=t \in [0 ,1)$. Then $p'_1(t)$  corresponds  to one of the two extremities of the carrousel tree of Figure \ref{fig:2} whose neighbour vertex is weighted by $5/2$.  Figure \ref{fig:bubble trick2-1} represents the intersection of the horn  
$\mathcal{H}_{q,r}$ with the line $\{ x = t \}$ for different values of $q \in [1, + \infty[$ and for $t \in \mathbb{C}^*$ of sufficiently small absolute value. 
This shows two jumps: a first jump at $q=5/2$, which says that $5/2$ is a characteristic exponent of a branch since $p_1'(t)$ and the new point appearing in the intersection belong to the same connected component $C_1$  of $C \setminus \{0\}$, while the second   jump at $3/2$ says that $3/2$ is  the contact exponent of $C_1$ with the other component since the new points appearing at $q = 3/2 -   \eta$ belong to $C_2$.   
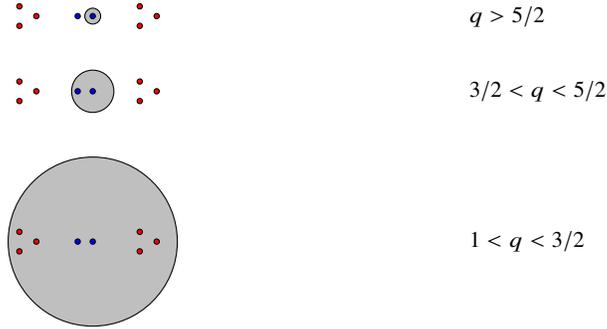
\begin{figure}[ht]
  \centering

\begin{tikzpicture}

   \node at(5,-1)[right]{$q >5/2$};

\draw[fill=red] (0.8,-1)+(0:.15)circle(1pt);
\draw[fill=red] (0.8,-1)+(120:.15)circle(1pt);
\draw[fill=red] (0.8,-1)+(240:.15)circle(1pt);
\draw[fill=red] (-0.8,-1)+(0:.15)circle(1pt);
\draw[fill=red] (-0.8,-1)+(120:.15)circle(1pt);
\draw[fill=red] (-0.8,-1)+(240:.15)circle(1pt);

\draw[fill=lightgray] (0,-1)+(0:0.1)circle(3pt);

\draw[fill=blue] (0,-1)+(0:0.1)circle(1pt);
\draw[fill=blue] (0,-1)+(180:0.1)circle(1pt);

\begin{scope}[yshift=-1cm] 

  \node at(5,-1)[right]{$3/2<q <5/2$};

\draw[fill=red] (0.8,-1)+(0:.15)circle(1pt);
\draw[fill=red] (0.8,-1)+(120:.15)circle(1pt);
\draw[fill=red] (0.8,-1)+(240:.15)circle(1pt);
\draw[fill=red] (-0.8,-1)+(0:.15)circle(1pt);
\draw[fill=red] (-0.8,-1)+(120:.15)circle(1pt);
\draw[fill=red] (-0.8,-1)+(240:.15)circle(1pt);

\draw[fill=lightgray] (0,-1)+(0:0.1)circle(8pt);

\draw[fill=blue] (0,-1)+(0:0.1)circle(1pt);
\draw[fill=blue] (0,-1)+(180:0.1)circle(1pt);

\end{scope}

\begin{scope}[yshift=-3cm] 

  \node at(5,-1)[right]{$1<q <3/2$};

\draw[fill=lightgray] (0,-1)+(0:0.1)circle(32pt);

\draw[fill=red] (0.8,-1)+(0:.15)circle(1pt);
\draw[fill=red] (0.8,-1)+(120:.15)circle(1pt);
\draw[fill=red] (0.8,-1)+(240:.15)circle(1pt);
\draw[fill=red] (-0.8,-1)+(0:.15)circle(1pt);
\draw[fill=red] (-0.8,-1)+(120:.15)circle(1pt);
\draw[fill=red] (-0.8,-1)+(240:.15)circle(1pt);

\draw[fill=blue] (0,-1)+(0:0.1)circle(1pt);
\draw[fill=blue] (0,-1)+(180:0.1)circle(1pt);

\end{scope}

\end{tikzpicture}
\caption{Sections of $C$ associated to the arc $p_1'(t)$}\label{fig:bubble trick2-1}
\end{figure}

This first exploration enables one to construct the left part of the carrousel tree of $C$ shown on Figure \ref{fig:bubble trick2-1 carrousel}, i.e., the one corresponding to the carrousel tree of $C_1$.
\begin{figure}[ht]
  \centering
 
\begin{tikzpicture}

  \draw[thin ](0,0)--(0,-2);
   \draw[thin ](0,-1)--(-1,-3);
       \draw[thin ](-1.1,-4)--(-1,-3);
         \draw[thin ](-0.9,-4)--(-1,-3);
         

\draw[fill=white] (0,0)circle(2pt);
\draw[fill=white] (0,-1)circle(2pt);

\draw[fill=white] (-1,-3)circle(2pt);
\draw[fill=blue] (-1.1,-4)circle(2pt);
\draw[fill=blue] (-0.9,-4)circle(2pt);

\node(a)at(0.2,0){   $1$};
\node(a)at(0.2,-.8){ \small{$\frac{3}{2}$}};
\node(a)at(0,-2.3){ ?};
\node(a)at(-1.2,-2.8){ \small{$\frac{5}{2}$}};
\end{tikzpicture}\caption{Partial carrousel tree associated to the arc $p_1'(t)$}\label{fig:bubble trick2-1 carrousel}
\end{figure}
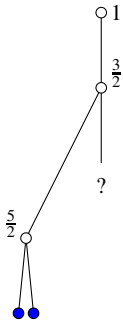

 To complete the picture, we now consider an arc $p'_2(t)$ inside $C_2$ corresponding to a component of $C_2 \cap \{x=t\}$. This means that $p'_2(t)$ corresponds to one of the 6 extremities of the carrousel tree of Figure \ref{fig:2} whose neighbour vertex is weighted by $13/6$. Figure \ref{fig:bubble trick2-2} represents the jumps for the horns   ${\cal H}_{q,r,}$ centered on $p'_2(t)$.  This shows two jumps: a first jump at $q=13/6$, which says that $13/6$ is a characteristic exponent of $C_2$, then a second jump at $3/2$ corresponding to   the contact exponent of $C_1$ and $C_2$.  

 \begin{figure}[ht]
  \centering

 \begin{tikzpicture}

   \node at(5,-1)[right]{$q >13/6$};
\draw[fill=lightgray]  (-0.8,-1)+(0:.15)circle(3pt);
\draw[fill=red] (0.8,-1)+(0:.15)circle(1pt);
\draw[fill=red] (0.8,-1)+(120:.15)circle(1pt);
\draw[fill=red] (0.8,-1)+(240:.15)circle(1pt);
\draw[fill=red] (-0.8,-1)+(0:.15)circle(1pt);
\draw[fill=red] (-0.8,-1)+(120:.15)circle(1pt);
\draw[fill=red] (-0.8,-1)+(240:.15)circle(1pt);

\draw[fill=blue] (0,-1)+(0:0.1)circle(1pt);
\draw[fill=blue] (0,-1)+(180:0.1)circle(1pt);

\begin{scope}[yshift=-1cm] 

  \node at(5,-1)[right]{$3/2<q <13/6$};
\draw[fill=lightgray]  (-0.8,-1)+(0:.15)circle(10pt);

\draw[fill=red] (0.8,-1)+(0:.15)circle(1pt);
\draw[fill=red] (0.8,-1)+(120:.15)circle(1pt);
\draw[fill=red] (0.8,-1)+(240:.15)circle(1pt);
\draw[fill=red] (-0.8,-1)+(0:.15)circle(1pt);
\draw[fill=red] (-0.8,-1)+(120:.15)circle(1pt);
\draw[fill=red] (-0.8,-1)+(240:.15)circle(1pt);

\draw[fill=blue] (0,-1)+(0:0.1)circle(1pt);
\draw[fill=blue] (0,-1)+(180:0.1)circle(1pt);

\end{scope}

\begin{scope}[yshift=-3.6cm] 

  \node at(5,-1)[right]{$1<q <3/2$};
\draw[fill=lightgray]  (-0.8,-1)+(0:.15)circle(50pt);

\draw[fill=red] (0.8,-1)+(0:.15)circle(1pt);
\draw[fill=red] (0.8,-1)+(120:.15)circle(1pt);
\draw[fill=red] (0.8,-1)+(240:.15)circle(1pt);
\draw[fill=red] (-0.8,-1)+(0:.15)circle(1pt);
\draw[fill=red] (-0.8,-1)+(120:.15)circle(1pt);
\draw[fill=red] (-0.8,-1)+(240:.15)circle(1pt);

\draw[fill=blue] (0,-1)+(0:0.1)circle(1pt);
\draw[fill=blue] (0,-1)+(180:0.1)circle(1pt);

\end{scope}

\end{tikzpicture}
\caption{Sections of $C$ associated to the arc $p_2'(t)$}\label{fig:bubble trick2-2}
\end{figure}
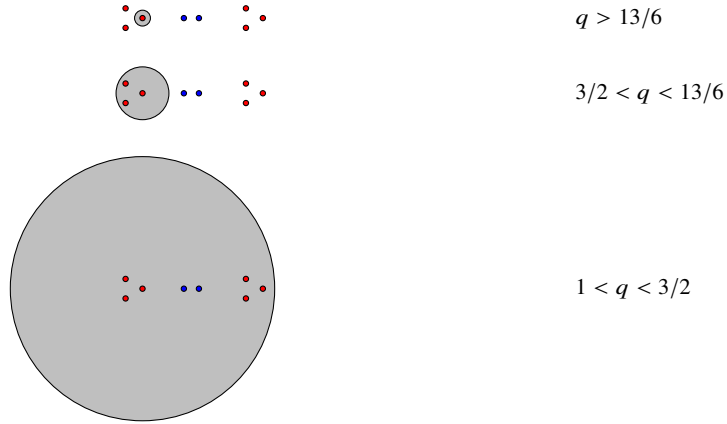

This   exploration of $C_2$ enables one to construct the right part of the carrousel tree of $C$ shown on Figure \ref{fig:bubble trick2-2 carrousel}, i.e., the one corresponding to the carrousel tree of $C_2$.
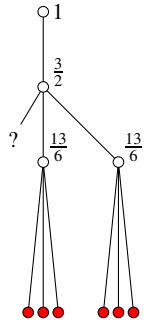
\begin{figure}[ht]
  \centering
 
\begin{tikzpicture}

  \draw[thin ](0,-2)--(0.2,-4);
   \draw[thin ](0,-2)--(-0.2,-4);
  \draw[thin ](0,0)--(0,-4);
   \draw[thin ](0,-1)--(-.3,-1.5);
         
         \draw[thin ](0,-1)--(1,-2);
            \draw[thin ](1,-4)--(1,-2);
       \draw[thin ](1.2,-4)--(1,-2);
         \draw[thin ](0.8,-4)--(1,-2);

\draw[fill=white] (0,0)circle(2pt);
\draw[fill=white] (0,-1)circle(2pt);
\draw[fill=white] (0,-2)circle(2pt);
\draw[fill=red] (0,-4)circle(2pt);
\draw[fill=red] (-0.2,-4)circle(2pt);
\draw[fill=red] (0.2,-4)circle(2pt);


\draw[fill=white] (1,-2)circle(2pt);
\draw[fill=red] (1,-4)circle(2pt);
\draw[fill=red] (0.8,-4)circle(2pt);
\draw[fill=red] (1.2,-4)circle(2pt);

\node(a)at(-0.4,-1.7){  ?};

\node(a)at(0.2,0){   $1$};
\node(a)at(0.2,-.8){ \small{$\frac{3}{2}$}};
\node(a)at(0.2,-1.8){ \small{$\frac{13}{6}$}};
\node(a)at(1.2,-1.8){ \small{$\frac{13}{6}$}};

\end{tikzpicture}\caption{Partial carrousel tree associated to the arc $p_2'(t)$ }\label{fig:bubble trick2-2 carrousel}
\end{figure}

Merging the two partial carrousel trees above, we obtain the carrousel tree of Figure \ref{fig:2}, recovering the embedded topology of $(C,0)$.
 \end{example}

 \section{The thick-thin decomposition of a surface singularity} \label{part 3}

\subsection{Fast loops as obstructions to metric conicalness} \label{sec:fast loops}

We know that every complex curve germ $(C,0)
  \subset (\C^N,0)$ is metrically conical for the inner geometry (Proposition \ref{prop:inner}). This is no longer true in higher dimensions. The first example of a non-metrically-conical   complex analytic  germ (X,0) appeared in \cite{BirbrairFernandes2008}: 
   for $k \geq 2$, the surface singularity $A_k \colon x^2+y^2+z^{k+1}=0$ is not metrically conical for the inner metric\footnote{Notice that in the real algebraic setting, it is easy to construct germs with dimension $\geq 3$   which are not metrically conical for the inner geometry. For example a  $3$-dimensional horn-shaped germ $(X,0)$ whose link $X^{(\epsilon)}$ is a $2$-torus with diameter $\Theta(\epsilon^2)$}. 
  The examples in \cite{BirbrairFernandesNeumann2008,BirbrairFernandesNeumann2009, BirbrairFernandesNeumann2010} then suggested that failure of metric conicalness is common. For example, among ADE singularities of surfaces, only $A_1$ and $D_4$ are metrically conical (Exercise \ref{ex:ADE}).  In \cite{BirbrairFernandesNeumann2010} it is also shown that the inner Lipschitz geometry of a singularity may not be determined by its topological type. 
  
  A complete classification of the Lipschitz inner geometry of normal complex surfaces is presented in \cite{BirbrairNeumannPichon2014}. It is based on the existence of the so-called thick-thin decomposition of the surface into two semi-algebraic sets. The aim of Sections \ref{sec:fast loops} to \ref{sec:thick-thin resolution} is to describe this decomposition. 
  
 The simplest  obstruction to the metric conicalness  of a germ $(X,0)$  is the existence of {\it fast loops} (see Definition \ref{def:fast loop} below). 
 
 Let $p$ and $q$ be two pairwise coprime positive integers such that $p\geq q$. Set $\beta=\frac{p}{q}$. The prototype of a fast loop is the {\bf $\beta$-horn}, which is the following semi-algebraic surface in $\mathbb{R}^3$: \index{$\beta$-horn}
 
 $${\cal H}_{\beta} =\{(x,y,z)\in \R^2 \times \R^+ \colon (x^2+y^2)^q=(z^2)^p\}.$$ 
 
   \begin{figure}[ht]
\centering
\begin{tikzpicture} 

\draw (0,3) ellipse (1cm and 0.3cm);
 \draw[thin] (-1,3)--(0,0);
  \draw[thin] (1,3)--(0,0);
   \node at(0,-0.5){$\beta=1$};
  
  \begin{scope}[xshift=4cm]
  \draw (0,3) ellipse (1cm and 0.3cm);
  \draw[thin]   (-1,2.95).. controls (-0.1,2)  .. (0,0);
   \draw[thin]   (1,2.95).. controls (0.1,2)  .. (0,0);
      \node at(0,-0.5){$\beta>1$};

 \end{scope}

 \end{tikzpicture} 
  \caption{The $\beta$-horns ${\cal H}_{\beta}$}\label{fig:5}
\end{figure}
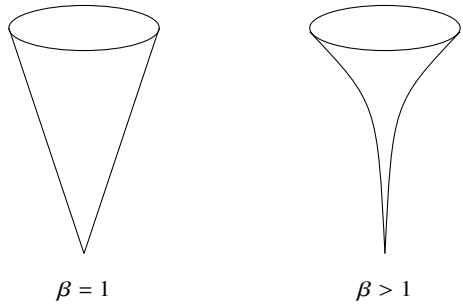

\begin{exercise} \label{exercise:horns} Show that ${\cal H}_{\beta}$ is inner bilipschitz homeomorphic to ${\cal H}_{\beta'}$ if and only if $\beta=\beta'$. \footnote{{\it hint:}  the length of the family of curves $C_t = {\cal H}_{\beta} \cap \{z=t\}$ is a $\Theta(t^{\beta})$ and this is invariant by a bilipschitz change of the metric. Show that such a family of curves cannot exist in ${\cal H}_{\beta'}$ if $\beta' \neq \beta$.}
  \end{exercise} 
  
 ${\cal H}_{1}$ is a straight cone, so it is metrically conical.  As a consequence of Exercise \ref{exercise:horns}, we obtain that for $\beta>1$, ${\cal H}_{\beta}$ is not metrically conical. For $t>0$, set $\gamma_{t}={\cal H}_{\beta} \cap \{z=t\}$. When $\beta>1$, the family of curves $(\gamma_t)_{t>0}$ is a {\it fast loop} inside ${\cal H}_{\beta}$. More generally:

\begin{definition} \label{def:fast loop} Let $(X,0) \subset (\R^n,0)$ be a semianalytic germ.  
    A  {\bf fast loop} \index{fast loop} in $(X,0)$ is a continuous family of loops  $\{ \gamma_\epsilon\colon
  S^1\to  X^{(\epsilon)}  \}_{0<\epsilon\le \epsilon_0}$ such that: 
  \begin{enumerate}
\item $ \gamma_\epsilon$ is essential  (i.e.,   homotopically non trivial)  in the link   $ X^{(\epsilon)}= X\cap S_{\epsilon}$;  
\item   there exists $q>1$ such that $$\lim_{\epsilon \to 0} \frac{\hbox{length}(\gamma_\epsilon)}{ \epsilon^q}=0.$$
\end{enumerate}
    \end{definition} 
    
In the next section, we will define what we call {\bf the thick-thin decomposition of a normal surface germ} $(X,0)$. It consists in decomposing $(X,0)$  as a union of two semi-algebraic sets 
$ (X,0) = (Y,0) \bigcup  (Z,0)$ where $(Z,0)$ is {\it thin} (Definition \ref{def:thin}) and where $(Y,0)$ is  {\it thick} (Definition \ref{def:thick}).   The  thin part $(Z,0)$  will contain all the fast loops of $(X,0)$ inside a Milnor ball with radius $\epsilon_0$. The  thick part $(Y,0)$ is the  closure of the complement of the thin part and has the property that it contains a maximal metrically conical set. This  enables one to characterize the germs $(X,0)$ which are metrically conical:

\begin{theorem} \cite[Theorem 7.5, Corollary 1.8]{BirbrairNeumannPichon2014}
Let $(X,0)$ be a normal complex surface  and let $$ (X,0) = (X_{thick},0) \bigcup  (X_{thin},0)$$ be its  thick-thin decomposition.  

Then $(X,0)$ is metrically conical if and only if $X_{thin} = \emptyset$, so  $ (X,0) = (X_{thick},0) $. 
\end{theorem}

\subsection{Thick-thin decomposition}

\begin{definition} \label{dfn:tangent cone} Let $(Z,0) \subset (\R^n,0)$ be a semi-algebraic germ. The {\it tangent cone} \index{tangent cone}  of $(Z,0)$   is the set $T_0Z$ of vectors $v \in \R^n$ such that there exists a sequence of points $(x_k)$ in $Z \setminus \{0\}$ tending to $0$ and a sequence of positive real numbers $(t_k)$ such that $$\lim_{k \to \infty} \frac{1}{t_k} x_k = v.$$

\end{definition}

\begin{definition} \label{def:thin} A semi-algebraic germ $(Z,0) \subset (\R^n,0)$ of pure dimension  is {\bf thin} \index{thin germ} if the dimension of its tangent cone $T_0X$ at $0$ satisfies $\dim(T_0Z)<\dim(Z)$. 
\end{definition} 
 \begin{example} \label{ex:thin1} For every $\beta>1$,   the $\beta$-horn ${\cal H}_{\beta}$ is thin since $\dim({\cal H}_{\beta})=2$ while $T_0 {\cal H}_{\beta}$ is a half-line.   On the other hand,  ${\cal H}_{1}$ is not thin. 
\end{example}

\begin{example} \label{ex:thin2}  Let $\lambda \in \C^*$ and denote by  ${\cal C}_{\lambda}$  the plane curve with Puiseux parametrization $y=\lambda x^{5/3}$. Let $a, b \in \R$ such that $0<a<b$. Consider the semi-algebraic germ $(Z,0) \subset (\C^2,0)$ defined by $Z=\bigcup_{a \leq |\lambda|\leq b} {\cal C}_{\lambda}$.   The tangent cone $T_0Z$ is the complex line $y=0$, while $Z$ is $4$-dimensional, so $(Z,0)$ is thin. 

Let  $Z^{(\epsilon)}$  be the intersection of $Z$ with the boundary of the  bidisc  $\{|x|\leq \epsilon\} \times \{|y|\leq \epsilon\} $. By \cite{Durfee1983}, one obtains, up to homeomorphism (or diffeomorphism in a stratified sense),  the same link $Z^{(\epsilon)}$ as when intersecting with a round sphere. When $\epsilon>0$ is small enough,  $Z^{(\epsilon)} \subset \{|x|=\epsilon\} \times \{|y|\leq \epsilon\} $ and  the projection  $Z^{(\epsilon)} \to S^1_{\epsilon}$ defined by $(x,y) \to x$ is a locally trivial fibration whose  fibers are the  flat annuli  $A_t = Z \cap \{x=t\},  |t| = {\epsilon }$, and the lengths of the boundary components   of $A_t$ are $\Theta(\epsilon^{5/3})$. 

Notice that $Z$ can be described through a resolution as follows. Let $\sigma \colon Y \to \C^2$ be the minimal embedded resolution of the curve $  E_1 \colon y=x^{5/3}$. It decomposes into four successive blow-ups of points. Denote $E_1, \ldots, E_4$ the corresponding  components of the exceptional divisor $\sigma^{-1}(0)$ indexed by their order of creation.  Then $\sigma$  is a simultaneous resolution of the curves ${\cal C}_{\lambda}$. Therefore, the strict transform of $Z$ by $\sigma$ is a neighbourhood of $E_4$ minus   neighbourhoods of the intersecting points $E_4 \cap E_2$ and $E_4 \cap E_3$ as pictured in Figure \ref{fig:6}. The tree $T$ on the left is the dual   tree  of $\sigma$. Its vertices are weighted by the self-intersections $E_i^2$ and the arrow represents the strict transform of ${\cal C}_1$. 

\begin{figure}[ht] 
\centering
\begin{tikzpicture}
  
  \node[ ](b)at
  (-1.2,1.3){$T$}; 
  
  \node[ ](b)at (6.5,1.5){$E_4$}; 
  
   \node[ ](b)at (3.6,1.5){$E_2$}; 
  
   \node[ ](b)at (7.2,0){$E_3$}; 
   
    \node[ ](b)at (7.9,1.5){$E_1$}; 
  
    \draw[thin,>-stealth,->](0.5,1)--+(-.4,1);
        
   \node[ ]at (0.3,0.7){$-1$};  
     \node[ ]at (-0.7,0.7){$-3$};  
      \node[ ]at (-0.7,0.7){$-3$};  
       \node[ ]at (2.3,0.7){$-2$};  
       \node[ ]at (1.3,0.7){$-2$};   
         
  \draw[ xshift=-0.5cm,yshift=1cm,thin] (0+2pt,0)--(2.95,0);
  \draw[ xshift=-0.5cm,yshift=1cm, fill ](0,0)circle(2pt);
 \draw[  xshift=-0.5cm,yshift=1cm,fill] (1,0)circle(2pt);
 \draw[ xshift=-0.5cm,yshift=1cm, fill] (2,0)circle(2pt);
 \draw[xshift=-0.5cm,yshift=1cm, fill](3,0)circle(2pt);
  
 \begin{scope}[xshift=-3cm]
 \draw[scale=0.7,xshift=11cm,fill, lightgray](2,2)+(-45:0.2) arc
  (-45:135:0.2)--++(-2,-2) arc (135:315:0.2)--++(2,2);
 \draw[scale=0.7,xshift=13cm, fill, white](1,0)+(-135:0.2) arc
  (-135:45:0.2)--++(-2,2) arc (45:225:0.2)--++(2,-2);
 \draw[scale=0.7,xshift=11cm,fill,  white](0.5,0.5)--++(-45:0.2)--++(45:0.2)
  --++(135:0.4)--++(225:0.4)--++(-45:0.4)--++(45:0.2);
 \draw[scale=0.7,xshift=12cm, yshift=1cm, fill,
  white](0.5,0.5)--++(-45:0.2)--++(45:0.2)
  --++(135:0.4)--++(225:0.4)--++(-45:0.4)--++(45:0.2);
 \draw[scale=0.7,xshift=11cm,very thin](2,2)+(-45:0.2) arc
  (-45:135:0.2)--++(-2,-2) arc (135:315:0.2)--++(2,2);
 \draw[scale=0.7,xshift=13cm,very thin](2,2)+(-45:0.2) arc
  (-45:135:0.2)--++(-2,-2) arc (135:315:0.2)--++(2,2);
 \draw[scale=0.7,xshift=11cm,very thin](1,0)+(-135:0.2) arc
  (-135:45:0.2)--++(-2,2) arc (45:225:0.2)--++(2,-2);
 \draw[scale=0.7,xshift=13cm, very thin](1,0)+(-135:0.2) arc
  (-135:45:0.2)--++(-2,2) arc (45:225:0.2)--++(2,-2);
 \draw[scale=0.7,xshift=11cm](0,0)--(2,2);
 \draw[scale=0.7,xshift=11cm](-1,2)--(1,0);
 \draw[scale=0.7,xshift=11cm](1,2)--(3,0);
 \draw[scale=0.7,xshift=11cm](2,0)--(4,2);
 \end{scope}
   \draw[thin,>-stealth,->](5.7,0.3)--+(-0.8,0.8);
 \node[ ](b)at (4.7,1.3){${\cal C}_{\lambda}^*$}; 
\end{tikzpicture} 
\caption{The strict transform of $Z$ by $\sigma$} \label{fig:6}
\end{figure}
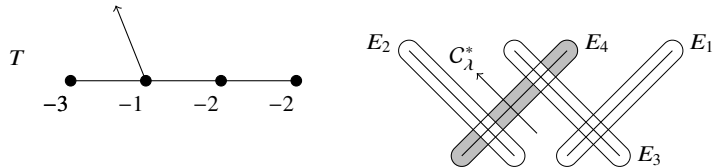

\end{example}

\begin{definition} Let $1< q\in \Q$. A {\bf $q$-horn neighbourhood} 
\index{horn neighbourhood} of a semi-algebraic
germ $(A,0)\subset (\R^N,0)$ is  a set of the form $\{x\in \R^n\cap
B_\epsilon:d(x,A)\le c|x|^q\}$ for some $c>0$, where $d$ denotes the Euclidean metric.   
\end{definition} 

The following proposition helps picture ``thinness'' 

\begin{proposition}  \cite[Proposition 1.3]{BirbrairNeumannPichon2014}  Any thin semi-algebraic germ $(Z,0)\subset (\R^N,0)$ is contained in some $q$-horn
  neighbourhood of its tangent cone $T_0 Z$. 
\end{proposition}

We will now define thick semi-algebraic sets. The definition is built on the following observation. Let $(X,0) \subset (\R^n,0)$ be a real algebraic germ; we would like to decompose  $(X,0)$ into two semialgebraic sets $(A,0)$ and $(B,0)$ glued along their boundary germs, where $(A,0)$ is thin and $(B,0)$ is metrically conical. But try to glue  a thin germ $(A,0)$  with a metrically conical germ $(B,0)$ so that they intersect only along their boundary germs.... It is not possible!  There would be a hole between them (see Figure \ref{fig:7}). So we have to replace $(B,0)$ by something else than conical. 

   \begin{figure}[ht]
\centering
\begin{tikzpicture}

  \path[ fill=lightgray] 
(0,0) -- (-0.55,2.44)  
(-0.55,2.44)    .. controls (0.1,1.8)   and (-0.1,1)  .. (0,0); 


  \draw[thick]   (-1,2.95).. controls (-0.1,2)  .. (0,0);
   \draw[thick]   (1,2.95).. controls (0.1,2)  .. (0,0);

 \draw[thick]   (-0.7,3.2)-- (0,0);
  \draw[thick]   (-2,2.7)-- (0,0);
  
   \draw[very thin]   (-2,1.5)-- (-0.2,1.4);
   
     \node at(-2.5,1.5){hole};

     \node at(-1,2){$B$};
     
      \node at(0,2.5){$A$};

 \end{tikzpicture}
  \caption{Trying to glue a thin germ with a metrically conical germ}\label{fig:7}
\end{figure}
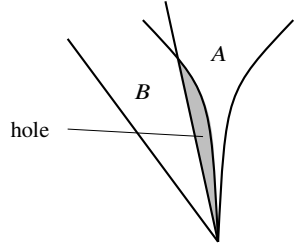

``Thick'' is a generalization of
  ``metrically conical.'' Roughly speaking, a thick
algebraic set is obtained by slightly inflating a metrically conical
set in order that it can interface along its boundary
with thin parts. The precise definition is as follows:

\begin{definition} \label{def:thick}
Let $B_\epsilon\subset
  \R^N$ denote the ball of radius $\epsilon$ centered at the origin,
  and $S_\epsilon$ its boundary.  A semi-algebraic germ $(Y,0)\subset
  (\R^N,0)$ is {\bf thick} \index{thick germ} if there exists $\epsilon_0>0$ and $K\ge
  1$ such that $Y\cap B_{\epsilon_0}$ is the union of subsets
  $Y_\epsilon$, $\epsilon\le\epsilon_0$ which are metrically conical
  with bilipschitz constant $K$ and satisfy the following properties
  (see Fig.~\ref{fig:1}):
  \begin{enumerate}
\item $Y_\epsilon\subset B_\epsilon$, $Y_\epsilon\cap
  S_\epsilon=Y\cap S_\epsilon$ and $Y_\epsilon$ is metrically
  conical;
\item For $\epsilon_1<\epsilon_2$ we have $Y_{\epsilon_2}\cap
  B_{\epsilon_1}\subset Y_{\epsilon_1}$ and this embedding respects
  the conical structures.  Moreover, the difference $(
  Y_{\epsilon_1}\cap S_{\epsilon_1})\setminus (Y_{\epsilon_2}\cap
  S_{\epsilon_1})$ of the links of these cones is 
  homeomorphic to $\partial(Y_{\epsilon_1}\cap
  S_{\epsilon_1})\times[0,1)$.
  \end{enumerate}
  \end{definition}

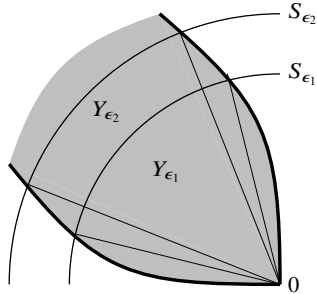
\begin{figure}[ht]
\centering
\begin{tikzpicture} 
 \node at(4,0)[right]{$0$};
 \node at(4,2.8)[right]{$ S_{\epsilon_1}$};
 \node at(4,3.6)[right]{$ S_{\epsilon_2}$};
 \path[ fill=lightgray] 
(4,0) .. controls (2,0) and (1.6,0.1)  .. (0.4,1.6) (0.4,1.6)--(2.4,3.6)
 (2.4,3.6)   .. controls (3.8,2.4)   and (4,2)  .. (4,0); 
 \path[ fill=lightgray]  (4,0)--(0.4,1.6).. controls (0.8,2.8) and 
(1.2,3.2)..(2.4,3.6)--cycle;
\draw[line width=1.3pt ]  (4,0) .. controls (2,0) and (1.6,0.1)  .. (0.4,1.6);
\draw[line width=1.3pt ]  (4,0) .. controls (4,2) and (3.8,2.4)  .. (2.4,3.6);
\draw[line width=0.5pt ]  (0.4,0) arc (180:90:3.6);
\draw[line width=0.5pt ]  (1.2,0) arc (180:90:2.8);
\draw[line width=0.25pt ] (4,0)--(0.65,1.33); 
\draw[line width=0.25pt ] (4,0)--(2.65,3.4); 
\draw[line width=0.25pt ] (4,0)--(1.28,0.67); 
\draw[line width=0.25pt ] (4,0)--(3.3,2.8); 
 \node at(2.5,1.5)[]{$ Y_{\epsilon_1}$};
 \node at(1.7,2.3)[]{$ Y_{\epsilon_2}$};
 \end{tikzpicture} 
  \caption{Thick germ}
  \label{fig:8}
\end{figure}

Clearly, a semi-algebraic germ cannot be both thick and thin.

\begin{example} The set $Z=\{(x,y,z)\in \R^3: x^2+y^2\le z^3\}$
gives a thin germ at $0$ since it is a $3$-dimensional germ whose
tangent cone is half the $z$-axis.  The intersection $Z\cap B_\epsilon$
is contained in a closed $3/2$-horn neighbourhood of the
$z$-axis.  The complement in $\R^3$ of this thin set is thick. 
\end{example}

\begin{example} Consider again the thin germ $(Z,0) \subset (\C^2,0)$ of Example \ref{ex:thin2}.  Then the germ  $(Y,0)$ defined by  $Y= \overline{\C^2 \setminus Z}$   is  a thick germ.  {To give an imaged picture of it, fix $\eta >0$  and consider the conical set $W \subset \C^2$ defined as the union of the complex lines $y=\alpha x$ for $|\alpha| \geq \eta$; then $(Y,0)$ is obtained by ``slighly inflating'' $W$.}
Notice that the strict transform of $Y$ by the resolution $\sigma$ introduced in  Example \ref{ex:thin2} is a neighbourhood of the union of curves $E_1  \cup E_3$. 
\end{example}

For any semi-algebraic germ $(A,0)$ of $(\R^N,0)$, we write 
$A^{(\epsilon)}:=A\cap S_\epsilon\,\subset\, S_\epsilon$. When $\epsilon$ is
sufficiently small, $A^{(\epsilon)}$ is the $\epsilon$-link of $(A,0)$.

\begin{definition}[Thick-thin decomposition]\label{def:thick-thin}
  A {\bf thick-thin decomposition}  \index{thick-thin decomposition} of the normal complex surface germ
  $(X,0)$ is a decomposition of it as a union of semi-algebraic germs of pure
  dimension $4$ called {\bf pieces}:
  \begin{equation}
    \label{eq:thick-thin}
    (X,0)=\bigcup_{i=1}^r(Y_i,0)\cup\bigcup_{j=1}^s(Z_j,0)\,,
  \end{equation}
  such
  that the  $Y_i\setminus\{0\}$ and $Z_j\setminus\{0\}$ are connected and:
  \begin{enumerate}
  \item Each $Y_i$ is thick and each $Z_j$ is thin;
  \item\label{it:1} The $Y_i\setminus\{0\}$ are pairwise disjoint and the
    $Z_j\setminus\{0\}$ are pairwise disjoint;
  \item\label{it:2} If $\epsilon_0$ is chosen small enough such  that
    $S_\epsilon$ is transverse to each of the germs $(Y_i,0)$
    and $(Z_j,0)$ for $\epsilon\le \epsilon_0$, then 
    $X^{(\epsilon)}=\bigcup_{i=1}^rY^{(\epsilon)}_i\cup\bigcup_{j=1}^sZ^{(\epsilon)}_j$
    decomposes the $3$-manifold $X^{(\epsilon)}\subset
    S_\epsilon$ into connected submanifolds with boundary,
    glued along their boundary components.
\end{enumerate}
\end{definition}
\begin{definition}\label{def:minimal}
  A thick-thin decomposition is {\bf minimal} if
  \begin{enumerate}
  \item\label{it:min1} the tangent cone of
  its thin part $\bigcup_{j=1}^sZ_j$ is contained in the
  tangent cone of the thin part 
  of any other thick-thin decomposition and
\item the number $s$ of its thin pieces is minimal among thick-thin
  decompositions satisfying \eqref{it:min1}.
  \end{enumerate}
\end{definition}

The following  theorem expresses the existence and  uniqueness of
a  minimal thick-thin decomposition of a normal complex surface singularity.  
\begin{theorem} \ \cite[Section 8]{BirbrairNeumannPichon2014} \label{th:uniqueness}
Let $(X,0)$ be a normal complex surface germ. Then a minimal thick-thin decomposition of $(X,0)$ exists.  For any two
  minimal thick-thin decompositions of $(X,0)$, there exists $q>1$ and
  a homeomorphism of the germ $(X,0)$ to itself which takes  one
  decomposition to the other and moves each $x\in X$ by a distance at most
  $\lVert x \rVert^q$.
\end{theorem}

\subsection{The thick-thin decomposition in a resolution} \label{sec:thick-thin resolution}

In this section, we describe explicitly the minimal thick-thin decomposition  of a normal complex surface germ $(X,0)\subset (\C^n,0)$ in terms of a suitable resolution of $(X,0)$ as presented in \cite[Section 2]{BirbrairNeumannPichon2014}. The uniqueness of the minimal thick-thin decomposition is proved in \cite[Section 8]{BirbrairNeumannPichon2014}. 

Let $\pi\colon(\widetilde X,E)\to (X,0)$ be the minimal resolution with
the following properties:
\begin{enumerate}
\item\label{r1} It is a good resolution, i.e., the irreducible components of the exceptional divisor are smooth and meet transversely, at most two at
a time.
\item\label{r2} It  factors through the
   blow-up  $e_0 \colon X_0 \to X$ of the origin.  An irreducible component  of the exceptional divisor  $\pi^{-1}(0)$ which projects surjectively on an irreducible component of  $e_0^{-1}(0)$ will be called an {\bf $\cal L$-curve}.
\item\label{r3} No two  $\cal L$-curves intersect.  
\end{enumerate}
 This is achieved by starting with a minimal good resolution, then
blowing up to resolve any base points of a general system  of hyperplane
sections, and finally blowing up any intersection point between  $\cal L$-curves.

\begin{definition} 
Let $\Gamma$ be the resolution graph of the above resolution. A vertex of $\Gamma$ is called a {\bf node} \index{node} if it has valence  $\ge 3$ or represents a curve of genus $>0$ or represents an $\cal L$-curve. If a node represents an $\cal L$-curve it is called an {\it  $\cal L$-node} \index{node!$\cal L$-node}. By the previous paragraph, $\cal L$-nodes cannot be adjacent to each other. Other types of nodes will be introduced in Definitions  \ref{dfn:inner nodes} and \ref{dfn:outer node}. 

A {\bf string}   is a
connected subgraph of $\Gamma$ containing no nodes.
A {\bf bamboo} is a string ending in a vertex of
valence $1$.
\end{definition}

For each irreducible curve $E_\nu$ in $E$, let $N(E_\nu)$ be a small
closed tubular neighborhood of $E_{\nu}$ in $\widetilde X$.  For any subgraph $\Gamma'$ of $\Gamma$
define (see Figure \ref{fig:9.1}):
$$N(\Gamma'):= \bigcup_{\nu\in\Gamma'}N(E_\nu)\quad\text{and}\quad
{\mathscr{N}}(\Gamma'):= 
\overline{N(\Gamma)\setminus \bigcup_{\nu\notin \Gamma'}N(E_\nu)}\,.$$
\begin{figure}[ht] 
\centering
\begin{tikzpicture}
  \node[ ](a)at (0.8,1.3){$\Gamma'$}; \node[ ](b)at
  (-1,0.8){$\Gamma$}; \node[ ](c)at (4.5,-1){$N(\Gamma')$}; \node[
  ](d)at (9,-1){${{\mathscr{N}}}(\Gamma')$};\draw[
  scale=0.8,xshift=-0.5cm,yshift=1cm,thin] (0+2pt,0)--(2.95,0);\draw
  [ scale=0.8,xshift=-0.5cm,yshift=1cm ](0,0)circle(2pt);
 \draw[scale=0.8, xshift=-0.5cm,yshift=1cm,fill] (1,0)circle(2pt);
 \draw[ scale=0.8,xshift=-0.5cm,yshift=1cm, fill] (2,0)circle(2pt);
 \draw[scale=0.8,xshift=-0.5cm,yshift=1cm, thin](3,0)circle(2pt);
 \draw[scale=0.7,xshift=5cm,fill, lightgray](2,2)+(-45:0.2) arc
  (-45:135:0.2)--++(-2,-2) arc (135:315:0.2)--++(2,2);
 \draw[scale=0.7,xshift=7cm, fill, lightgray](1,0)+(-135:0.2) arc
  (-135:45:0.2)--++(-2,2) arc (45:225:0.2)--++(2,-2);
 \draw[scale=0.7,xshift=5cm,very thin](2,2)+(-45:0.2) arc
  (-45:135:0.2)--++(-2,-2) arc (135:315:0.2)--++(2,2);
 \draw[scale=0.7,xshift=7cm,very thin](2,2)+(-45:0.2) arc
  (-45:135:0.2)--++(-2,-2) arc (135:315:0.2)--++(2,2);
 \draw[scale=0.7,xshift=5cm,very thin](1,0)+(-135:0.2) arc
  (-135:45:0.2)--++(-2,2) arc (45:225:0.2)--++(2,-2);
 \draw[scale=0.7,xshift=7cm, very thin](1,0)+(-135:0.2) arc
  (-135:45:0.2)--++(-2,2) arc (45:225:0.2)--++(2,-2);
 \draw[scale=0.7,xshift=5cm](0,0)--(2,2);
 \draw[scale=0.7,xshift=5cm](-1,2)--(1,0);
 \draw[scale=0.7,xshift=5cm](1,2)--(3,0);
 \draw[scale=0.7,xshift=5cm](2,0)--(4,2);
 \draw[scale=0.7,xshift=11cm,fill, lightgray](2,2)+(-45:0.2) arc
  (-45:135:0.2)--++(-2,-2) arc (135:315:0.2)--++(2,2);
 \draw[scale=0.7,xshift=13cm, fill, lightgray](1,0)+(-135:0.2) arc
  (-135:45:0.2)--++(-2,2) arc (45:225:0.2)--++(2,-2);
 \draw[scale=0.7,xshift=11cm,fill,
  white](0.5,0.5)--++(-45:0.2)--++(45:0.2)
  --++(135:0.4)--++(225:0.4)--++(-45:0.4)--++(45:0.2);
 \draw[scale=0.7,xshift=13cm,fill,
  white](0.5,0.5)--++(-45:0.2)--++(45:0.2)
  --++(135:0.4)--++(225:0.4)--++(-45:0.4)--++(45:0.2);
 \draw[scale=0.7,xshift=11cm,very thin](2,2)+(-45:0.2) arc
  (-45:135:0.2)--++(-2,-2) arc (135:315:0.2)--++(2,2);
 \draw[scale=0.7,xshift=13cm,very thin](2,2)+(-45:0.2) arc
  (-45:135:0.2)--++(-2,-2) arc (135:315:0.2)--++(2,2);
 \draw[scale=0.7,xshift=11cm,very thin](1,0)+(-135:0.2) arc
  (-135:45:0.2)--++(-2,2) arc (45:225:0.2)--++(2,-2);
 \draw[scale=0.7,xshift=13cm, very thin](1,0)+(-135:0.2) arc
  (-135:45:0.2)--++(-2,2) arc (45:225:0.2)--++(2,-2);
 \draw[scale=0.7,xshift=11cm](0,0)--(2,2);
 \draw[scale=0.7,xshift=11cm](-1,2)--(1,0);
 \draw[scale=0.7,xshift=11cm](1,2)--(3,0);
 \draw[scale=0.7,xshift=11cm](2,0)--(4,2);
\end{tikzpicture} 
\caption{The sets ${{N}}(\Gamma')$ and  ${\mathscr{N}}(\Gamma')$}\label{fig:9.1}
\end{figure}
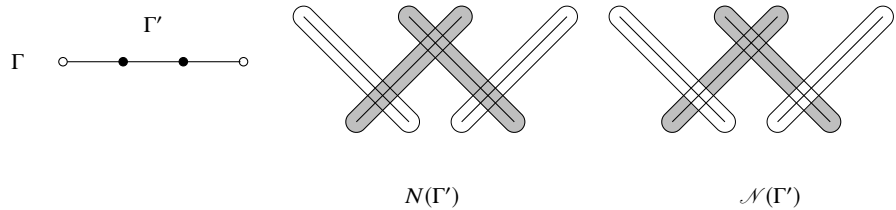

  The subgraphs of $\Gamma$ resulting by removing the $\cal L$-nodes and adjacent
edges from $\Gamma$ are called the {\bf Tyurina components} \index{Tyurina component} of
$\Gamma$ (following \cite[Definition III.3.1]{Spivakovsky1990}). 

 Let $\Gamma_1 ,
  \ldots, \Gamma_s$ denote the  Tyurina components of $\Gamma$ which are not
  bamboos, and by $\Gamma'_1,\ldots, \Gamma'_r$ the maximal connected
  subgraphs in $\Gamma \setminus \bigcup_{j=1}^s \Gamma_j$. Therefore   each $\Gamma'_i$ consists
of an $\cal L$-node  and any attached bamboos and strings. 
 
  Assume that $\epsilon_0$ is sufficiently small such that $\pi^{-1}(X\cap
  B_{\epsilon_0})$ is included in $N(\Gamma)$.
  
  For each each $i=1,\ldots,r$, define
$$Y_i:=\pi(N(\Gamma'_i))\cap  B_{\epsilon_0},$$
and for each $j=1,\ldots,s$, define
$$Z_j:=\pi({\mathscr{N}}(\Gamma_j))\cap  B_{\epsilon_0}.$$
 
 Notice that  the $Y_i$ are in
one-to-one correspondence with the $\cal L$-nodes.

\begin{theorem}  \cite[Section 2, Proposition 5.1, Proposition 6.1]{BirbrairNeumannPichon2014}  \label{th:main in intro}
\begin{enumerate}
\item  \label{thm:thick} For each  $i=1,\ldots,r$, $(Y_i,0)$ is thick;
\item  \label{thm:thin} For each  $j=1,\ldots,s$, $(Z_j,0)$ is thin;
\item  \label{thm:minimal}  The decomposition $(X,0)=\bigcup(
Z_j,0)\cup\bigcup (Y_i,0)$ is a   minimal thick-thin decomposition of $(X,0)$. 
\end{enumerate}
\end{theorem} 

The proof of (\ref{thm:thin}) is easy: 
\begin{proof} 
Choose an embedding  $(X,0) \subset (\C^n,0)$ and let  $e_0 \colon X_0 \to X$ be the blow-up of the origin. If $x \in \C^n\setminus \{0\}$, denote by $L_x$ the class of $x$ in $\mathbb P^{n-1}$, so $L_x$ represents  the line through $0$ and $x$  in $\C^n$.  By definition  $X_0$ is the closure in $ \C^n    \times  {\mathbb P}^{n-1}$ of the set $\{(x,L_x) \colon x \in X \setminus \{0\}  \}  $. 

The semi-algebraic set $Z_j$  is of real dimension $4$. On the other hand, the strict transform of $Z_j$ by $e_0$ meets the exceptional divisor $e_0^{-1}(0)$ at a single point $(x,L_x) $, so  the tangent cone at $0$ to $Z_j$ is the complex line $L_x$. Therefore $(Z_j,0)$ is thin. 
\end{proof}

In the next section,  we present  the first part of the  proof of the thickness of $(Y_i,0)$,   which consists of proving  the following intermediate Lemma:

\begin{lemma} \label{lem:first part} For each $\cal L$-node $\nu$,  the subset  $\pi({{\mathscr{N}}}(\nu))$ of $(X,0)$ is metrically conical.  
\end{lemma}

 The rest of the proof of Point   (\ref{thm:thick}) of Theorem \ref{th:main in intro}  is more delicate. In particular, it  uses the key Polar Wedge Lemma \cite[Proposition 3.4]{BirbrairNeumannPichon2014} which is stated later in the present notes (Proposition \ref{polar wedge lemma}) and a geometric decomposition of $(X,0)$ into standard pieces which is a refinement of the thick-thin decomposition and which leads to the complete classification of the inner Lipschitz geometry of $(X,0)$ presented in Section \ref{sec: inner geometric decomposition}.  We refer to \cite{BirbrairNeumannPichon2014} for the proofs. 

The minimality (\ref{thm:minimal}) is proved in \cite[Section 8]{BirbrairNeumannPichon2014}.

We  now give several explicit examples of thick-thin decompositions.  More examples can be found in \cite[Section 15]{BirbrairNeumannPichon2014}.

\begin{example}
 \index{surface singularity!$E_8$}
Consider the normal surface singularity  $(X,0)\subset (\C^3,0)$ with equation $x^2+y^3+z^5=0$. This is the standard singularity $E_8$ (see \cite{Durfee1979}). Its minimal resolution  has exceptional divisor a tree of eight $\mathbb P^1$ having self intersections $-2$ and it  factors through the blow-up of the point $0$.  The dual graph  $\Gamma$ is represented on Figure \ref{fig:10}.  It can be constructed with Laufer's method (see Appendix \ref{appendix}).  The arrow represents the strict transform of a generic linear form $h= \alpha x+\beta y + \gamma z $ on $(X,0)$.  The vertex adjacent to it is the unique  $\cal L$-node and  $\Gamma$ has two nodes which are circled on the figure.  The thick-thin decomposition of $(X,0)$ has one thick piece $(Y_1,0)$ and one thin piece $(Z_1,0)$. The  subgraph $\Gamma'_1$ of $\Gamma$  such that $Y_1 = \pi(N(\Gamma'_1))$   is in black and the  subgraph   $\Gamma_1$ such that $Z_1 =\pi({\mathscr{N}}(\Gamma_1))  $  is in white. 
 
\begin{figure}[ht] 
\centering
 \begin{tikzpicture}

  \draw[thin ](-2,0)--(4,0);
      \draw[thin ](0,0)--(0,1);
      
  \draw[fill=white   ] (-2,0)circle(2pt);
  \draw[fill=white  ] (-1,0)circle(2pt);
   \draw[fill=white] (0,0)circle(4pt);
     \draw[fill =white ] (0,0)circle(2pt);      
          \draw[fill=white   ] (0,1)circle(2pt);
          
\draw[thin,>-stealth,->](4,0)--+(0.7,0.7);

  \draw[fill   ] (1,0)circle(2pt);
    \draw[fill    ] (2,0)circle(2pt);
  \draw[fill=white] (4,0)circle(4pt);
  \draw[fill   ] (4,0)circle(2pt);
   \draw[fill] (3,0)circle(2pt);
\node(a)at(5,0){$\cal L$-node};

\node(a)at(-2,-0.4){$-2$};
\node(a)at(-1,-0.4){$-2$};
\node(a)at(-0,-0.4){$-2$};
\node(a)at(1,-0.4){$-2$};
\node(a)at(2,-0.4){$-2$};
\node(a)at(3,-0.4){$-2$};
\node(a)at(4,-0.4){$-2$};
\node(a)at(-0.4,1){$-2$};

  \end{tikzpicture} 
 \caption{The thick-thin decomposition of the singularity $E_8$}\label{fig:10}
\end{figure}
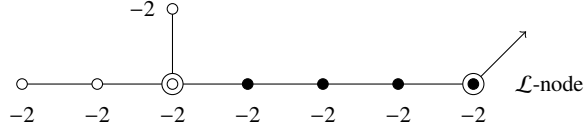
\end{example}

\begin{example}
Consider the normal surface singularity  $(X,0)\subset (\C^3,0)$ with equation $x^2+zy^2+z^3 =0$. This is the standard singularity $D_4$. Its minimal resolution  has exceptional divisor a tree of  four $\mathbb P^1$'s having self intersections $-2$ and it  factors through the blow-up of the point $0$.  The dual graph  $\Gamma$ is represented on Figure \ref{fig:11}.  It  has one $\cal L$-node, which is the central vertex circled on the figure and no other node.  Therefore, the thick-thin decomposition of $(X,0)$ has empty  thin part  and $(X,0)$ is metrically conical. The  subgraph of $\Gamma$ corresponding to the thick part is the whole $\Gamma$.

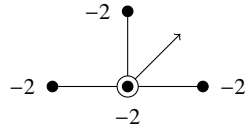
\begin{figure}[ht] 
\centering
\begin{tikzpicture}
    
         \draw[thin ](0,0)--(0,1);
        \draw[thin ](-1,0)--(1,0);
    

\draw[thin,>-stealth,->](0,0)--+(0.7,0.7);
  \draw[fill   ] (0,1)circle(2pt);
  \draw[fill   ] (-1,0)circle(2pt);
    \draw[fill   ] (1,0)circle(2pt);
     
      \draw[fill=white] (0,0)circle(4pt);
 \draw[fill   ] (0,0)circle(2pt);

 \node(a)at(-1.4,0){$-2$};
  \node(a)at(1.4,0){$-2$};
    \node(a)at(-0.4,1){$-2$};
      \node(a)at(0,-0.4){$-2$};
 

  \end{tikzpicture} 
 \caption{ The thick-thin decomposition of the singularity $D_4$}\label{fig:11}\label{fig:11}
\end{figure}
 \end{example}
 
 \begin{example} Consider the family of surface  singularities in $(X_t,0)\subset (\C^3,0)$ with equations $x^5+z^{15}+y^7z+txy^6=0$
  depending on the parameter $t \in \C$. This is a $\mu$-constant family introduced by Brian\c{c}on and Speder in \cite{BrianSpeder1975}. The  thick-thin decomposition  changes radically when $t$ becomes $0$. Indeed, the minimal resolution graph of every $(X_t,0)$ is the first graph on Figure \ref{fig:12} while the  two other resolution graphs   describe the thick-thin decompositions for $t=0$ and for $t \neq 0$. For $t \neq 0$ it has three thick components and a single thin one. For $t =0$, it has one   component of each type.  We refer to \cite[Example 15.7]{BirbrairNeumannPichon2014} for further details.

  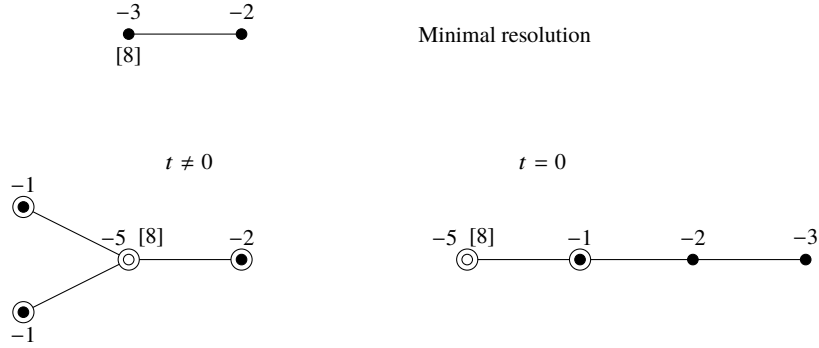
\begin{figure}[ht] 
\centering
 \begin{tikzpicture}
     \draw[thin ](0,0)--(1.5,0);
   
        \draw[fill ] (0,0)circle(2pt);
   \draw[fill    ] (1.5,0)circle(2pt);

\node(a)at(0,0.3){ $-3$};
\node(a)at(0,-0.3){ [8]};
\node(a)at(1.5,0.3){ $-2$};

\node(a)at(5,0){ Minimal resolution};

  
 \begin{scope}[yshift=-3cm]
   \draw[thin ](0,0)--+(-1.4,-0.7);
      \draw[thin ](0,0)--(1.5,0);
       \draw[thin ](0,0)--+(-1.4,0.7);
      
   \draw[fill=white] (1.5,0)circle(4pt);
      \draw[fill=white] (-1.4,0.7)circle(4pt);  
      \draw[fill=white] (-1.4,-0.7)circle(4pt);
    
  \draw[fill=white] (0,0)circle(4pt);
 \draw[fill=white] (0,0)circle(2pt);
   \draw[fill   ] (1.5,0)circle(2pt);
   
     \draw[fill    ] (-1.4,0.7)circle(2pt);
       \draw[fill   ] (-1.4,-0.7)circle(2pt);
   
\node(a)at(-0.2,0.3){ $-5$};
\node(a)at(0.3,0.3){ [8]};
    
\node(a)at(1.5,0.3){ $-2$};

\node(a)at(-1.4,1){ $-1$};
\node(a)at(-1.4,-1){ $-1$};

 \node(a)at(0.8,1.3){ $t\neq0$};

\begin{scope}[xshift=4.5cm]
 \node(a)at(1.0,1.3){ $t=0$};
  \draw[thin ](0,0)--(4.5,0);
  
    \draw[fill=white] (0,0)circle(4pt);
   \draw[fill=white   ] (0,0)circle(2pt);
   
     \draw[fill=white] (1.5,0)circle(4pt);
    \draw[fill   ] (1.5,0)circle(2pt);
    
     \draw[fill   ] (3,0)circle(2pt);
      \draw[fill  ] (4.5,0)circle(2pt);
      
      \node(a)at(-0.3,0.3){ $-5$};
         \node(a)at(0.2,0.3){ [8]};

        \node(a)at(1.5,0.3){ $-1$};
      
\node(a)at(3,0.3){ $-2$};

\node(a)at(04.5,0.3){ $-3$};

\end{scope}
\end{scope}
 \end{tikzpicture} 
  \caption{The two  thick-thin decompositions in the Brian\c{c}on-Speder family $x^5+z^{15}+y^7z+txy^6=0$  }\label{fig:12}
\end{figure}
\end{example}

\begin{exercise} \label{ex:ADE} Describe the thick-thin decomposition  of every  ADE surface singularity and show that among them, only $A_1$ and $D_4$ are metrically conical (Answer: \cite[Example 15.4]{BirbrairNeumannPichon2014}). The equations are: 
\begin{itemize}
\item $A_n \colon x^2+y^2+z^{k+1}=0, \, \ k \geq 1 $
\item  $D_n  \colon x^2+zy^2+z^{k-1}=0, \, \ k \geq 4$
\item $E_6 \colon  x^2+y^3+z^4=0$
\item $E_7 \colon  x^2+y^3+yz^3=0$
\item $E_8 \colon  x^2+y^3+z^5=0$ 
\end{itemize}

\end{exercise}

\subsection{Generic projection and inner metric: a key lemma}  
 
In this section, we state and prove  Lemma  \ref{lem:local bil bound1} which is one of the key results which will lead to the complete classification of the inner metric of $(X,0)$.  We give two applications. The first one is the proof of Lemma \ref{lem:first part}. The second  describes the inner contact between complex curves on a complex surface.    

 We first need to introduce the polar curves of generic projections and  the Nash modification of $(X,0)$. 
  
 \subsubsection{Polar curves and generic projections}
 
 Let $(X,0)\subset (\C^n,0)$ be a normal surface singularity. We restrict ourselves to those $\cal D$ in ${\mathbf G}(n-2,\C^n)$ such that the restriction
$\ell_{\cal D}{\mid_{(X,0)}}\colon(X,0)\to(\C^2,0)$ is finite.
The {\bf polar curve} \index{polar curve}
$\Pi_{\cal D}$ of $(X,0)$ for the direction $\cal D$ is the closure in
$(X,0)$ of the critical locus of the restriction of $\ell_{\cal D} $
to $X \setminus \{0\}$. The {\bf discriminant curve} \index{discriminant curve} $\Delta_{\cal D}
\subset (\C^2,0)$ is the image $\ell_{\cal D}(\Pi_{\cal D})$ of the polar
curve $\Pi_{\cal D}$.

\begin{proposition}[{\cite[Lemme-cl\'e V 1.2.2]{Teissier1982}}]\label{prop:generic} An open
dense subset $\Omega \subset {\mathbf G}(n-2,\C^n)$ exists 
such that: 
\begin{enumerate}
\item \label{cond:generic1} The family of  curve germs  $(\Pi_{\cal D})_{{\cal D} \in \Omega}$
is equisingular in terms of strong simultaneous resolution;
\item \label{cond:generic2} The curves  $\ell_{\cal D}(\Pi_{\cal D'})$, $({\cal D}, {\cal D}') \in \Omega \times \Omega$ form an equisingular family of reduced plane curves;
\item \label{cond:generic3}  For each $\cal D$, the projection $\ell_{\cal D}$ is generic for its polar curve $\Pi_{\cal D}$ (Definition \ref{def:generic projection curve}). 
\end{enumerate}
  \end{proposition}

\begin{definition}  \label{def:generic linear projection} 
The projection $\ell_{\cal D} \colon \C^n \to \C^2$
  is {\bf generic} \index{generic!projection} for $(X,0)$ if ${\cal D} \in \Omega$.
\end{definition}

\subsubsection{Nash modification}
 
 \begin{definition} \label{def:Nash modification}
Let $\lambda \colon X\setminus\{0\} \to {\mathbf G}(2,\C^n)$ be the
  map which sends $x \in X\setminus\{0\}$ to the tangent plane
  $T_xX$. The closure $X_{\nu}$ of the graph of $\lambda$ in $X
  \times {\mathbf G}(2,\C^n)$ is a reduced analytic surface. By
  definition, the {\bf Nash modification} \index{Nash modification} of $(X,0)$ is the 
  morphism $\nu  \colon X_{\nu} \to X$ induced by projection on the first factor. 
\end{definition}
 
 \begin{lemma}[{\cite[Part III,  Theorem
    1.2]{Spivakovsky1990}}]\label{le:nash}
A morphism $\pi \colon Y \to X$ factors through Nash modification if and
  only if it has no base points for the family of polar curves of generic projections, i.e., there is no point $p \in \pi^{-1}(0)$ such that for every  ${\cal D} \in \Omega$, the strict transform of $\Pi_{\cal D}$  by $\pi$   passes through $p$.
  \end{lemma}

\begin{definition} \label{def:Gauss} Let $(X,0) \subset (\C^n,0)$ be a
  complex surface germ and let ${\nu } \colon X_{\nu} \to X$ be the Nash
  modification of $X$. The
  {\bf Gauss map} \index{Gauss map} $\widetilde{\lambda} \colon X_{\nu} \to {\mathbf G}(2,\C^n) $ is  the restriction to $X_{\nu}$ of the projection of $X  \times {\mathbf G}(2,\C^n)$ on the second factor. 
\end{definition}

 Let  $\ell \colon\C^n\to \C^2$ be a linear projection such that the restriction $\ell |_X \colon (X,0) \to (\C^2,0)$ is generic. 
Let $\Pi$ and $\Delta$ be the polar and discriminant curves of $\ell|_X$.
 
 \begin{definition}   \label{local bilipschitz constant}  The {\bf local bilipschitz constant of $\ell|_X$} 
 is the map
$K\colon X\setminus \{0\}\to \R\cup\{\infty\}$ defined as follows. It is
infinite on the polar curve $\Pi$ and at a point $p\in X\setminus \Pi$ it is
the reciprocal of the shortest length among images of unit vectors in
$T_{p} X$ under the projection $\ell \mid_{T_{p} X} \colon T_{p} X\to \C^2$.
\end{definition}
 
    Let  $\Pi^*$ denote  the
  strict transform of the polar curve $\Pi$ by the Nash modification $\nu $. Set $B_{\epsilon}  = \{ x \in \C^n \colon \lVert x \rVert _{\C^n} \leq \epsilon\}$. 
 
 \begin{lemma} \label{lem:local bil bound1} Given any neighbourhood $U$ of\/ ${\Pi}^*\cap
  \nu^{-1}(B_\epsilon\cap X)$ in ${X_{\nu}}\cap
  \nu ^{-1}(B_\epsilon\cap X)$, the local bilipschitz constant $K$ of $\ell \mid_X$
  is bounded on $B_\epsilon \cap (X\setminus \nu (U))$.
\end{lemma}

\begin{proof}  
  Let $\kappa \colon
  {\mathbf G}(2,\C^n) \to \R \cup \{\infty\}$ be the map sending $H \in
  {\mathbf G}(2,\C^n)$ to the bilipschitz constant of the restriction
  $\ell|_{H}\colon H \to \C^2$.  The map $\kappa \circ \widetilde{\lambda}$ coincides with $K \circ {\nu }$ on $X_{\nu}  \setminus
  {\nu }^{-1}(0)$ and takes finite values outside ${\Pi}^*$. The map
  $\kappa  \circ \widetilde{\lambda}$ is continuous and therefore bounded on the
  compact set ${\nu }^{-1}(B_{\epsilon} \cap X) \setminus U$.
  \end{proof}

We will use ``small" special  versions of $U$ called polar wedges, defined as follows.

\begin{definition} \label{def:polar wedge} \index{polar wedge}  Let  $\Pi_0$ be a component of $\Pi$ and let $(u,v)$ be local coordinates in $X_{\nu}$ centered at $p =  \Pi_0^*\cap {\nu }^{-1}(0)$ such that $v=0$ is the local equation of ${\nu }^{-1}(0)$. For $\alpha >0$, consider the polydisc   $U_{\Pi_0}(\alpha)=\{(u,v) \in \C^2 \colon |u|\leq \alpha\}$.
 For small $\alpha$, the set $W_{\Pi_0}  = \nu(U_{\Pi_0}(\alpha))$  is called a {\bf polar wedge} around $\Pi_0$ and the union $W = \bigcup_{\Pi_0 \subset \Pi}  W_{\Pi_0}$ a {\bf polar wedge} around $\Pi$. 
\end{definition}

\subsubsection{Application 1}
\begin{proof}[of Lemma \ref{lem:first part}] We want to prove that for every $\cal L$-node $\nu$,  the germ   $\pi({{\mathscr{N}}}(\nu))$ is metrically conical. 

Consider a polar wedge  $W$ around $\Pi$.   A direct  consequence of Lemma \ref{lem:local bil bound1} is that the restriction of $\ell$ to  $X \setminus W$  is a local bilipschitz homeomorphism for the inner metric. Therefore, for any  metrically conical germ $C$ in $(\C^2,0)$, the intersection of the   lifting  $\ell^{-1}(C)$ with  $\overline{X \setminus W}$  will be a metrically  conical germ.

For each $j=1,\ldots,s$, let $L_j \subset \C^n$ be the complex tangent line of the thin  germ $(Z_j,0)$ and let $L'_j \subset \C^2$ be image of $L_j$ by the generic linear form $\ell \colon  \C^n \to \C^2$. Assume $L'_j$ has equation $y=a_j x$. For a  real number $\alpha >0$, we consider the conical subset  $V_{\alpha}  \subset \C^2$  defined  as the union of the complex  lines $y=\eta x$ such that $|\eta -a_j| \geq \alpha$, so $V_{\alpha}$ is the closure of a set obtained by removing conical neighbourhoods of the lines $L'_j$.
Applying the above result, we obtain that for all $\alpha>0$,  the  intersection of the  lifting  $\ell^{-1}(V_{\alpha})$ with  $\overline{X \setminus W}$  gives  a metrically  conical germ at $0$. Since there exist two  real numbers $\alpha_1, \alpha_2 $ with $0<\alpha_1 < \alpha_2$ such that inside a  small ball $B_{\epsilon}$, we have  $\ell^{-1}(V_{\alpha_1}) \subset  \pi({{\mathscr{N}}}(\nu)) \subset   \ell^{-1}(V_{\alpha_2})$, then the germ 
$\pi({{\mathscr{N}}}(\nu)) \cap (\overline{X \setminus W})$ at $0$ is also metrically conical.

 If the strict transform of $\Pi$ by $\pi$ does not intersect the $\cal L$-curve $E_{\nu}$, then the intersection $\pi({{\mathscr{N}}}(\nu)) \cap (\overline{X \setminus W})$ is the whole $\pi({{\mathscr{N}}}(\nu))$. Therefore $\pi({{\mathscr{N}}}(\nu))$ is  metrically conical. 
 
 If the  strict transform of $\Pi$ by $\pi$ intersects the $\cal L$-curve $E_{\nu}$, then we have to use a second generic projection $\ell' \colon (X,0) \to (\C^2,0)$ such that the strict transform of the polar curve $\Pi'$ of $\ell'$ by $\pi$ does not intersect  $U$, and we prove that  $ \overline{\pi({{\mathscr{N}}}(\nu)) \cap W}$ is metrically conical using the above argument. Therefore $\pi({{\mathscr{N}}}(\nu))$ is metrically conical as the union of two metrically conical sets. 
\end{proof}

\subsubsection{Application 2.} 
 
 Let $(X,0)$ be a normal complex surface singularity. 
\begin{definition} \label{def:inner rate} Let $S_{\epsilon} = \{ x \in \C^n \colon \lVert x \rVert_{\C^n} = \epsilon\}$. Let $(\gamma,0)$ and $(\gamma',0)$ be two  distinct irreducible germs   of complex curves inside $(X,0)$. Let  $q_{inn}=q_{inn}(\gamma, \gamma')$ be the   rational number $\geq 1$ defined by 
$$ d_i(\gamma \cap S_{\epsilon}, \gamma' \cap S_{\epsilon}) =  \Theta(\epsilon^{q_{inn}}),$$
where $d_{i}$ means inner distance in $(X,0)$ as before. 

 We call  $q _{inn}(\gamma, \gamma')$ the  {\it  the inner  contact exponent}  or  {\it  inner  contact order} \index{contact order!inner} between $\gamma$ and $\gamma'$.
\end{definition}

 The proof of the existence and rationality of the inner contact $q_{inn}$ needs deep arguments of \cite{KurdykaOrro1997}. We refer to this paper for details. 
 
 \begin{remark}
One can also define the outer contact  exponent $q_{out}$   \index{contact order!outer} between two curves by using the outer metric instead of the inner one. In that case, the existence and rationality of $q_{out}$ come easily from the fact that the outer distance $d_o$ is a semialgebraic function.   (While  the inner distance $d_i$ is not semi-algebraic.)
\end{remark}
   
   \begin{definition} \label{dfn:curvette} Let $\pi \colon Z \to X$ be a resolution of $X$ and let  $E$ be an irreducible component of the exceptional divisor $\pi^{-1}(0)$.  A {\it curvette}  
   of $E$ is a smooth curve $\delta \subset Z$ which is transversal to $E$ at a smooth point of the exceptional divisor $\pi^{-1}(0)$.   
   \end{definition}
   
\begin{lemma} \label{rk:inner rate}  \cite[Lemma 15.1]{NeumannPedersenPichon2019-1} Let $\pi \colon Z \to X$ be a resolution of $(X,0)$ and let  $E$ be an irreducible component of the exceptional divisor $\pi^{-1}(0)$. Let $(\gamma,0)$ and $(\gamma',0)$ be  the $\pi$-images of two  curvettes of $E$ meeting $E$ at two distinct points.  Then  $q_{inn}(\gamma, \gamma')$  is independent of the choice of $\gamma$ and $\gamma'$.
  \end{lemma}
  
   \begin{definition}  \label{def:inner rate2} We set $q_E = q_{inn}(\gamma, \gamma')$ and we call $q_E$ the {\it  inner rate} \index{inner rate} of $E$. 
 \end{definition}

\begin{remark}  \label{rk:inner rates} When $X=\C^2$, inner and outer metrics coincide and  the result  is well known and comes from classical plane curve theory: in that case, $q_{inn}(\gamma, \gamma')$ is the coincidence exponent between Puiseux expansions of the curves $\gamma$  and $\gamma'$ { (see for example \cite[page 401]{GarciaBarrosoTeissier1999}).}  The inner rate at each vertex of a  sequence of blow-ups can be computed by using the  classical dictionary between characteristic exponents of an irreducible curve and its resolution graph. We refer to  \cite[page 148]{EisenbudNeumann1985} or  \cite[Section 8.3]{Wall2004} for details. {   As  a consequence of this,  the inner rates along any path from the root vertex to a leaf of $T$ form a strictly increasing sequence.} 
 \end{remark}
 
 \begin{example} \label{example:inner rates 1}   The  dual tree $T_0$ of the minimal  resolution $\sigma_0 \colon Y_0 \to \C^2$ of the curve $ \gamma$ with Puiseux expansion $y = z^{5/3}$ is obtained (Figure \ref{fig:13}) by computing the continued fraction development  
 $$\frac{5}{3} =  1+ \frac{1}{1+\frac{1}{2}} =: [1,1,2]^+.$$ 
 Since $1+1+2 = 4$, $\sigma_0$ consists of four successive  blow-ups  of points starting with the blow-up of the origin of $\C^2$ which  
correspond to  the four vertices of $T_0$. The irreducible curves $E_{1}, \ldots, E_4$ are  labelled in their order of appearance and the vertices of $T_0$ are  also weighted by their self-intersections $E_i^2$. 

 \begin{figure}[ht] 
\centering
 \begin{tikzpicture}

   \draw[thin ](-2,0)--(1,0);
  \draw[fill ] (-2,0)circle(2pt);
   \draw[fill ] (-1,0)circle(2pt);
    \draw[fill ] (0,0)circle(2pt);
     \draw[fill ] (1,0)circle(2pt);
     
\draw[thin,>-stealth,->](-1,0)--+(-0.7,0.7);

\node(a)at(-2,0.4){\scriptsize{$E_2$ }};
\node(a)at(0,0.4){ \scriptsize{$E_3$}};
\node(a)at(-1,0.4){\scriptsize{$E_4$}};
\node(a)at(1,0.4){ \scriptsize{$E_1$}};
\node(a)at(2.2,0){root vertex};

\node(a)at(-2,-0.4){$-3$};
\node(a)at(-1,-0.4){$-1$};
\node(a)at(0,-0.4){$-2$};
\node(a)at(1,-0.4){$-3$};

\end{tikzpicture} 
 \caption{The resolution tree $T_0$  of the curve $x^5+z^{15}+y^7z+txy^6=0$ }\label{fig:13}
\end{figure}
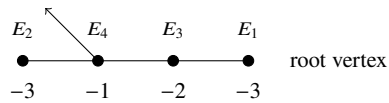
 
The inner rates   are computed by using the   approximation numbers associated with the sequence  $[1,1,2]^+$: $q_{v_1}=[1]^+ = 1$, $q_{v_2}=[1,1]^+ = 1+\frac{1}{1} = 2$, $q_{v_3}=[1,1,1]^+  = \frac{3}{2}$ and $q_{v_3}=[1,1,2]^+   = \frac{5}{3}$

 This gives the   tree  $T_0$  of Figure \ref{fig:14} where each vertex is weighted by the self intersection of the corresponding exceptional  curve $E_i$ and with the inner rate  $q_{E_i}$ (in bold). 
 
 Let us blow up every intersection point between irreducible components of  the total transform $\sigma_0^{-1}(\gamma)$. The resulting  tree $T$ is that used to compute  the dual resolution graph of $E_8 \colon x^2+y^3-z^5=0 $ by Laufer's method (Appendix \ref{appendix}). Again, the inner rates are in bold. Their computation is left to the reader as an  exercise.

 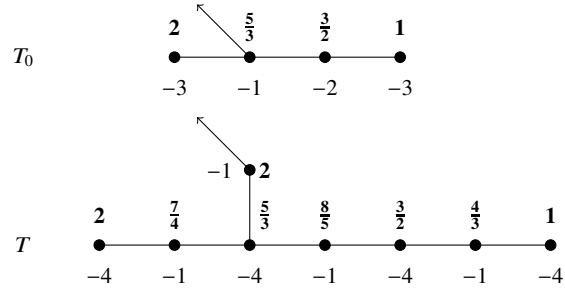
\begin{figure}[ht] 
\centering
 \begin{tikzpicture}
 
   \node(a)at(-4,0){$T_0$};
     \node(a)at(-4,-2.5){$T$};

   \draw[thin ](-2,0)--(1,0);
  \draw[fill ] (-2,0)circle(2pt);
   \draw[fill ] (-1,0)circle(2pt);
    \draw[fill ] (0,0)circle(2pt);
     \draw[fill ] (1,0)circle(2pt);
     
\draw[thin,>-stealth,->](-1,0)--+(-0.7,0.7);

\node(a)at(-2,0.4){ $\mathbf 2$};
\node(a)at(0,0.4){ $ \frac{\mathbf 3}{\mathbf 2}$};
\node(a)at(-1,0.4){$   \frac{\mathbf 5}{\mathbf 3}$};
\node(a)at(1,0.4){ $\mathbf 1$};

\node(a)at(-2,-0.4){$-3$};
\node(a)at(-1,-0.4){$-1$};
\node(a)at(0,-0.4){$-2$};
\node(a)at(1,-0.4){$-3$};

\begin{scope}[xshift=-1cm,yshift=-2.5cm]
  \draw[thin ](-2,0)--(4,0);
      \draw[thin ](0,0)--(0,1);
  \draw[fill ] (-2,0)circle(2pt);
   \draw[fill ] (-1,0)circle(2pt);
    \draw[fill ] (0,0)circle(2pt);
     \draw[fill ] (1,0)circle(2pt);
        \draw[fill ] (2,0)circle(2pt);

   \draw[fill ] (3,0)circle(2pt);

   \draw[fill ] (4,0)circle(2pt);

      \draw[fill ] (0,1)circle(2pt);
\draw[thin,>-stealth,->](0,1)--+(-0.7,0.7);

\node(a)at(4,0.4){ $\mathbf 1$};
\node(a)at(3,0.4){ $\frac{\mathbf 4}{\mathbf 3}$};
\node(a)at(2,0.4){ $ \frac{\mathbf 3}{\mathbf 2}$};
\node(a)at(1,0.4){ $ \frac{\mathbf 8}{\mathbf 5}$};
\node(a)at(0.2,0.4){$   \frac{\mathbf 5}{\mathbf 3}$};
\node(a)at(-1,0.4){$   \frac{\mathbf 7}{\mathbf 4}$};
\node(a)at(-2,0.4){ $\mathbf 2$};
\node(a)at(0.2,1){ $\mathbf 2$};
 
\node(a)at(-2,-0.4){$-4$};
\node(a)at(-1,-0.4){$-1$};
\node(a)at(0,-0.4){$-4$};
\node(a)at(1,-0.4){$-1$};
\node(a)at(2,-0.4){$-4$};
\node(a)at(3,-0.4){$-1$};
\node(a)at(4,-0.4){$-4$};
\node(a)at(-0.4,1){$-1$};

\end{scope}
\end{tikzpicture} 
  \caption{The inner rates in resolutions of the curve $y=x^{5/3}$}\label{fig:14}
\end{figure}

 \end{example}
 
 \begin{proof}[of Lemma \ref{rk:inner rate}] 
  Consider  a generic projection $\ell \colon (X,0) \to (\C^2,0)$  which is also generic for the  curve  germ $(\gamma   \cup   \gamma',0)$ (Definition \ref{def:generic projection curve}). Then consider the minimal  sequence of blow-ups $\sigma \colon Y \to \C^2$  such that the strict transforms  $\ell(\gamma)^*$  and $\ell(\gamma')^*$ by $\sigma$ do not intersect. Then  $\ell(\gamma)^*$  and $\ell(\gamma')^*$  are two curvettes  of the  last exceptional curve  $C$ created by $\sigma$ and we then  have $q_{inn}(\ell(\gamma), \ell(\gamma')) = q_C $. Moreover, an easy argument using Hirzebruch-Jung resolution of surfaces (see \cite{PopescuPampu2011} for an introduction to this resolution method)  shows that $\sigma$ does not depend on the choice of the curvettes $\gamma^*$ and ${\gamma'}^*$ of $E$.  Now, since $\ell$ is generic for  the curve $\gamma \cup \gamma'$, the strict transform of the polar curve $\Pi$ of $\ell$ by $\pi$ does not intersect the strict transform of $\gamma \cup \gamma'$, and then, $\gamma^* \cup {\gamma'}^*$ is outside any sufficiently small  polar wedge of $\ell$ around   $\Pi$. Therefore, by Lemma   \ref{lem:local bil bound1}, we obtain $q_{inn}(\gamma, \gamma') = q_{inn}(\ell(\gamma), \ell(\gamma')) = q_C $  
\end{proof}

\begin{example} \label{example:inner rates 2} The proof of Lemma \ref{rk:inner rate} shows   that the  inner rates $q_E$ can be computed by using inner rates in $\C^2$ through a generic  projection $\ell \colon (X,0) \to (\C^2,0)$. Applying this, Figure \ref{fig:15} shows the inner rate at each vertex of the minimal resolution graph of the surface singularity $E_8  \colon x^2+y^3+z^5=0$.     \index{surface singularity!$E_8$}  They are obtained by lifting the inner rates of the graph $T$ of Example \ref{example:inner rates 1}. 

  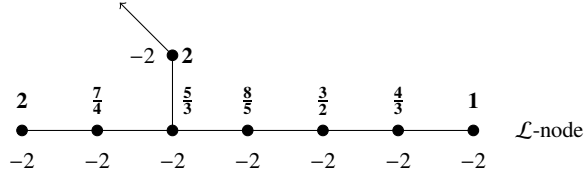
\begin{figure}[ht] 
\centering
 \begin{tikzpicture}

   \draw[thin ](-2,0)--(4,0);
      \draw[thin ](0,0)--(0,1);
  \draw[fill ] (-2,0)circle(2pt);
   \draw[fill ] (-1,0)circle(2pt);
    \draw[fill ] (0,0)circle(2pt);
     \draw[fill ] (1,0)circle(2pt);
        \draw[fill ] (2,0)circle(2pt);

   \draw[fill ] (3,0)circle(2pt);

   \draw[fill ] (4,0)circle(2pt);

      \draw[fill ] (0,1)circle(2pt);
\draw[thin,>-stealth,->](0,1)--+(-0.7,0.7);

\node(a)at(4,0.4){ $\mathbf 1$};
\node(a)at(3,0.4){ $\frac{\mathbf 4}{\mathbf 3}$};
\node(a)at(2,0.4){ $ \frac{\mathbf 3}{\mathbf 2}$};
\node(a)at(1,0.4){ $ \frac{\mathbf 8}{\mathbf 5}$};
\node(a)at(0.2,0.4){$   \frac{\mathbf 5}{\mathbf 3}$};
\node(a)at(-1,0.4){$   \frac{\mathbf 7}{\mathbf 4}$};
\node(a)at(-2,0.4){ $\mathbf 2$};
\node(a)at(0.2,1){ $\mathbf 2$};
\node(a)at(5,0){$\cal L$-node};

\node(a)at(-2,-0.4){$-2$};
\node(a)at(-1,-0.4){$-2$};
\node(a)at(0,-0.4){$-2$};
\node(a)at(1,-0.4){$-2$};
\node(a)at(2,-0.4){$-2$};
\node(a)at(3,-0.4){$-2$};
\node(a)at(4,-0.4){$-2$};
\node(a)at(-0.4,1){$-2$};
\end{tikzpicture} 
  \caption{ The inner rates for the singularity $E_8$}\label{fig:15}
\end{figure}
\end{example}

\subsection{Fast loops in the thin pieces} \label{sec:fast loops2}

Consider a normal surface germ $(X,0) \subset (\C^n,0)$. 
We choose  coordinates $(z_1\dots,z_n)$ in $\C^n$ so that $z_1$ and $z_2$ are generic linear forms and
$\ell:=(z_1,z_2)\colon X\to \C^2$ is a generic linear projection. The family of Milnor balls we use in the sequel consists of standard
``Milnor tubes'' associated with the Milnor-L\^e fibration for the map
$\zeta:=z_1|_X\colon X\to \C$
 Namely, for some sufficiently small
$\epsilon_0$ and some $R>0$ we define for $\epsilon\le\epsilon_0$:
$$B_\epsilon:=\{(z_1,\dots,z_n) \in \C^n \colon |z_1|\le \epsilon,\lVert (z_1,\dots,z_n) \rVert    \le
R\epsilon\}\quad\text{and}\quad S_\epsilon=\partial B_\epsilon.$$

By \cite[Proposition 4.1]{BirbrairNeumannPichon2014}, on can choose  
 $\epsilon_0$ and $R$  so that for $\epsilon\le \epsilon_0$:
\begin{enumerate}
\item\label{it:mb1} $\zeta^{-1}(t)$ intersects the round sphere 
$$S_{R\epsilon}^{2n-1} = \{ (z_1,\dots,z_n) \in \C^n \colon \lVert (z_1,\dots,z_n) \rVert = R \epsilon\}$$ transversely for $|t|\le \epsilon$;
\item\label{it:mb2} the polar curve of the projection
$\ell=(z_1,z_2)$ meets $S_\epsilon$ in the part
$|z_1|=\epsilon$.
\end{enumerate}

 If $(A,0)$ is a semialgebraic germ, we denote by $A^{(\epsilon)} = S_{\epsilon} \cap X$ its link with respect to the Milnor ball $B_{\epsilon}$.

\begin{theorem}\cite[Theorem 1.7]{BirbrairNeumannPichon2014} \label{thm:fastloops} Consider the 
    minimal thick-thin decomposition 
    $$(X,0) = \bigcup_{i=1}^r (Y_i,0) \cup \bigcup_{j=1}^s (Z_j,0)$$ of $(X,0)$.  For $0<\epsilon\le \epsilon_0$ and  for each $j=1,\ldots,s$,  let  $\zeta_j^{(\epsilon)}\colon
    Z^{(\epsilon)}_j\to S^1$ be the restriction to  $Z_j^{(\epsilon)}$ of the generic linear form  $h  = z_1$. Then there exists  $q_j>1$ such that the fibers $\zeta_j^{-1}(t)$ have diameter   $\Theta(\epsilon^{q_j})$.   
\end{theorem}
 
 \begin{proof}[Sketch of proof of Theorem \ref{thm:fastloops}]
 The proof of  Theorem  \ref{thm:fastloops} is based on two keypoints:  Lemma \ref{lem:local bil bound1}, which implies that  $\ell$ is an inner Lipschitz homeomorphism outside a polar wedge $W$, and the so called Polar Wedge Lemma \cite[Proposition 3.4]{BirbrairNeumannPichon2014}  \index{polar wedge!Lemma} which describes the geometry of a polar wedge.  The idea is to use a generic linear projection $\ell =(z_1,z_2) \colon (X,0) \to (\C^2,0)$ and to describe $(Z_j,0)$ as a component of the lifting by $\ell$ of some semi-algebraic germ $(V_j,0)$ in $\C^2$ which has the properties  described in the Theorem, i.e., for small  $\epsilon >0$,  $V_j^{(\epsilon)}$ fibers over $S^1$ with fibers having diameter $\Theta(\epsilon^{q_j})$ for some $q_j>1$.

Consider   a sequence $\sigma \colon Y  \to \C^2$ of blow-ups of points which resolves the base points of the  family of projected polar curves $\ell(\Pi_{\cal D})_{{\cal D} \in \Omega}$ and let $T$ be its dual tree. Notice that the strict transforms of  the curves $\ell(\Pi_{\cal D}), {\cal D} \in \Omega$ form an equisingular family of complex curves, but that these curves are not necessarily smooth, i.e., $\sigma$ is not, in general, a resolution of   $\ell(\Pi_{\cal D})$. 

Denote by $v_1$ the root vertex of $T$, i.e., the vertex corresponding to the exceptional curve created by the  first blow-up  and by $T_0$ the subtree of $T$ consisting of $v_1$ union any adjacent string or bamboo. Then $Z_j$ is a component of $\ell^{-1}(V_j)$ where $V_j = \sigma({\mathscr{N}}(T_j))$ and where $T_j$ is a component of $T \setminus T_0$. Let $v_j$ be the vertex of $T_j$ adjacent to $T_0$. By classical curve theory, $V_j$  is a set of the form 
     $V_j = \{ z_2 = \lambda z_1^{q_j},a \leq  |\lambda| \leq b\}, $ where $q_j$ is the inner rate of the exceptional  curve   represented by the  vertex $v_j$. In particular, the $3$-manifold $V_j^{(\epsilon)} = V_j \cap \{|z_1 | = \epsilon\}$ is fibered over the circle $S^1_{\epsilon}$  by the projection $z_1 \colon V_j^{(\epsilon)}  \to S^1_{\epsilon}$ and the fibers have diameter $\Theta(\epsilon^{q_j})$.

Let $W$ be a polar wedge around $\Pi$. By Lemma \ref{lem:local bil bound1}, we know that  $\ell$ is a  locally  inner bilipschitz homeomorphism  outside $W$. Therefore, the fibers of the restriction  $\zeta_j^{(\epsilon)}\colon  Z^{(\epsilon)}_j \setminus W^{(\epsilon)}\to S^1$ have diameter $\Theta(\epsilon^{q_j})$.   Moreover the Polar Wedge Lemma \cite[Proposition 3.4]{BirbrairNeumannPichon2014}   guarantees that the fibers of the restriction  of $\zeta_j^{(\epsilon)}$  to the link of a component of a polar wedge  inside $(Z_j,0)$ have diameter at  most $\Theta(\epsilon^{q_j})$.  
\end{proof}
 
 In  \cite[Section 7]{BirbrairNeumannPichon2014}, it is proved that each $Z_j^{(\epsilon)}$  contains loops which are essential in $X^{(\epsilon)}$.  As a consequence of  Theorem   \ref{thm:fastloops}, we obtain the existence of families of fast    loops $\gamma_{\epsilon}$ inside each $Z_j^{(\epsilon)}$.

\section{Geometric decompositions of a surface singularity} \label{part 4}

\index{geometric decomposition}

In this part, we explain how to break the thin pieces of the  thick-thin decomposition into standard pieces  which are still invariant by bilipschitz change of the inner metric. The resulting decomposition of $(X,0)$ is what we call the inner geometric decomposition of $(X,0)$.  Then, we will define  the outer geometric decomposition of $(X,0)$, which is a refinement of the inner one, and which is  invariant by bilipschitz change of the outer metric. 

The inner and outer geometric decompositions  will lead to several key results: 
\begin{enumerate}
\item The complete classification of the  inner Lipschitz geometry of a normal surface germ (Theorem \ref{th:classification});
\item A refined geometric decomposition which is an invariant of the outer Lipschitz geometry (Theorem \ref{thm:outer invariant});
\item A list of analytic invariants of the surface which are in fact invariants of the  outer Lipschitz geometry (Theorem \ref{th:invariants from geometry});
\end{enumerate}

\subsection{The standard pieces}

In this section, we introduce the standard pieces of our geometric  decompositions. We refer to \cite[Sections 11
and 13]{BirbrairNeumannPichon2014} for more details.

 The pieces are topologically
  conical, but usually with metrics that make them shrink non-linearly
  towards the cone point.  We will consider these pieces as germs at
  their cone-points, but for the moment, to simplify notation, we
  suppress this.
  
\subsubsection{The $\mathbf B$-pieces} \index{geometric decomposition!$B$-piece of a}

Let us start with a prototype which already appeared earlier in these notes (Example \ref{ex:thin2}).
Choose $q>1$ in $\Q$ and  $0< a < b$
   in $\R$. Let $Z\in \C^2$ be defined as the semi-algebraic set
   $$Z:=\{(x,y)\subset\C^2:y=\lambda x^q, a\le|\lambda|\le b \}. $$
   Then for all $\epsilon >0$, the intersection $Z^{(\epsilon)} = Z \cap \{|x|=\epsilon \}$ is a $3$-manifold (namely a thickened torus) and the restriction of the function $x$ to $Z^{(\epsilon)}$ defines a locally trivial fibration  $x \colon Z^{(\epsilon)}  \to  S^1_{\epsilon}$  whose fibers are annuli with diameter  $\Theta(\epsilon^q)$.
 
 \begin{definition}[\bf$\boldsymbol{B(q)}$-pieces]\label{def:Bq}  
   Let $F$ be a compact oriented $2$-manifold, $\phi\colon F\to F$ an
   orientation preserving diffeomorphism, and $M_\phi$ the mapping
   torus of $\phi$, defined as:
$$M_\phi:=([0,2\pi]\times F)/((2\pi,x)\sim(0,\phi(x)))\,.$$
Given a rational number   $q > 1$ , we will define a metric space
$B(F,\phi,q)$ which is topologically the cone on the mapping torus
$M_\phi$.
 
For each $0\le \theta\le 2\pi$ choose a Riemannian metric $g_\theta$
on $F$, varying smoothly with $\theta$, such that for some small
$\delta>0$:
$$
g_\theta=
\begin{cases}
g_0&\text{ for } \theta\in[0,\delta]\,,\\
\phi^*g_{0}&\text{ for }\theta\in[2\pi-\delta,2\pi]\,.
\end{cases}
$$
Then for any $r\in(0,1]$ the metric $r^2d\theta^2+r^{2q}g_\theta$ on
$[0,2\pi]\times F$ induces a smooth metric on $M_\phi$. Thus
$$dr^2+r^2d\theta^2+r^{2q}g_\theta$$ defines a smooth metric on
$(0,1]\times M_\phi$. The metric completion of $(0,1]\times M_\phi$
adds a single point at $r=0$.  Denote this completion by $B(F,\phi,q)$. We call a metric space which is bilipschitz homeomorphic to $B(F,\phi,q)$ a  
     {\bf $B(q)$-piece} or simply {a} {\bf $B$-piece}. 

A $B(q)$-piece such that $F$ is a disc   is called a {\bf $D(q)$-piece} or simply
 {a}   {\bf $D$-piece}.   \index{geometric decomposition!$D$-piece of a}
 
 A $B(q)$-piece such that $F$ is an annulus $S^1 \times [0,1]$ is called an $A(q,q)$-piece.  
\end{definition}

\begin{example} \label{example:B} The following is based on classical theory of plane curve singularities and is a generalization of the prototype given before Definition \ref{def:Bq}. Let $\sigma \colon Y \to \C^2$ be a sequence of blow-ups of points starting with the blow-up of the origin and let $E_i $ be a component of $\sigma^{-1}(0)$ which is not the component created by the first blow-up. Then the inner rate $q_{E_i}$ is strictly greater than $1$,  $B_i = \sigma({\mathscr{N}}(E_i))$ is a $B(q_{E_i})$-piece fibered by the restriction of a generic linear form and the fiber consists of a disc minus a finite union of open discs inside it.

This is based on the fact that  in suitable coordinates $(x,y)$, one may construct such a piece $B_i$  as a union of curves $ \gamma_{\lambda} \colon  y= \sum_{k=1}^m a_k x^{p_k} +  \lambda x^{q_{E_i}} $, where $p_1 < \ldots < p_m < q_{E_i}$. Here $ y= \sum_{k=1}^m a_k x^{p_k}$ is the common part of their Puiseux series and   the coefficient $\lambda \in \C^*$ varies in a compact disc minus a finite union of  open discs inside it. 

Notice that if  $E_i$ intersects exactly one other exceptional curve $E_j$, then one gets a $D(q_{E_i})$-piece.   If $E_i$ intersects exactly two  other curves $E_j$ and $E_k$,  one gets an $A(q_{E_i},q_{E_i})$-piece. 
\end{example}

\subsubsection{The $\mathbf A$-pieces}

\index{geometric decomposition!$A$-piece of a}
 
Again, we start with a prototype.
Choose $1 \leq q < q'$ in $\Q$ and  $0 < a 
   $ in $\R$ and let $Z \subset  \C^2$  be defined as the semi-algebraic set
   $$Z:=
   {\{(x,y)\subset\C^2: y=\lambda x^s, 
|\lambda|=a,  q\leq s \leq q' \}}\,.$$
Then for all $\epsilon >0$, the
intersection $Z^{(\epsilon)} = Z \cap \{|x|=\epsilon \}$ is a
thickened torus whose restriction of the function $x$ to
$Z^{(\epsilon)}$ defines a locally trivial fibration
$x \colon Z^{(\epsilon)} \to S^1_{\epsilon}$ whose fibers are flat
annuli with outer boundary of length $\Theta(\epsilon^q)$ and inner
boundary of length $\Theta(\epsilon^{q'})$.

\begin{definition}[\bf$\boldsymbol{A(q,q')}$-pieces]\label{def:Aqq'}
  Let $q,q'$ be rational numbers such that $1\le q  \leq q'$. Let $A$ be
  the Euclidean annulus $\{(\rho,\psi):1\le \rho\le 2,\, 0\le \psi\le
  2\pi\}$ in polar coordinates and for $0<r\le 1$ let $g^{(r)}_{q,q'}$
  be the metric on $A$:
$$g^{(r)}_{q,q'}:=(r^q-r^{q'})^2d\rho^2+((\rho-1)r^q+(2-\rho)r^{q'})^2d\psi^2\,.
$$ 
Endowed  with this metric, $A$ is isometric to the Euclidean annulus with
inner and outer radii $r^{q'}$ and $r^q$. The metric completion of
$(0,1]\times \S^1\times A$ with the metric
$$dr^2+r^2d\theta^2+g^{(r)}_{q,q'}$$ compactifies it by adding a single point at
$r=0$.  We call a metric space which is bilipschitz homeomorphic to
this completion an {\bf $A(q,q')$-piece} or simply  {an} {\bf $A$-piece}.
 \end{definition}

Notice that when $q=q'$,  this definition of  $A(q,q)$ coincides with that introduced in Definition \ref{def:Bq}. 

\begin{example}  \label{example:A}  Let $\sigma \colon Y \to \C^2$ be as in Example \ref{example:B} and let $T$ be its dual tree. As already mentioned in Remark \ref{rk:inner rates},  the inner rates along any path from the root vertex to a leaf of $T$ form a strictly increasing sequence.  In particular, any edge  $e$   in $T$ joins two vertices $v$ and $v'$, with  inner rates respectively $q$ and $q'$ with  $1 \leq q <q'$. Moreover,  the  semialgebraic set $Z = \sigma (N(v) \cap N(v'))$ is an $A(q,q')$-piece fibered by the restriction of a generic linear form  and is bounded by  the $B(q)$- and $B(q')$-pieces  $\sigma({\mathscr{N}}(v))$ and  $\sigma({\mathscr{N}}(v'))$.  

More generally, let $S \subset T$ be a string in $T$ which does not contain the root vertex of $T$.  let $1<q<q'$  be the two inner rates associated with the two vertices adjacent to $S$. Then  $Z = \sigma(N(S))$ is an $A(q,q')$-piece fibered by the restriction of a generic linear form. 
\end{example}

\begin{definition}[\bf Rate] \index{inner rate}
  The rational number $q$ is called the  {\bf rate} of $B(q)$ or  $D(q)$. The rational numbers $q$ and $q'$ are the two {\bf rates} of $A(q,q')$.
\end{definition}
\subsubsection{Conical pieces \normalfont{(or $B(1)$-pieces)}}

\index{geometric decomposition!conical piece of a}

\begin{definition}[\bf Conical pieces]\label{def:p3}
 Given a compact smooth $3$-manifold $M$,
  choose a Riemannian metric $g$ on $M$ and consider the metric
  $dr^2+r^2g$ on $(0,1]\times M$. The completion of this adds a point
  at $r=0$, giving a {\bf metric cone on $M$}.  We call a metric space which is bilipschitz homeomorphic to a metric cone a {\bf conical piece} or  a {\bf $B(1)$-piece}  (they were called $CM$-pieces in \cite{BirbrairNeumannPichon2014}).
\end{definition}

\begin{example}  \label{example:C} Let  $\sigma \colon Y \to X$ and $T$  be as in Example \ref{example:B} and let $v_1$ be the root vertex of $T$. Then $\sigma({{\mathscr{N}}}(v_1))$ is a conical piece. 
\end{example}

\subsection{Geometric decompositions of $\C^2$}

A geometric decomposition of a semi-algebraic  germ $(Y,0)$ consists of a decomposition of $(Y,0)$ as a union of  $A$, $B$ and conical pieces glued along their boundary components in such a way that the fibrations of $B$ and $A$ pieces coincide on the gluing. 

Examples \ref{example:B},  \ref{example:A} and  \ref{example:C} show that any sequence  $\sigma \colon Y \to \C^2$  of blow-ups of points starting with the blow-up of the origin defines a geometric decomposition of $(\C^2,0)$ whose $B$-pieces are in bijection with the exceptional curves  $E_i$ in $\sigma^{-1}(0)$ and the intermediate $A(q,q')$-pieces, $q<q'$ with the intersection points $E_i \cap E_j$. 

\begin{definition} \label{def:geometric decomposition C2}
We call this geometric decomposition of $(\C^2,0)$  the geometric decomposition  associated with $\sigma$. 
  \end{definition}

\begin{example} \label{example:geometric decomposition1} 

Consider the minimal resolution $\sigma$ of the curve germ $\gamma$ with Puiseux expansion $y=x^{3/2}+ x^{7/4}$.  Its resolution tree $T$, with exceptional curves  $E_i$ labelled in order of occurence in the sequence of blow-ups, is pictured on Figure \ref{fig:16}. Each vertex is also weighted by the corresponding  self-intersection $E_i^2$ and by  the inner rate $q_{E_i}$ in bold. The  inner rates  $q_{E_1} = 1, q_{E_2} = 2$ and $q_{E_3} = \frac{3}{2}$   are computed as in example \ref{example:inner rates 1} using the first characteristic exponent $\frac{3}{2}=[1,2]^+$. The two last inner rates are computed using the  characteristic Puiseux exponents  $\frac{3}{2}$ and $\frac{7}{4}$ as follows. Set $\frac{p_1}{q_1} = \frac{3}{2}$  and   $\frac{p_2}{q_2} = \frac{7}{4}$ and  write $\frac{p_2}{q_2} = \frac{p_1}{q_1} + \frac{1}{q_1} \frac{p'_2}{q'_2}$.  Then the two last inner rates are computed by using the continued fraction development   $\frac{p'_2}{q'_2} = [a_1,\ldots,a_r]^+$.   In our case, we  have  $\frac{7}{4}= \frac{3}{2} +\frac{1}{2}. \frac{1}{2}$, so $ \frac{p'_2}{q'_2} = \frac{1}{2} = [0,2]^+$. This gives $q_{E_4} = \frac{3}{2} + \frac{1}{1} = \frac{5}{2}$ and $q_{E_5} = \frac{3}{2} + \frac{1}{2} = \frac{7}{4}$. (Again, we refer to \cite{EisenbudNeumann1985} or \cite{Wall2004} for details on  these   computations).

  \begin{figure}[ht] 
\centering
 \begin{tikzpicture}
   \draw[thin](0,0)--(0,3);
    \draw[thin](0,1.5)--(1.5,1.5);
   
   \draw[thin](0, 3)--(1.5,3);
   \draw[fill](0,0)circle(2pt);\node(a)at(0.1,-0.2){$-3$};\node(a)at(0.2,0.2){\scriptsize{$E_1$}}; \node(a)at(-0.2,0){{\bf 1}};
   \draw[fill](0,1.5)circle(2pt);\node(a)at(0.35,1.3){$-3$};\node(a)at(0.2,1.7){\scriptsize{$E_3$}}; \node(a)at(-0.2,1.5){{$\bf\frac32$}};
   \draw[fill](0,3)circle(2pt);\node(a)at(0.35,2.8){$-1$};\node(a)at(0.2,3.2){\scriptsize{$E_5$}}; \node(a)at(-0.2,2.9){{$\bf\frac74$}};
   \draw[fill](1.5,1.5)circle(2pt);\node(a)at(1.3,1.2){$-2$};\node(a)at(1.3, 1.7){\scriptsize{$E_2$}}; \node(a)at(1.8,1.5){{\bf 2}};
   \draw[fill](1.5,3)circle(2pt);\node(a)at(1.3,2.7){$-2$};\node(a)at(1.3, 3.2){\scriptsize{$E_4$}};\node(a)at(1.8,3){{$\bf\frac52$}};

\end{tikzpicture}
 \caption{Geometric decomposition of $(\C^2,0)$ associated with the resolution of the curve $y=x^{3/2}+ x^{7/4}$ }\label{fig:16}
\end{figure}
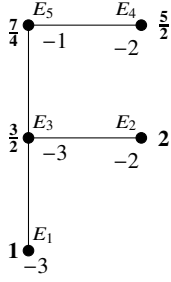

The underlying geometric decomposition of $(\C^2,0)$ consists of: 
\begin{itemize}
\item Five  $B$-pieces  $\sigma({\mathscr{N}}(E_i)), i=1,\ldots, 5$  in bijection with the  vertices of $T$    having rates respectively $1, 2, \frac{3}{2},  \frac{5}{2},  \frac{7}{4}$.  Notice that the $B$-pieces corresponding to $E_2$ and $E_4$ are respectively a    $D(2)$- and a $D(\frac{5}{2})$-piece since the corresponding vertices have valence one.
\item Four  $A$-pieces in bijection with the  edges of $T$: $\sigma(N(E_1) \cap N(E_3)),    \sigma(N(E_3) \cap N(E_2)), \sigma(N(E_3) \cap N(E_5))$ and $\sigma(N(E_5) \cap N(E_4))$ which are respectively an $A(1,\frac{3}{2})$-piece, an $A(\frac{3}{2},2)$-piece, an $A(\frac{3}{2} ,\frac{7}{4})$-piece and an   $A(\frac{7}{4}, \frac{5}{2})$-piece.
\end{itemize}
\end{example}

\begin{example} \label{example:geometric decomposition2}  The  trees $T_0$ and $T$ in Example \ref{example:inner rates 1} describe two different geometric decompositions of $(\C^2,0)$ associated with two resolutions of the curve $y=x^{5/3}$. 
\end{example}

   The following lemma   shows that one can simplify a  geometric decomposition by amalgamating pieces. In this lemma $\cong$ means bilipschitz equivalence
and $\cup$ represents gluing along appropriate boundary components 
by an isometry. $D^2$ means the standard $2$-disc. 
\begin{lemma} \label{amalgamation} [Amalgamation Lemma] \index{amalgamation}
\begin{enumerate}
\item\label{rule:D2} $B(D^2,\phi,q)\cong B(D^2,id,q)$; $B(S^1\times
  I,\phi,q)\cong B(S^1\times I,id,q)$.
\item\label{rule:AA} $A(q,q')\cup A(q',q'')\cong A(q,q'')$.
\item\label{rule:FF} If $F$ is the result of gluing a surface
  $F'$ to a disk $D^2$ along boundary components then
  $B(F',\phi|_{F'},q)\cup B(D^2,\phi|_{D^2},q)\cong B(F,\phi,q)$.
\item\label{rule:AD} $A(q,q')\cup B(D^2,id,q')\cong B(D^2,id,q)$.
\item\label{rule:CM}Each $B(D^2,id,1)$, $B(S^1\times I, id,
    1)$ or $B(F, \phi, 1)$ piece is a conical piece and a union of conical 
    pieces glued along boundary components is a conical piece.\qed
\end{enumerate}
\end{lemma}

\begin{example} Consider again the geometric decomposition of $(\C^2,0)$ introduced in Example \ref{example:geometric decomposition1}. We  can amalgamate    the $D(2)$-piece union the  $A(\frac32,2)$-piece to the neighbour  $B(\frac32)$-piece. We can also amalgamate the $D(\frac52)$-piece union the adjacent   $A(\frac74, \frac52)$-piece to the neighbour  $B(\frac74)$-piece.   This produces a geometric decomposition of $(\C^2,0)$  represented by the tree of Figure \ref{fig:17},  where we write inner rates only at  the central vertices of $B$-pieces and not at the amalgamated pieces. This decomposition has five pieces: a  conical $B(1)$ (black vertex), a $B(\frac32)$-piece (red vertices), a $B(\frac74)$-piece (blue vertices)  and intermediate $A(1,\frac32)$- and $A(\frac74,\frac52)$-pieces.

\begin{figure}[ht] 
\centering

 \begin{tikzpicture}
   \draw[thin](0,0)--(0,3);
    \draw[thin](0,1.5)--(1.5,1.5);
   
   \draw[thin](0, 3)--(1.5,3);
   \draw[fill](0,0)circle(2pt);\node(a)at(0.35,0){$-3$}; \node(a)at(-0.2,0){{\bf 1}};
   \draw[fill=red](0,1.5)circle(2pt);\node(a)at(0.35,1.3){$-3$}; \node(a)at(-0.2,1.5){{$\bf\frac32$}};
   \draw[fill=blue](0,3)circle(2pt);\node(a)at(0.35,2.8){$-1$};  \node(a)at(-0.2,2.9){{$\bf\frac74$}};
   \draw[fill=red](1.5,1.5)circle(2pt);\node(a)at(1.6,1.3){$-2$}; 
   \draw[fill=blue](1.5,3)circle(2pt);\node(a)at(1.6,2.8){$-2$};    

\end{tikzpicture}
 \caption{Amalgamated geometric decomposition}\label{fig:17}
\end{figure}
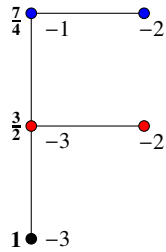
  \end{example}

\begin{remark}
Notice that  the new $B(\frac74)$-piece is now a $D$-piece. Then we  could continue the amalgamation  process by amalgamating iteratively all $D$-pieces. Of course, in the case of a geometric decomposition of $(\C^2,0)$, an iterative amalgamation of the pieces always produces eventually a unique conical piece which is the whole $(\C^2,0)$. 
\end{remark}

\begin{example} In the  tree $T_0$  of Example \ref{example:geometric decomposition2}, the amalgamation of the $D(2)$-piece union the $A(\frac53,2)$-piece to the neighbour $B(\frac53)$-piece forms a bigger $B(\frac53)$-piece. The amalgamation of the $A(\frac32, \frac32)$-piece with the two neighbour $A(1,\frac32)$- and  $A(\frac32,\frac53)$-pieces creates an intermediate $A(1,\frac53)$-piece  between the $B(1)$- and  the $B(\frac53)$-pieces. This creates a new geometric decomposition  of $(\C^2,0)$ with two $B$-pieces and one $A$-piece represented on Figure \ref{fig:18}. The red vertices correspond to the $B(\frac53)$-piece and the white one to the $A$-piece. 

 \begin{figure}[ht] 
\centering

 \begin{tikzpicture}
 

   \draw[thin ](-2,0)--(1,0);
  \draw[fill=red ] (-2,0)circle(2pt);
   \draw[fill =red] (-1,0)circle(2pt);
    \draw[fill =white] (0,0)circle(2pt);
     \draw[fill ] (1,0)circle(2pt);
     

\node(a)at(1,0.4){ $\mathbf 1$};
\node(a)at(-0.9,0.4){$   \frac{\mathbf 5}{\mathbf 3}$};

\node(a)at(-2,-0.4){$-3$};
\node(a)at(-1,-0.4){$-1$};
\node(a)at(0,-0.4){$-2$};
\node(a)at(1,-0.4){$-3$};

\end{tikzpicture} 
  \caption{Amalgamated geometric decomposition for the curve $y=x^{5/3}$ }\label{fig:18}
\end{figure}
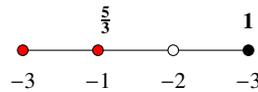
\end{example}
  \subsection{The Polar Wedge Lemma}

\index{polar wedge!Lemma}

Let $(X,0) \subset (\C^2,0)$  be a normal surface singularity. Consider a  linear projection  $\C^n \to \C^2$ which is generic for $(X,0)$ (e.g. \cite[Definition 2.4 ]{NeumannPedersenPichon2019-1}) and denote again by  $\ell \colon (X,0) \to (\C^2,0)$ its restriction to $(X,0)$. Let $\Pi$ be  the polar curve of $\ell$ and let  $\Delta=\ell(\Pi)$ be  its discriminant curve. 

\begin{proposition} [Polar Wedge Lemma] \cite[3.4]{BirbrairNeumannPichon2014}  \label{polar wedge lemma}  Consider the resolution $\sigma \colon Y  \to \C^2$ which resolves the base points of the family of projections of generic polar curves  $(\ell(\Pi_{\cal D}))_{{\cal D} \in \Omega}$. Let  $\Pi_0$ be an irreducible  component of $\Pi$ and let $\Delta_0 = \ell(\Pi_0)$.  Let $C$ be the irreducible component of $\sigma^{-1}(0)$ which intersects the strict transform of $\Delta_0^*$ by $\sigma$.  

 Let $W_{\Pi_0}$ be a polar wedge around $\Pi_0$ as introduced in Definition  \ref{def:polar wedge}.  Then $W_{\Pi_0}$ is a   $D(q_C)$-piece, and when $q_C>1$, $W_{\Pi_0}$ is  fibered by its  intersections with the real surfaces $\{h=t\} \cap X$, where $h\colon \C^n\to \C$ is a generic linear form.
\end{proposition}
   
\subsection{The     geometric decomposition and the complete Lipschitz classification for the inner metric} \label{sec: inner geometric decomposition}

Let $(X,0)$ be a surface germ, let $\ell \colon (X,0) \to (\C^2,0)$ be a generic linear projection with polar curve $\Pi$ and let $W$ be a polar wedge around $\Pi$. Let $\sigma \colon Y \to \C^2$ be  the minimal  sequence of blow-ups which resolves the base points of the family of projected polar curves $(\ell(\Pi_{\cal D}))_{{\cal D} \in \Omega}$ and consider the  geometric decomposition of $(\C^2,0)$ associated with $\sigma$ (Definition \ref{def:geometric decomposition C2}). 

\begin{definition} 
Let $T$ be the resolution tree of $\sigma$. We call {\it $\Delta$-curve}     any component of $\sigma^{-1}(0)$ which  intersects the strict transform of the discriminant curve $\Delta$ of $\ell$, and we call {\it $\Delta$-node} of $T$ any vertex representing a $\Delta$-curve. 

We call {\it node} of $T$ any vertex which is either the root-vertex or a $\Delta$-node or a vertex with valence $\geq 3$. 
\end{definition}

Using Lemma \ref{amalgamation}, we amalgamate iteratively all the $D$-pieces of the  geometric decomposition of $(\C^2,0)$ associated with $\sigma$  with the rule that we never amalgamate a piece containing a component of the discriminant curve $\Delta$ of $\ell$. We then obtain a geometric decomposition of $(\C^2,0)$ whose pieces are in bijection with the nodes of $T$. 

\begin{definition} We call this  decomposition {\it the geometric decomposition of $(\C^2,0)$ associated with  the projection $\ell \colon (X,0) \to (\C^2,0)$}. 
\end{definition}

\begin{example}    \index{surface singularity!$E_8$} Consider again the germ  $(X,0)$ of the  surface  $E_8$  with equation $x^2+y^3+z^5=0$ and the projection $\ell \colon (x,y,z) \to (y,z)$. In order to compute the  geometric decomposition of $(\C^2,0)$ associated with $\ell$, we need to compute  a resolution graph of $\sigma \colon Y \to \C^2$ as defined above with its inner rates. We will first compute the minimal resolution of $(X,0)$ which factors through Nash modification. 

We first consider the graph $\Gamma$ of the minimal resolution  $\pi$ of $E_8$ as computed in the Appendix of the present notes.  We add to $\Gamma$   decorations by arrows  corresponding to the strict transforms  of the coordinate functions $x, y$ and $z \colon (X,0) \to (\C,0)$ and we denote the exceptional curves by $E_i, i=1,\ldots,8$ (the order is random).   All the self-intersections of the exceptional curves equal $-2$ so we do not  write them on the graph. We obtain the graph of Figure \ref{fig:19}.

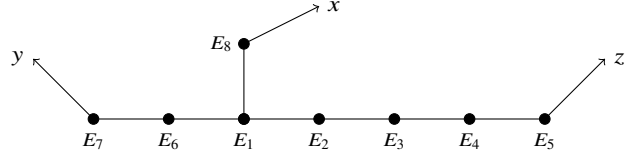
\begin{figure}[ht] 
\centering
\begin{tikzpicture}
   
     \draw[thin ](-1,0)--(5,0);
      \draw[thin ](1,0)--(1,1);
      \draw[thin,>-stealth,->](-1,0)--+(-0.8,0.8);
       \draw[thin,>-stealth,->](5,0)--+(0.8,0.8);
        
        \draw[thin,>-stealth,->](1,1)--+(1,0.5);
  \draw[fill   ] (1,0)circle(2pt);
  \draw[fill ] (3,0)circle(2pt);
     \draw[fill  ] (1,1)circle(2pt);      
        \draw[fill  ] (4,0)circle(2pt);
  \draw[fill  ] (0,0)circle(2pt);
   \draw[fill  ] (2,0)circle(2pt);
    \draw[fill ] (-1,0)circle(2pt);
   \draw[fill ] (1,0)circle(2pt); 
      \draw[fill] (5,0)circle(2pt); 
 
\node(a)at(1,-0.3){   \scriptsize{$E_ 1$}};
\node(a)at(2,-0.3){    \scriptsize{$E_ 2$}};
\node(a)at(3,-0.3){    \scriptsize{$E_3$}};
\node(a)at(4,-0.3){    \scriptsize{$E_ 4$}};
\node(a)at(5,-0.3){    \scriptsize{$E_ 5$}};
\node(a)at(0,-0.3){    \scriptsize{$E_ 6$}};
\node(a)at(-1,-0.3){    \scriptsize{$E_ 7$}};
\node(a)at(0.7,1){    \scriptsize{$E_ 8$}};

\node(a)at(-2,0.8){   $y$};
\node(a)at(2.2,1.5){   $x$};
\node(a)at(6,0.8){   $z$};
 
  \end{tikzpicture} 
  \caption{Resolution of the coordinates functions on the $E_8$ singularity}\label{fig:19}
\end{figure}

  Let $h \colon (X,0) \to (\C,0)$ be an analytic function, and let $(h \circ \pi) = \sum_{j=1}^8 m_j E_j + h^*$ be its total transform by $\pi$, so $m_j$ denotes the multiplicity of $h$ along $E_j$ and $h^*$ its strict transform by $\pi$. Then, for all $j=1,\ldots,8$, we have $(h \circ \pi).E_j=0$ (\cite[Theorem 2.6]{Laufer1971}). Using this, we compute    the total 
  transforms by $\pi$ of the coordinate
  functions $x, y$ and $z$:
  \begin{align*}
    (x \circ \pi) &= 15E_1 +12E_2+9E_3+6E_4+3E_5+10E_6+5E_7 +8E_8+   x^* \\
(y \circ \pi) &= 10E_1 +8E_2+6E_3+4E_4+2E_5+7E_6+4E_7 +5E_8+   y^* \\
(z \circ \pi) &= 6E_1 +5E_2+4E_3+3E_4+2E_5+4E_6+2E_7 +3E_8+   z^* 
  \end{align*}

  Set $f(x,y,z) = x^2+y^3+z^5$. The polar curve $\Pi$ of a generic
  linear projection $\ell\colon (X,0) \to (\C^2,0)$ has equation $g=0$
  where $g$ is a generic linear combination of the partial derivatives
  $f_x = 2x$, $f_y=3y^2$ and $f_z=5z^4$. The multiplicities of $g$ are
  given by the minimum of the compact part of the three divisors
  \begin{align*}
(f_x \circ \pi) &= 15E_1 +12E_2+9E_3+6E_4+3E_5+10E_6+5E_7 +8E_8+  f_x^* \\
(f_y \circ \pi) &=  20E_1 +16E_2+12E_3+8E_4+4E_5+14E_6+8E_7 +10E_8+  f_y^*  \\
(f_z \circ \pi) &= 24E_1 +20E_2+16E_3+12E_4+8E_5+16E_6+8E_7 +12E_8+ f_z^*
  \end{align*}
We then obtain that the total transform of $g$ is equal to:
$$(g \circ \pi) = 15E_1 +12E_2+9E_3+6E_4+3E_5+10E_6+5E_7 +8E_8+ \Pi^*\,.$$
In particular, $\Pi$ is resolved by $\pi$ and its strict transform
$\Pi^*$ has just one component, which intersects $E_8$.  

\begin{exercise} \begin{enumerate}
\item Prove that since the
multiplicities $m_8(f_x)=8$, $m_8(f_y)=10$ and $m_8(z)=12$ along $E_8$
are distinct, the family of polar curves, i.e., the linear
  system generated by $f_x, f_y$ and $f_z$, has a base point on
$E_8$. 
\item Prove that one  must blow up twice to get an exceptional curve $E_{10}$
along which $m_{10}(f_x)=m_{10}(f_y)$, which resolves the linear
system and,  that this gives  the resolution graph $\Gamma'$ of Figure \ref{fig:20}.
\end{enumerate}
\end{exercise}

 \begin{figure}[ht] 
\centering
\begin{tikzpicture}
   
   \draw[thin ](-1,0)--(5,0);
      \draw[thin ](1,0)--(1,3);

        \draw[thin,>-stealth,->](1,3)--+(-0.5,0.8);
  \draw[fill   ] (1,0)circle(2pt);
  \draw[fill   ] (3,0)circle(2pt);
     \draw[fill   ] (1,1)circle(2pt);      
        \draw[fill   ] (4,0)circle(2pt);
  \draw[fill    ] (0,0)circle(2pt);
   \draw[fill    ] (2,0)circle(2pt);
    \draw[fill  ] (-1,0)circle(2pt);
       \draw[fill  ] (5,0)circle(2pt); 
      
      \draw[fill=white ] (1,2)circle(2pt); 
\draw[fill=white ] (1,3)circle(2pt); 
    
  \node(a)at(0.6,1){   $-3$};
 \node(a)at(0.6,2){   $-2$};
\node(a)at(0.6,3){   $-1$};

\node(a)at(1.3,2){   \scriptsize{$ E_ 9$}};
\node(a)at(1.4,3){   \scriptsize{$ E_ {10}$}};

\node(a)at(1,4){   $\Pi^*$};
  \end{tikzpicture} 
 \caption{The graph $\Gamma'$}\label{fig:20}
\end{figure}
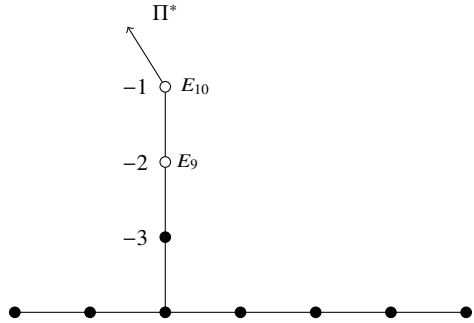
  
 Now, consider the computation of the resolution of $E_8$ by Laufer's method (see Appendix \ref{appendix}) which consists of computing the double over  $ \ell \colon( X,0) \to (\C^2,0)$ branched over the discriminant curve $\Delta \colon y^3+z^5=0 $. We start with the minimal resolution $\sigma' \colon Y' \to \C^2$ of $\Delta$, and we see  from  the computation of self-intersections given in  \ref{appendix} that we need to blow up five times the strict transform $\Delta^*$ in order to get the resolution graph $\Gamma'$. The resulting map  is the morphism $\sigma \colon Y \to \C^2$ which resolves the  base points of the family of projected polar curves $(\ell(\Pi_{\cal D}))_{{\cal D} \subset \Omega}$.  The morphism $\sigma$ is a composition of blow-ups of points and the  last  exceptional curve  created in the process is the $\Delta$-curve. Its inner rate is  $\frac{5}{3} + 5.\frac{1}{3} = \frac{10}{3}$. 
  
  The    geometric decomposition of $\C^2$ associated with $\ell$ is described by the resolution  tree of $\sigma$ with nodes weighted by the inner rates (Figure \ref{fig:21}). 
 
\begin{figure}[ht] 
\centering
 \begin{tikzpicture}
 

   \draw[thin ](-1,0)--(2,0);
    \draw[thin ](0,0)--(3.5,3.5);
  \draw[fill =red] (-1,0)circle(2pt);
   \draw[fill =red] (0,0)circle(2pt);
    \draw[fill =white] (1,0)circle(2pt);
     \draw[fill ] (2,0)circle(2pt);
     
        \draw[fill =white] (0.7,0.7)circle(2pt);
           \draw[fill =white] (1.4,1.4)circle(2pt);
              \draw[fill =white] (2.1,2.1)circle(2pt);
                 \draw[fill =white] (2.8,2.8)circle(2pt);
                    \draw[fill =red] (3.5,3.5)circle(2pt);
     

\node(a)at(2,0.4){ $\mathbf 1$};
\node(a)at(0,0.4){$   \frac{\mathbf 5}{\mathbf 3}$};
\node(a)at(3.2,3.6){ $\frac{\mathbf {10}}{\mathbf 3}$};

\node(a)at(2,-0.4){$-3$};
\node(a)at(1,-0.4){$-2$};
\node(a)at(0,-0.4){$-2$};
\node(a)at(-1,-0.4){$-3$};

\node(a)at(1.1,0.7){$-2$};
\node(a)at(1.8,1.4){$-2$};
\node(a)at(2.5,2.1){$-2$};
\node(a)at(3.2,2.8){$-2$};
\node(a)at(3.9,3.5){$-1$};

\end{tikzpicture} 
  \caption{Geometric decomposition of $(\C^2,0)$ associated with $\ell$ }\label{fig:21}
\end{figure}
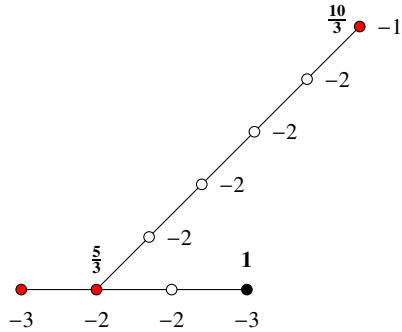

  Notice that the inner rate of the $\Delta$-curve, which is also the inner rate of the curve $E_{v_{10}}$ of $\pi^{-1}(0)$  can also be computing using the equations as follows.  For  a generic $(a,b) \in \C^2$, 
  $x+ay^2+bz^4=0$ is the equation of the polar curve $\Pi_{a,b}$ of a
  generic projection. The image $\ell(\Pi_{a,b}) \subset \C^2$ under
  the projection $\ell=(y,z)$ has equation
$$y^3 +a^2y^4 + 2aby^2z^4+z^5+b^2z^8=0$$
The discriminant curve $\Delta= \ell(\Pi_{0,0})$ has Puiseux expansion $y=(-z)^{\frac53}$,
while for $(a,b) \neq (0,0)$, we get for   $\ell(\Pi_{a,b})$ a Puiseux expansion
$y=(-z)^{\frac53} - \frac{a^2}3z^{\frac{10}3}+\cdots$. So the discriminant curve $\Delta$ has highest characteristic
exponent $\frac53$ and its contact exponent  with a generic  $\ell(\Pi_{a,b})$  is $\frac{10}3$. 

 \end{example}

 By construction, the projection $\ell(W)$  of a polar wedge  $W$ is a union of $D$-pieces which refines the geometric decomposition of $(\C^2,0)$ associated with $\ell$. By Lemma \ref{lem:local bil bound1}, which guarantees that $\ell$ is a local bilipschitz homeomorphism for the inner metric outside $W$, any piece of this geometric decomposition outside the polar wedge $W$  lifts to a piece of the same type.  We  obtain a geometric decomposition of $\overline{X \setminus W}$. Finally, the Polar Wedge Lemma  \ref{polar wedge lemma} says that $W$ is a union of $D$-pieces whose fibrations match with those of its neighbour $B$-pieces  in $\overline{X \setminus W}$. We  obtain the following result: 
 
 \begin{proposition} Each  $B(q)$-piece (resp.\ $A(q,q')$-piece) of the  geometric decomposition of $(\C^2,0)$ associated with $\ell$ lifts by $\ell$ to a union of  $B(q)$-pieces (resp.\ $A(q,q')$-pieces) in $(X,0)$ (with  the same rates). 
 \end{proposition}
 
    Therefore, we obtain a geometric decomposition of $(X,0)$ into a union of $B$-pieces and $A$-pieces obtained by  lifting by $\ell$  the $A$- and $B$-pieces of the geometric decomposition of $(\C^2,0)$ associated with $\ell$.   
   
   \begin{definition} \label{defn:initial decomposition} We call this decomposition  the {\it initial geometric decomposition} of $(X,0)$. 
\end{definition}

\begin{example} \label{example:E8 outer}  \index{surface singularity!$E_8$} The initial geometric decomposition of the surface germ $E_8$ is represented by the  graph of Figure \ref{fig:22}. The vertices corresponding to the $B(\frac53)$-piece  and the  $B(\frac{10}3)$-piece are in red, the $\cal L$-node is in black, the white vertices correspond to the $A$-pieces.

  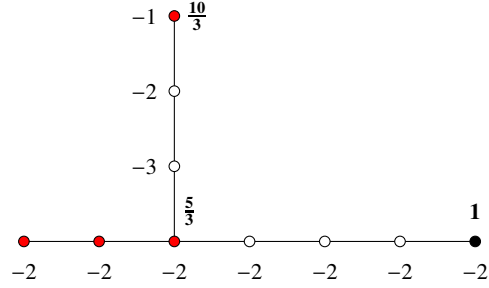
\begin{figure}[ht] 
\centering
 \begin{tikzpicture}

   \draw[thin ](0,0)--(6,0);
      \draw[thin ](2,0)--(2,3);
      
  \draw[fill] (6,0)circle(2pt);
   \draw[fill=white ] (5,0)circle(2pt);
    \draw[fill =red] (0,0)circle(2pt);
     \draw[fill =red] (1,0)circle(2pt);
        \draw[fill =red] (2,0)circle(2pt);

   \draw[fill =white] (3,0)circle(2pt);

   \draw[fill =white] (4,0)circle(2pt);

      \draw[fill =white] (2,1)circle(2pt);
          \draw[fill =white] (2,2)circle(2pt);
              \draw[fill =red] (2,3)circle(2pt);

\node(a)at(6,0.4){ $\mathbf 1$};

\node(a)at(2.2,0.4){$   \frac{\mathbf 5}{\mathbf 3}$};
\node(a)at(2.3,3){$   \frac{\mathbf {10}}{\mathbf 3}$};

\node(a)at(6,-0.4){$-2$};
\node(a)at(5,-0.4){$-2$};
\node(a)at(0,-0.4){$-2$};
\node(a)at(1,-0.4){$-2$};
\node(a)at(2,-0.4){$-2$};
\node(a)at(3,-0.4){$-2$};
\node(a)at(4,-0.4){$-2$};
\node(a)at(1.6,1){$-3$};
\node(a)at(1.6,2){$-2$};
\node(a)at(1.6,3){$-1$};
\end{tikzpicture} 
 \caption{Initial geometric decompostion for the singularity $E_8$}\label{fig:22}
\end{figure}
\end{example}

We will now amalgamate some pieces to define the {\it inner geometric decomposition} of $(X,0)$. We first  need to specify some special vertices in the resolution graph. 

\begin{definition} [Nodes]  \label{dfn:inner nodes}Let  $\pi \colon Z \to X$ be a resolution of $(X,0)$ which factors through the blow-up of the  maximal ideal  $e_0 \colon X_0 \to X$  and through the Nash modification.  Let $\Gamma$ be the dual resolution graph of $\pi$. 

We call {\it  $ \cal L$-curve}   any component of $\pi^{-1}(0)$ which corresponds to a component of $e_0^{-1}(0)$   and {\it $\cal L$-node}  any vertex of $\Gamma$ which represents an $\cal L$-curve. 

We call  {\it special $\cal P$-curve}  any component $E_i$ of $\pi^{-1}(0)$ which corresponds to a component of $\nu^{-1}(0)$ (i.e., it intersects the strict transform of the polar curve $\Pi)$ and such that 
\begin{enumerate}
\item   The curve $E_i$  intersects exactly two other components of $E_j$ and $E_k$ of $\pi^{-1}(0)$;
\item  The  inner rates   satisfy: $max(q_{E_j}, q_{E_k}) < q_{E_i}$.
\end{enumerate}
 We call  {\it special $\cal P$-node} any vertex of $\Gamma$ which represents a special $\cal P$-curve. 
 
We call {\it inner node} any vertex of $\Gamma$ which has at least three incident edges or which represents a curve with genus $>0$ or which is an $\cal L$- or a special $\cal P$-node. 
\end{definition}

Using Lemma \ref{amalgamation}, we now amalgamate iteratively  $D$ and $A$-pieces but with the rule that we never amalgamate  the special $A$-pieces with a neighbouring piece. 

\begin{definition} \label{def:inner decomposition}
We call this decomposition the {\it inner geometric decomposition} of $(X,0)$. \index{geometric decomposition! inner}
\end{definition}

The following is a straightforward consequence of this amalgamation rule. The  pieces of the inner geometric decomposition  of $(X,0)$  can be  described as follows: 

\begin{proposition} \label{prop:inner decomposition}  For each inner node $(i)$ of $\Gamma$, let $\Gamma_i$ be the subgraph of $\Gamma$ consisting of $(i)$ union any attached bamboo.  
\begin{enumerate}
\item The $B$-pieces are  the sets $B_i = \pi({\mathscr{N}}(\Gamma_i))$, in bijection with the inner nodes of $\Gamma$. Moreover, for each node $(i)$, $B_i$ is a $B(q_i)$-piece, where $q_i$ is the inner rate of the exceptional curve represented by  $(i)$  and the link $B_i^{(\epsilon)}$ is a Seifert manifold. 
\item The $A$-pieces are the sets $A_{i,j} = \pi( N(S_{i,j}))$ where $S_{i,j}$ is a string or an edge joining two nodes $(i)$ and $(j)$ of $\Gamma$. Moreover, $A_{i,j}$ is an $A(q_i,q_j)$-piece and the link $A_{i,j}^{(\epsilon)}$ is a thickened torus having a common boundary component with both $B_i^{(\epsilon)}$ and  $B_j^{(\epsilon)}$.  \end{enumerate}
In particular, the inner  geometric decomposition of $(X,0)$ induces a graph decomposition of the link $X^{(\epsilon)}$ whose Seifert components are the links  $B_i^{(\epsilon)}$ and the separating tori are in bijection with the  thickened tori $A_{i,j}^{(\epsilon)}$. 
\end{proposition}

\begin{remark} \label{rk:thin-thick decomposition}
The inner geometric decomposition is a refinement of the thick-thin decomposition. Indeed, the thick part is the union of the $B(1)$-pieces and adjacent $A(1,q)$-pieces, and the thin part is the union of the remaining pieces. 
\end{remark}

\begin{example}   \index{surface singularity!$E_8$} The inner geometric decomposition of the surface germ $E_8$ is represented by the  graph of Figure \ref{fig:23}.  The vertices corresponding to the $B(\frac53)$-piece are in red, the $\cal L$-node is in black, the white vertices correspond to the $A$-piece.

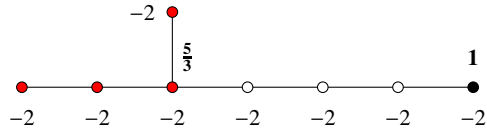
\begin{figure}[ht] 
\centering
 \begin{tikzpicture}

   \draw[thin ](0,0)--(6,0);
      \draw[thin ](2,0)--(2,1);
  \draw[fill] (6,0)circle(2pt);
   \draw[fill=white ] (5,0)circle(2pt);
    \draw[fill =red] (0,0)circle(2pt);
     \draw[fill =red] (1,0)circle(2pt);
        \draw[fill =red] (2,0)circle(2pt);

   \draw[fill =white] (3,0)circle(2pt);

   \draw[fill =white] (4,0)circle(2pt);

      \draw[fill =red] (2,1)circle(2pt);

\node(a)at(6,0.4){ $\mathbf 1$};
\node(a)at(2.2,0.4){$   \frac{\mathbf 5}{\mathbf 3}$};

\node(a)at(6,-0.4){$-2$};
\node(a)at(5,-0.4){$-2$};
\node(a)at(0,-0.4){$-2$};
\node(a)at(1,-0.4){$-2$};
\node(a)at(2,-0.4){$-2$};
\node(a)at(3,-0.4){$-2$};
\node(a)at(4,-0.4){$-2$};
\node(a)at(1.6,1){$-2$};
\end{tikzpicture} 
  \caption{Inner geometric decompostion for the singularity $E_8$}\label{fig:23}
\end{figure}

\end{example}

\begin{exercise} Draw the resolution graph with inner rates at inner nodes representing the inner geometric decomposition of the surface germ $z^2+f(x,y)=0$ where $f(x,y)=0$ is the plane curve with Puiseux expansion $y=x^{\frac32}+ x^{\frac74}$. 
\end{exercise}

 \begin{example}  \label{example:minimal1} Here is an example with a special $\cal P$-node. This is a minimal surface singularity (see \cite{Kollar1985}).  Minimal singularities  are special rational singularities which play a key role in resolution theory of surfaces, and they also share a  remarkable metric property, as shown in \cite{NeumannPedersenPichon2019-2}:  they are Lipschitz normally embedded, i.e., their inner and outer metrics are Lipschitz equivalent.  We refer to   \cite{NeumannPedersenPichon2019-2} for details on minimal singularities and for the computations on this particular example. 
 
 Consider the minimal surface singularity given by the minimal resolution graph of Figure \ref{fig:23bis}.  The $\cal L$-nodes are the black vertices.
  
\begin{figure}[ht] 
\centering
\begin{tikzpicture}
 
   \draw[thin ](-1,0)--(4,0);
      \draw[thin ](1,0)--(1,-3);
       
                \draw[thin ](0,1)--(0,0);
           \draw[fill=black   ] (0,1)circle(2pt);

           \draw[fill=black   ] (-1,0)circle(2pt);
           \draw[fill=white  ] (0,0)circle(2pt);
          \draw[fill=white ] (1,0)circle(2pt); 
           \draw[fill=black] (2,0)circle(2pt);
  \draw[fill=white ] (3,0)circle(2pt);
        \draw[fill =black ] (4,0)circle(2pt);
        
          \draw[fill=white ] (1,-1)circle(2pt);
         \draw[fill=white ] (1,-2)circle(2pt);
        \draw[fill =black ] (1,-3)circle(2pt);  
 

\node(a)at(-1,0.3){   $-4$};
\node(a)at(-0.3,0.3){   $-3$};
\node(a)at(-0.4,1){   $-2$};
\node(a)at(1,0.3){   $-3$};
\node(a)at(2,0.3){   $-3$};
\node(a)at(3,0.3){   $-2$};
\node(a)at(4,0.3){   $-2$};
\node(a)at(0.6,-1){   $-2$};
\node(a)at(0.6,-2){   $-2$};
\node(a)at(0.6,-3){   $-2$};
 
 

  \end{tikzpicture} 
   \caption{Minimal resolution of a minimal surface singularity}\label{fig:23bis}
\end{figure}
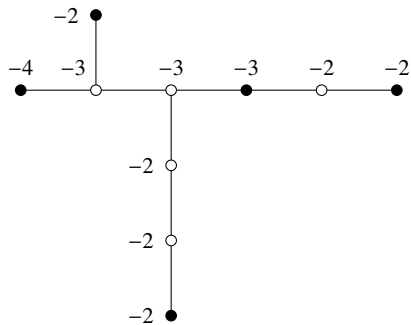

As shown in \cite{NeumannPedersenPichon2019-2}, one has to blow up once to obtain the minimal resolution which factors through Nash modification, creating the circled vertex on the graph of Figure \ref{fig:23ter}.  The arrows on this graph correspond to the components of the polar curve.  The inner    rates (in bold) are  computed in \cite{NeumannPedersenPichon2019-2}. We obtain two   special $\cal P$-nodes (the blue vertices).

\begin{figure}[ht] 
\centering
\begin{tikzpicture}
 
   \draw[thin ](-2,0)--(4,0);
      \draw[thin ](1,0)--(1,-3);
      \draw[thin,>-stealth,->](-2,0)--+(-0.8,0.8);
       \draw[thin,>-stealth,->](-2,0)--+(-0.8,-0.8);       
        \draw[thin,>-stealth,->](-2,0)--+(-1,-0.5);
          \draw[thin,>-stealth,->](-2,0)--+(-1,0.5);
          
          \node(a)at(-3.2,0.8){   $\Pi_1$};
       \node(a)at(-3.2,-0.8){   $\Pi_2$};

             \draw[thin,>-stealth,->](0,0)--+(-0.2,-1.1);
               \node(a)at(-0.2,-1.4){   $\Pi_6$};
               
                 \draw[thin,>-stealth,->](-1,0)--+(-0.2,-1.1);
          \draw[thin,>-stealth,->](-1,0)--+(0.2,-1.1);
               \node(a)at(-1.6,-1.1){   $\Pi_5$};

              \draw[thin,>-stealth,->](1,-1)--+(1.1,-0.2);
          \draw[thin,>-stealth,->](1,-1)--+(1.1,0.2);
               \node(a)at(2.1,-1.5){   $\Pi_4$};
               
            \draw[thin,>-stealth,->](3,0)--+(-0.2,-1.1);
          \draw[thin,>-stealth,->](3,0)--+(0.2,-1.1);
               \node(a)at(3.6,-1.1){   $\Pi_3$};

      \draw[thin ](-1,1)--(-1,0);
           \draw[fill=black   ] (-1,1)circle(2pt);

               \draw[fill=black ] (-2,0)circle(2pt);
           \draw[fill=white   ] (-1,0)circle(2pt);
            \draw[fill=white  ] (0,0)circle(4pt);
           \draw[fill=blue   ] (0,0)circle(2pt);
          \draw[fill=white ] (1,0)circle(2pt); 
           \draw[fill=black] (2,0)circle(2pt);
  \draw[fill=blue ] (3,0)circle(2pt);
        \draw[fill =black ] (4,0)circle(2pt);
        
          \draw[fill=white ] (1,-1)circle(2pt);
         \draw[fill=white ] (1,-2)circle(2pt);
        \draw[fill =black ] (1,-3)circle(2pt);  
 

\node(a)at(-2,0.3){   $\mathbf 1$};

\node(a)at(-1.2,0.3){   $\mathbf  2$};

\node(a)at(-1.2,1){   $\mathbf 1$};

\node(a)at(0,0.5){   $\frac{\mathbf  5}{\mathbf  2}$};

\node(a)at(1,0.3){   $\mathbf  2$};

\node(a)at(2,0.3){ $ \mathbf 1$};

\node(a)at(3,0.3){   $\mathbf  2$};
\node(a)at(4,0.3){   $\mathbf 1$};

\node(a)at(0.7,-3){ $\mathbf 1$}; 
 \node(a)at(0.7,-2){ $\mathbf  2$}; 

\node(a)at(0.7,-1){ $\mathbf  3$}; 
           
  \end{tikzpicture} 
  \caption{$\cal P$-nodes and resolution of the generic polar curve}\label{fig:23ter}
\end{figure}
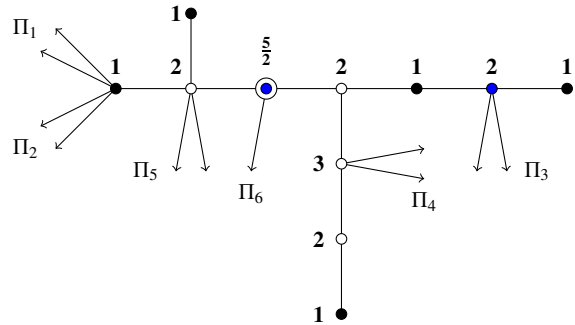

 We then obtain the inner geometric decomposition described  on Figure \ref{fig:23quatro}. There are nine inner nodes, which correspond to   five $B(1)$-pieces (in black), two special $A$-pieces  with rates $\frac{5}{2}$ and $2$ (in blue)  and two $B(2)$-pieces (in red). 
 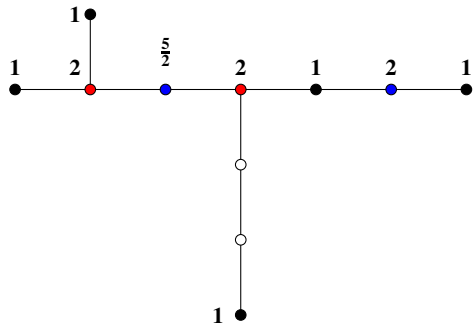
\begin{figure}[ht] 
\centering
\begin{tikzpicture}
 
   \draw[thin ](-2,0)--(4,0);
      \draw[thin ](1,0)--(1,-3);
     
      \draw[thin ](-1,1)--(-1,0);
           \draw[fill=black   ] (-1,1)circle(2pt);

               \draw[fill=black ] (-2,0)circle(2pt);
           \draw[fill=red   ] (-1,0)circle(2pt);
           \draw[fill=blue   ] (0,0)circle(2pt);
          \draw[fill=red ] (1,0)circle(2pt); 
           \draw[fill=black] (2,0)circle(2pt);
  \draw[fill=blue ] (3,0)circle(2pt);
        \draw[fill =black ] (4,0)circle(2pt);
        
          \draw[fill=white ] (1,-1)circle(2pt);
         \draw[fill=white ] (1,-2)circle(2pt);
        \draw[fill =black ] (1,-3)circle(2pt);  
 

\node(a)at(-2,0.3){   $\mathbf 1$};

\node(a)at(-1.2,0.3){   $\mathbf  2$};

\node(a)at(-1.2,1){   $\mathbf 1$};

\node(a)at(0,0.5){   $\frac{\mathbf  5}{\mathbf  2}$};

\node(a)at(1,0.3){   $\mathbf  2$};

\node(a)at(2,0.3){ $ \mathbf 1$};

\node(a)at(3,0.3){   $\mathbf  2$};
\node(a)at(4,0.3){   $\mathbf 1$};

\node(a)at(0.7,-3){ $\mathbf 1$}; 

           
  \end{tikzpicture} 
  \caption{Minimal resolution which factors through Nash transform}\label{fig:23quatro}
\end{figure}

  \end{example}
 
The terminology   {\it inner geometric decomposition} comes from the following result: 

  \begin{theorem}[\cite{BirbrairNeumannPichon2014} Complete Classification Theorem for inner Lipschitz geometry]\label{th:classification}
  The inner Lipschitz  geometry of $(X,0)$ determines and is uniquely
  determined by the following data:
  \begin{enumerate}
  \item\label{it:1.9.1} The graph decomposition of $X^{(\epsilon)}$ as the union of the links $B_i^{(\epsilon)}$ and $A_{i,j}^{(\epsilon)}$. 
   \item\label{it:1.9.2} for each $B_i^{(\epsilon)}$, the  inner rate  $q_i \geq 1$. 
   \item\label{it:1.9.2} for each $B_i^{(\epsilon)}$  such that $q_i >1$, the 
    homotopy class of the foliation by fibers of the fibration
    $z_1\colon B_i^{(\epsilon)}\to S_{\epsilon}^1$. 
  \end{enumerate}
  Moreover, these data are completely encoded  in the resolution graph $\Gamma$ whose nodes are weighted by the rates $q_i$ and by the multiplicities of a generic linear form $h$ along the  exceptional curves $E_i$ up to a  multiplicative constant. The latter is equivalent to the data of the maximal ideal $Z_{max}$ (see \cite{Nemethi1999}) up to a multiple.  
\end{theorem}

\subsection{The outer Lipschitz decomposition}

We now define on $(X,0)$ a geometric decomposition of $(X,0)$ which is a refinement of the inner geometric decomposition. 

\begin{definition} \label{dfn:outer node} We use again the notations of Definition \ref{dfn:inner nodes}. 
We call {\it  $ \cal P$-curve}   any component of $\pi^{-1}(0)$ which corresponds to a component of $\nu^{-1}(0)$   and {\it $\cal P$-node} \index{node!$\cal P$-node}  any vertex of $\Gamma$ which represents a $\cal P$-curve. 
We call {\it outer node} 
any vertex of $\Gamma$ which has at least three incident edges or which represents a curve with genus $>0$ or which is an $\cal L$- or a  $\cal P$-node. 
\end{definition}

We start again with the initial geometric decomposition of $(X,0)$  (Definition \ref{defn:initial decomposition}). Using Lemma \ref{amalgamation}, we  amalgamate iteratively  $D$ and $A$-pieces but with the rule that we never amalgamate  any  $B$-pieces corresponding to a $\cal P$-node.   

\begin{definition} \label{def:outer decomposition}
We call this decomposition the {\it outer geometric decomposition}  \index{geometric decomposition!outer} of $(X,0)$. 
\end{definition}

Let us now state an analog of Proposition \ref{prop:inner decomposition}: 

\begin{proposition} \label{prop:outer decomposition} The  pieces of the outer geometric decomposition  of $(X,0)$   can be  described as follows. For each outer node $(i)$ of $\Gamma$, let $\Gamma_i$ be the subgraph of $\Gamma$ consisting of $(i)$ and any attached bamboo.  
\begin{enumerate}
\item The $B$-pieces are  the sets $B_i = \pi({\mathscr{N}}(\Gamma_i))$, in bijection with the outer nodes of $\Gamma$, and $B_i$ is a $B(q_i)$-piece. 
\item The $A$-pieces are the sets $A_{i,j} = \pi( N(S_{i,j}))$ where $S_{i,j}$ is a string or an edge joining two outer nodes $i$ and $j$ of $\Gamma$. Moreover, $A_{i,j}$ is an $A(q_i,q_j)$-piece.
 \end{enumerate}
\end{proposition}

\begin{example}  The outer decomposition of the minimal singularity of Example \ref{example:minimal1}  is  described  on Figure \ref{fig:23cinquo}. There is exactly  one outer node which is not an inner node. So the outer decomposition is a refinement of the inner one: there is an extra $B(3)$-piece. 

 \begin{figure}[ht] 
\centering
\begin{tikzpicture}
 
   \draw[thin ](-2,0)--(4,0);
      \draw[thin ](1,0)--(1,-3);
     
      \draw[thin ](-1,1)--(-1,0);
           \draw[fill=black   ] (-1,1)circle(2pt);

               \draw[fill=black ] (-2,0)circle(2pt);
           \draw[fill=red   ] (-1,0)circle(2pt);
           \draw[fill=blue   ] (0,0)circle(2pt);
          \draw[fill=red ] (1,0)circle(2pt); 
           \draw[fill=black] (2,0)circle(2pt);
  \draw[fill=blue ] (3,0)circle(2pt);
        \draw[fill =black ] (4,0)circle(2pt);
        
          \draw[fill=red ] (1,-1)circle(2pt);
         \draw[fill=white ] (1,-2)circle(2pt);
        \draw[fill =black ] (1,-3)circle(2pt);  
 

\node(a)at(-2,0.3){   $\mathbf 1$};

\node(a)at(-1.2,0.3){   $\mathbf  2$};

\node(a)at(-1.2,1){   $\mathbf 1$};

\node(a)at(0,0.5){   $\frac{\mathbf  5}{\mathbf  2}$};

\node(a)at(1,0.3){   $\mathbf  2$};

\node(a)at(2,0.3){ $ \mathbf 1$};

\node(a)at(3,0.3){   $\mathbf  2$};
\node(a)at(4,0.3){   $\mathbf 1$};

\node(a)at(0.7,-3){ $\mathbf 1$}; 

\node(a)at(0.7,-1){ $\mathbf  3$}; 
           
  \end{tikzpicture} 
  \caption{ }\label{fig:23cinquo}
\end{figure}

\end{example}

\begin{example}    \index{surface singularity!$E_8$} The outer  geometric decomposition of the $E_8$ singularity coincides with the  initial geometric decomposition. So its graph is the one of Example \ref{example:E8 outer}. Notice that the $E_8$ example is very special.  In general the outer geometric decomposition has much less pieces than the  initial geometric decomposition. 
\end{example}

\begin{theorem} \label{thm:outer invariant}  The outer Lipschitz geometry of a normal surface singularity $(X,0)$ determines the geometric decomposition of $(X,0)$ up to self-bilipschitz homeomorphism.

Moreover, these data are completely encoded in the resolution graph $\Gamma$ where each   outer node is weighted by the  inner rate $q_{E_i}$ and the  self-intersection  $E_i^2$ of the corresponding exceptional curve $E_i$ and by the multiplicity $m_i$ of a generic linear form $h$ along $E_i$.  

Notice that the latter is equivalent to the datum of the maximal ideal  cycle $Z_{max} := \sum_i m_i E_i$  in the resolution $\pi$. (see \cite[2.I]{Nemethi1999} for   details on  $Z_{max}$). 
\end{theorem}

The  statement of  Theorem   \ref{thm:outer invariant} has some similarities with that of Theorem \ref{th:classification}, but the proof is radically different. The proof of the  Lipschitz invariance of the outer geometric decomposition is based on a bubble trick which enables one to  recover  first the  $B$-pieces of the decomposition which have highest inner rate.  Then the whole decomposition is determined by an inductive process  based again on a second bubble trick  by exploring the surface with bubbles having radius $\epsilon^q$, with  decreasing rates $q$. The proof is delicate. We refer to \cite{NeumannPichon2012} for details.

\begin{remark} As a byproduct  of the bilipschitz invariance of the maximal ideal cycle $Z_{max}$ stated in Theorem \ref{thm:outer invariant} we obtain that the multiplicity $m(X,0)$ is an invariant of the Lipschitz geometry of a complex normal surface germ. Indeed, $m(X,0)$ is nothing but the sum of the multiplicities of $Z_{max}$ at the $\cal L$-nodes of $\Gamma$. 

In \cite{FFS}, the authors   prove a broad generalization of this fact:   the outer Lipschitz geometry of a surface singularity (not necessarily normal) determines its multiplicity 

The Lipschitz invariance of the multiplicity is no  longer true in higher dimension as proved in \cite{BFSV}. Actually, the proofs of the bilipschitz invariance in \cite{NeumannPichon2012}   and  \cite{FFS}  deeply  use  the classification of $3$-dimensional manifolds up to diffeomorphisms. 
 \end{remark}
 
 Using again bubble tricks,  we can prove that beyond the weighted graph $\Gamma$ and the maximal cycle $Z_{max}$, the outer Lipschitz geometry determines  a  large amount of other classical analytic invariants. These invariants are of two  types. The first is related  to the generic hyperplane sections and the blow-up of the
 maximal ideal, and the second is related to the polar and discriminant  curves  of generic plane projections and the
Nash modification: 

\begin{theorem} \cite{NeumannPichon2012} \label{th:invariants from geometry}
  If $(X,0)$ is a   normal  complex surface singularity, then the outer
  Lipschitz geometry on $X$ determines:
  \begin{enumerate}
  \item{\bf Invariants from generic hyperplane sections:}\label{Invariants from generic hyperplane sections}
    \begin{enumerate}
        \item\label{it:decoration} the decoration of the  resolution graph $\Gamma$  by arrows
      corresponding to the strict transform of a generic hyperplane
      section (these data are equivalent to   the maximal ideal cycle $Z_{max}$);
    \item\label{it:hyperplane section} for a generic hyperplane $H$,
      the outer Lipschitz geometry of the curve $(X\cap H,0)$.
    \end{enumerate}
  \item{\bf Invariants from generic plane projections:}\label{Invariants from generic plane projections}
    \begin{enumerate}
    \item\label{it:decoration1} the decoration of  the resolution  graph  $\Gamma$ by arrows corresponding to the strict transform of the polar
      curve of a generic plane projection;
    \item\label{it:discriminant} the embedded topology of the discriminant
      curve of a generic plane projection;
    \item\label{it:curves} the outer Lipschitz geometry of the polar
      curve of a generic plane projection.
    \end{enumerate}
  \end{enumerate}
\end{theorem}

 \section{Appendix:  the resolution of the E8 surface singularity} \label{appendix}
 
  \index{surface singularity!$E_8$}
  
 In this appendix, we explain how to  compute the   minimal resolution graph  of a  singularity with equation of the form $z^2+f(x,y)=0$ by Laufer's method,  described in \cite[Chapter 2]{Laufer1971}  (page 23 to 27 for the $E_8$ singularity).   Here we will just  introduce the method and perform  it in the particular case of $E_8$.  We invite the reader to study it in \cite{Laufer1971}.
 
Laufer's method is based on the   Hirzebruch-Jung algorithm  which resolves any surface singularity. 

\subsection{Hirzebruch-Jung algorithm}  We refer to the  paper  \cite{PopescuPampu2011}  of Patrick Popescu-Pampu for more details on this part.   The Hirzebruch-Jung algorithm consists in considering a  finite morphism $\ell \colon (X,0) \to (\C^2,0)$. Then one takes a resolution $\sigma \colon Y \to \C^2$ of  the discriminant curve $\Delta$ of $\ell$, one resolves the singularities of $\Delta$ and one considers  the pull-back $\widetilde{\sigma} \colon Z \to X$  of $\sigma$ by $\ell$. We then also have a finite morphism $\widetilde{\ell} \colon Z \to Y$ such that $\sigma \circ \widetilde{\ell} = \widetilde{\sigma}  \circ \ell$. Let $n \colon Z_0 \to Z$ be the normalization of $Z$.

The singularities of $Z_0$  are quasi-ordinary singularities relative to the projection $\widetilde{\ell} \circ n \colon Z_0 \to Y$ and with discriminant the singularities of the curve $\sigma^{-1}(\Delta)$, which are ordinary double points. Resolving these remaining singularities, one gets a morphism $\alpha \colon Z \to Z_0$. The composition $\pi = \widetilde{\sigma} \circ \alpha \colon Z \to X$  is a resolution of $(X,0)$ (in general far from being minimal). 
 
\subsection{Laufer's method} It resolves the surface germ   $(X,0) \colon x^2+f(y,z)=0$  by applying  Hirzebruch-Jung algorithm  with  the projection $\ell \colon (x,y,z) \mapsto (y,z)$ and then  by giving  an easy way to compute $Z$ from a specific resolution tree $T$ of the discriminant curve $\Delta \colon f(y,z)=0$ of $\ell$. 

Let us explain it on the singularity  $E_8$.    The discriminant $\Delta$ of the projection $\ell \colon (X,0) \to (\C^2,0)$  has  equation $f(y,z) = 0$ where $f(y,z) = y^3+z^5$. We start with the minimal  resolution $\sigma \colon Y \to \C^2$ of $\Delta$.  Its exceptional divisor  consists in four curves $E_1,\ldots,E_4$ labelled in their order of occurence in the sequence of blow-ups. Let $m_i$ be the multiplicity  of the function $f$ along $E_i$. 
The integer   $m_i$ is defined  as the exponent $u^{m_{i}}$ appearing in the total transform of  $f$ by $\sigma$ in coordinates centered at a smooth point of $E_i$, where $u=0$ is the local equation of $E_i$.
 So it can be computed when performing the sequence of blow-ups resolving $f=0$.  

By \cite[Theorem 2.6]{Laufer1971}   these multiplicities   can also be computed from the self-intersections $E_j^2$ using  the fact that for each $j=1,\ldots,4$, the intersection $(  \sigma^* f  ).E_j$  in $Y$ equals $0$, where $( \sigma^* f ) =  m_1 E_1 + \ldots+ m_4 E_4  +f^*$,  with $f^*$ the strict transform of $f=0$ by $\sigma$.  

One obtains the following resolution tree $T$ on which  each vertex $(i)$ is weighted   by the self intersection  $E_i^2$ and by the  multiplicity $m_i$  (into parenthesis).  The arrow represents the strict transform of $\Delta$.

\begin{figure}[ht] 
\centering
\begin{tikzpicture}
  
  \node[ ](b)at
  (-2,1){$T$}; 
  
    \draw[thin,>-stealth,->](0.5,1)--+(-.4,1);
           \node[ ](b)at (0.2,2.2){$(1)$}; 
   \node[ ]at (0.3,0.7){$-1$};  
     \node[ ]at (-0.7,0.7){$-3$};  
    
       \node[ ]at (2.3,0.7){$-2$};  
       \node[ ]at (1.3,0.7){$-2$};   
       
        \node[ ]at (0.7,1.3){ $(15)$};  
     \node[ ]at (-0.5,1.3){$(5)$};  
    
       \node[ ]at (2.5,1.3){$(3)$};  
       \node[ ]at (1.5,1.3){$(9)$};

  \draw[ xshift=-0.5cm,yshift=1cm,thin] (0+2pt,0)--(2.95,0);
  \draw[ xshift=-0.5cm,yshift=1cm, fill ](0,0)circle(2pt);
 \draw[  xshift=-0.5cm,yshift=1cm,fill] (1,0)circle(2pt);
 \draw[ xshift=-0.5cm,yshift=1cm, fill] (2,0)circle(2pt);
 \draw[xshift=-0.5cm,yshift=1cm, fill](3,0)circle(2pt);
  

\end{tikzpicture} 
\caption{The minimal resolution tree of $y^3+z^5=0$}\label{fig:28}
\end{figure}
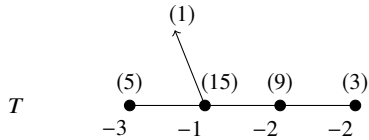

Now, we  blow up  any intersection point between two components of the total transform $(  \sigma^* f  )$ having both even multiplicities. In the case of $E_8$, all the multiplicities are even, so we blow up every double point of $( \sigma^* f )$. We obtain the resolution tree  $T'$ of Figure \ref{fig:29}.

 \begin{figure}[ht] 
\centering
\begin{tikzpicture}
  
  \node[ ](b)at
  (-2,1){$T'$}; 
   \draw[thin,](0.5,1)--+(0,1);
    \draw[thin,>-stealth,->](0.5,2)--+(-.5,1);
    
      \draw[  fill ](0.5,2)circle(2pt);

          \node[ ](b)at (-0.3,3.2){$(1)$};
           \node[ ](b)at (0.8,2.2){$(16)$}; 
   \node[ ]at (0.5,0.7){$-4$};   
    
 \node[ ]at (1.5,1.3){ $(24)$}; 
\node[ ]at (1.5,0.7){$-1$};  
        
\node[ ]at (3.5,1.3){ $(12)$};  
 \node[ ]at (3.5,0.7){$-1$};  
 
\node[ ]at (0.7,1.3){ $(15)$};  
\node[ ]at (0.5,0.7){$-4$}; 
          
\node[ ]at (-0.5,1.3){$(20)$};  
  \node[ ]at (-0.5,0.7){$-1$}; 
          
     \node[ ]at (-1.5,1.3){$(5)$};  
     \node[ ]at (-1.5,0.7){$-4$};
    
       \node[ ]at (4.5,1.3){$(3)$};  
         \node[ ]at (4.5,0.7){$-4$};  
         
       \node[ ]at (2.5,1.3){$(9)$};   
          \node[ ]at (2.5,0.7){$-4$};

  \draw[ xshift=-0.5cm,yshift=1cm,thin] (-1,0)--(5,0);
   \draw[ xshift=-0.5cm,yshift=1cm,thin] (0,0)--(5,0);
     \draw[ xshift=-0.5cm,yshift=1cm, fill ](-1,0)circle(2pt);

  \draw[ xshift=-0.5cm,yshift=1cm, fill ](0,0)circle(2pt);
 \draw[  xshift=-0.5cm,yshift=1cm,fill] (1,0)circle(2pt);
 \draw[ xshift=-0.5cm,yshift=1cm, fill] (2,0)circle(2pt);
 \draw[xshift=-0.5cm,yshift=1cm, fill](3,0)circle(2pt);
  \draw[xshift=-0.5cm,yshift=1cm, fill](4,0)circle(2pt);
   \draw[xshift=-0.5cm,yshift=1cm, fill](5,0)circle(2pt);
  

\end{tikzpicture} 
\caption{} \label{fig:29}
\end{figure}

In the particular case where there are no adjacent vertices having both odd multiplicities (this is the case in the  above tree $T'$), a resolution graph  $\Gamma$ of $(X,0) \colon x^2=f(y,z)$ is  obtained as follows:   $\Gamma$ is isomorphic to $T'$, and for any vertex $(i)$ of $T'$, the corresponding vertex of $\Gamma$ carries  self-intersection $2 E_i^2$ if the multiplicity $m_i$ is odd and $\frac{1}{2} E_i^2$ if it  is even. Moreover, the multiplicity of the function $f \circ \ell \colon (X,0) \to (\C,0)$ is $m_i$ if $m_i$ is odd and $\frac{1}{2} m_i$ if $m_i$ is even. In the case of $E_8$, we obtain the  resolution graph $\Gamma$ of Figure \ref{fig:30}, where the arrow represents the strict transform of  $ f \circ \ell   \colon (X,0) \to (\C,0)$. 

\begin{figure}[ht] 
\centering
\begin{tikzpicture}
  
   \draw[thin,](0.5,1)--+(0,1);
    \draw[thin,>-stealth,->](0.5,2)--+(-.5,1);
    
      \draw[  fill ](0.5,2)circle(2pt);

          \node[ ](b)at (-0.3,3.2){$(1)$};
           \node[ ](b)at (0.8,2.2){$(8)$}; 
   \node[ ]at (0.5,0.7){$-2$};   
    
 \node[ ]at (1.5,1.3){ $(12)$}; 
\node[ ]at (1.5,0.7){$-2$};  
        
\node[ ]at (3.5,1.3){ $(6)$};  
 \node[ ]at (3.5,0.7){$-2$};  
 
\node[ ]at (0.7,1.3){ $(15)$};  
\node[ ]at (-0.5,2){$-2$}; 
          
\node[ ]at (-0.5,1.3){$(10)$};  
  \node[ ]at (-0.5,0.7){$-2$}; 
          
     \node[ ]at (-1.5,1.3){$(5)$};  
     \node[ ]at (-1.5,0.7){$-2$};
    
       \node[ ]at (4.5,1.3){$(3)$};  
         \node[ ]at (4.5,0.7){$-2$};  
         
       \node[ ]at (2.5,1.3){$(9)$};   
          \node[ ]at (2.5,0.7){$-2$};

  \draw[ xshift=-0.5cm,yshift=1cm,thin] (-1,0)--(5,0);
   \draw[ xshift=-0.5cm,yshift=1cm,thin] (0,0)--(5,0);
     \draw[ xshift=-0.5cm,yshift=1cm, fill ](-1,0)circle(2pt);

  \draw[ xshift=-0.5cm,yshift=1cm, fill ](0,0)circle(2pt);
 \draw[  xshift=-0.5cm,yshift=1cm,fill] (1,0)circle(2pt);
 \draw[ xshift=-0.5cm,yshift=1cm, fill] (2,0)circle(2pt);
 \draw[xshift=-0.5cm,yshift=1cm, fill](3,0)circle(2pt);
  \draw[xshift=-0.5cm,yshift=1cm, fill](4,0)circle(2pt);
   \draw[xshift=-0.5cm,yshift=1cm, fill](5,0)circle(2pt);
  

\end{tikzpicture} 
\caption{The resolution graph $\Gamma$ for  $E_8$ with multiplicities of $y^3+z^5$} \label{fig:30}
\end{figure}
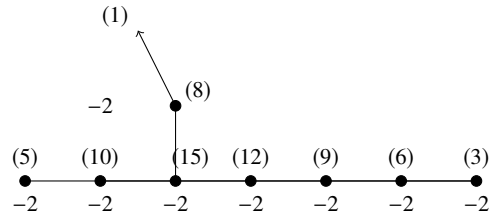

There are no $-1$-exceptional curves which could be blown down.  So forgetting $f$ and its multiplicities we get the  well known graph of the minimal resolution of $E_8$ (Figure \ref{fig:31}).

\begin{figure}[ht]
\centering
\begin{tikzpicture}
  
   \draw[thin,](0.5,1)--+(0,1);
          \draw[  fill ](0.5,2)circle(2pt);

          \node[ ](b)at (0.5,0.7){$-2$};

\node[ ]at (1.5,0.7){$-2$};  
 \node[ ]at (3.5,0.7){$-2$};  
 \node[ ]at (0.1,2){$-2$}; 

  \node[ ]at (-0.5,0.7){$-2$}; 
 
     \node[ ]at (-1.5,0.7){$-2$};
   
         \node[ ]at (4.5,0.7){$-2$};  
      
          \node[ ]at (2.5,0.7){$-2$};

  \draw[ xshift=-0.5cm,yshift=1cm,thin] (-1,0)--(5,0);
   \draw[ xshift=-0.5cm,yshift=1cm,thin] (0,0)--(5,0);
     \draw[ xshift=-0.5cm,yshift=1cm, fill ](-1,0)circle(2pt);

  \draw[ xshift=-0.5cm,yshift=1cm, fill ](0,0)circle(2pt);
 \draw[  xshift=-0.5cm,yshift=1cm,fill] (1,0)circle(2pt);
 \draw[ xshift=-0.5cm,yshift=1cm, fill] (2,0)circle(2pt);
 \draw[xshift=-0.5cm,yshift=1cm, fill](3,0)circle(2pt);
  \draw[xshift=-0.5cm,yshift=1cm, fill](4,0)circle(2pt);
   \draw[xshift=-0.5cm,yshift=1cm, fill](5,0)circle(2pt);
  

\end{tikzpicture} 
\caption{The resolution minimal graph for  $E_8$} \label{fig:31}
\end{figure}
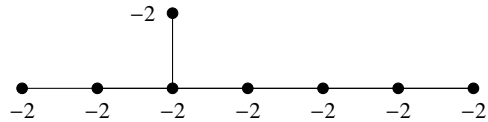

In the case where some consecutive vertices have  multiplicities which are odd, some  vertices of $T'$ may give two vertices in $\Gamma$. We refer to \cite{Laufer1971} for details.

\end{document}